\documentclass{amsart}
\pdfoutput=1
 \usepackage{hyperref}
\usepackage{enumerate} 
\usepackage{upref}
\usepackage{verbatim} 
\usepackage{color}
\usepackage{graphicx}
\newcommand{\sign}{\mathop{\rm sign}}
\newcommand*{\mailto}[1]{\href{mailto:#1}{\nolinkurl{#1}}}
\usepackage[latin1]{inputenc}
\usepackage{amssymb, amsmath, amsfonts, amsthm}
\usepackage[all]{xy}

\DeclareMathOperator{\id}{Id}
\DeclareMathOperator{\meas}{meas}
\DeclareMathOperator{\supp}{supp}
\newcommand{\dott}{\, \cdot\,}
\newcommand{\epsi}{\varepsilon}

\newcommand{\quot}{{\F/\Gr}}
\newcommand{\Gr}{G}

\newcommand{\D}{\ensuremath{\mathcal{D}}}
\newcommand{\G}{\ensuremath{\mathcal{G}}}
\newcommand{\F}{\ensuremath{\mathcal{F}}}

\newcommand{\inv}{{^{-1}}}
\newcommand{\abs}[1]{\left\vert#1\right\vert}
\newcommand{\abss}[1]{\vert#1\vert}

\newcommand{\Real}{\mathbb R}

\newcommand{\norm}[1]{\left\Vert#1\right\Vert}

\newcommand{\Linf}{{L^\infty(\Real)}}

\newcommand{\muac}{\mu_{\text{\rm ac}}}

\DeclareMathOperator{\sgn}{sgn}
\newcommand{\nn}{\nonumber}

\newcommand{\PP}{\mathcal{P}}
\DeclareMathOperator{\Ima}{Im}

\newtheorem{theorem}{Theorem}[section]

\newtheorem{lemma}[theorem]{Lemma}
\newtheorem{definition}[theorem]{Definition}
\newtheorem{proposition}[theorem]{Proposition}

\newtheorem{remark}[theorem]{Remark}
\numberwithin{equation}{section}

\allowdisplaybreaks

\begin{document}

\title[Global dissipative solutions for the 2CH equation]{Global dissipative solutions of the
  two-component Camassa--Holm system for initial data with
  nonvanishing asymptotics}

\author[K. Grunert]{Katrin Grunert}
\address{Department of Mathematical Sciences\\ Norwegian University of Science and Technology\\ NO-7491 Trondheim\\ Norway}
\email{\mailto{katring@math.ntnu.no}}
\urladdr{\url{http://www.math.ntnu.no/~katring/}}

\author[H. Holden]{Helge Holden}
\address{Department of Mathematical Sciences\\
  Norwegian University of Science and Technology\\
  NO-7491 Trondheim\\ Norway\\
  {\rm and} 
  Centre of Mathematics for Applications\\
  University of Oslo\\ NO-0316 Oslo\\ Norway}
\email{\mailto{holden@math.ntnu.no}}
\urladdr{\url{http://www.math.ntnu.no/~holden/}}

\author[X. Raynaud]{Xavier Raynaud}
\address{Centre of Mathematics for Applications\\
  University of Oslo\\ NO-0316 Oslo\\ Norway}
\email{\mailto{xavierra@cma.uio.no}}
\urladdr{\url{http://folk.uio.no/xavierra/}}

\date{\today} 
\thanks{Research supported in part by the
  Research Council of Norway project NoPiMa, and by the Austrian Science Fund (FWF) under Grant No.~J3147.}  
\subjclass[2010]{Primary: 35Q53, 35B35; Secondary: 35Q20}
\keywords{Two-component Camassa--Holm system, dissipative solutions, nonvanishing asymptotics}

\begin{abstract}
  We show existence of a global weak dissipative solution of the
  Cauchy problem for the two-component Camassa--Holm  (2CH) system on the line with
  nonvanishing and distinct spatial asymptotics.  The influence from the second component in the 2CH system on the regularity of the solution, and, in particular, the consequences for wave breaking, is discussed. Furthermore, the interplay between dissipative and conservative solutions is treated.  
\end{abstract}
\maketitle

\section{Introduction}\label{sec:intro}

We show existence of a weak global dissipative solution of  the Cauchy problem for the two-component Camassa--Holm (2CH) system with arbitrary 
$\kappa\in\Real$ and $\eta\in(0,\infty)$, given by
\begin{subequations}
\label{eq:chsys2A}
\begin{align}
\label{eq:chsys21A}
u_t-u_{txx}+\kappa u_x+3uu_x-2u_xu_{xx}-uu_{xxx}+\eta\rho\rho_x&=0,\\ \label{eq:chsys22A}
\rho_t+(u\rho)_x&=0, 
\end{align}
\end{subequations}
with initial data $u|_{t=0}=u_0$ and $\rho|_{t=0}=\rho_0$.  The initial data may have nonvanishing
limits at infinity, that is, 
\begin{equation}
  \label{eq:nonvanlimA}
  \lim_{x\to\pm\infty} u_0(x)=u_{\pm\infty}\quad\text{ and }\quad \lim_{x\to\pm\infty} \rho_0(x)=\rho_{\infty}.
\end{equation}
The 2CH system was first analyzed by Constantin and Ivanov \cite{ConstantinIvanov:2008}. Global
existence, well-posedness and blow-up conditions have been further studied in a Sobolev setting in
\cite{ELY:2007, GuYi:2010} and in Besov spaces in \cite{GuLi:2011}. The scalar CH equation (i.e.,
with $\rho$ identically equal to zero), was introduced and studied by Camassa and Holm in the
fundamental paper \cite{CH:93}, see also \cite{CHH:94}, and its analysis has been pervasive.

The CH equation possesses many intriguing properties. Here we concentrate on global solutions for the Cauchy problem on the line. The challenge is that
the CH equation experiences blow-up in finite time, even for smooth initial data, in the sense that the $H^1_{\text{\rm loc}}$ norm of the solution remains finite while $u_x$ blows up. Continuation of the solution past blow-up is intricate. It has turned out to be two distinct ways to continue the solution past blow-up, denoted conservative and dissipative solutions, respectively. Conservative solutions are associated with preservation of the $H^1$ norm, while dissipative solutions are characterized by a sudden drop in $H^1$ norm at blow-up.  This dichotomy has consequences for the well-posedness of the initial value problem as the two solutions coincide prior to blow-up.  Here we focus on the dissipative case. 

Three features are novel in this paper: First of all we include distinct and nonvanishing asymptotics of the initial data, and hence of the solution, at infinity. Since prior work has been on solutions in $H^1$, this has considerable consequences for the analysis. Secondly, we extend previous results for the CH equation to the 2CH system. It is not at all clear a priori that the highly tailored construction for the CH equation extends to the 2CH system.  Finally, we greatly simplify the analysis of two of us \cite{HolRay:09} even in the scalar case of the CH equation with vanishing asymptotics.  One advantage of the present approach is that we can use the same change of variables as in the conservative case, in contrast to the approach chosen in \cite{HolRay:09}.  We reformulate the 2CH system in terms of Lagrangian coordinates, and in this respect it relates to \cite{BreCons:05,BreCons:05a,HolRay:07,HolRay:09, GHR2} for the CH equation. Previous work on the CH equation, covering also the periodic case, includes, e.g.,  \cite{cons:97,cons:98, cons_esc, cons_esc1, cons:00, xin_zhang,xin_zhang1,CHK1, CHK2, GHR1,HolRay:07a,GHR5}. See also \cite{GY}. 

The intricate problems regarding wave breaking can best be exemplified in the context of multipeakon solutions of the  CH equation with $\kappa=0$. For reasons of brevity, and since  this example has been discussed in detail in \cite{HolRay:09}, we omit the discussion here. For additional theory on multipeakons, see
 \cite{BealsSattingerSzm:99,BealsSattingerSzm:01, wahlen, HolRay:06b,HolRay:07B}.

The continuation of the solution past wave breaking has been studied both in the conservative  \cite{BreCons:05,HolRay:07} and dissipative \cite{BreCons:05a, HolRay:09} case. In both cases the approach has been to reformulate the partial differential equation as a system of Banach space-valued ordinary differential equations, and we follow that approach here.  A different approach, based on vanishing viscosity, has been advocated in \cite{xin_zhang,xin_zhang1}.

If we for a moment assume vanishing asymptotics, the dichotomy can be further understood if one considers the associated energy, that is, the $H^1$ norm of the solution $u$ for the CH equation. In the case of a symmetric antipeakon-peakon collision, the $H^1$ norm is constant prior to wave breaking. At collision time it vanishes, and remains zero for dissipative solutions, while returning to the previous value in the conservative case. Thus we need to keep the information about the energy in the conservative case and this is handled by augmenting the solution with the energy. More precisely, we consider as  solution the pair $(u,\mu)$ where $\mu$ is a Radon measure with absolute continuous part $\muac=u_x^2dx$. This allows for energy concentration in terms of Dirac masses, while keeping the information about the energy. On the other hand, in the dissipative case, energy is not preserved, rather it is strictly decreasing at wave breaking.  The extension from scalar CH equation to the two-component  2CH system follows by augmenting the Lagrangian reformulation by an additional variable.

Let us now turn to a more detailed description of the results in this paper. First we observe that we can put $\kappa=0$ and $\eta=1$ since if $(u,\rho)$ solves \eqref{eq:chsys2A}, then $(v, \tau)$ with $v(t,x)=u(t,x-\kappa t/2)+\kappa/2$ and $\tau(t,x)=\sqrt{\eta} \rho(t,x-\kappa t/2)$ will solve \eqref{eq:chsys2A} with $\kappa=0$ and $\eta=1$. Note that this only applies since we allow for non decaying initial data at infinity.  Furthermore, we assume that $u_{-\infty}=0$. We reformulate the 2CH system as
%\begin{subequations}
\begin{align*}%\label{ch1A}
  u_t+uu_x +P_x&=0,\\ %\label{ch3A}
\rho_t+(u\rho)_x&=0,\\ %\label{ch2A}
 P-P_{xx}&=u^2+\frac12 u_x^2+\frac12\rho^2.
\end{align*}
%\end{subequations}
Next we introduce the characteristics $y(t,\xi)$, that is, the solution of  $y_t(t,\xi)=u(t,y(t,\xi))$
for a given $u$ and initial data $y(0,\xi)$. The Lagrangian velocity is given by $U(t,\xi)=u(t,y(t,\xi))$. As long as  $y(t,\xi)$ remains strictly increasing as a function of $\xi$, the solution remains smooth, and, in particular, conservative and dissipative solutions coincide.  Thus we introduce the time for wave breaking, or energy dissipation, by
\begin{equation}\label{eq:taudefA}
 \tau(\xi)=\begin{cases} 0 , & \text{ if } y_\xi(0,\xi)=0,\\
\sup \{t\in \Real_+\mid y_\xi(t^\prime,\xi)>0 \text{ for all } 0\leq t^\prime <t\}, & \text{ otherwise}.
\end{cases}
\end{equation}
We can rewrite the full 2CH system as a system of ordinary differential equations.  First define
\begin{align*}
 h&=u_x^2\circ yy_\xi+\bar\rho^2\circ y y_\xi, \\
 \bar r&=\bar\rho\circ y y_\xi, \\
 U&=\bar U+c \chi(y),\\
r&=\bar r+ky_\xi,
\end{align*}
where  $\rho=\bar\rho+k$ with $\bar\rho\in L^2(\Real)$ and  $u=\bar u+c \chi$ with $\bar u\in H^1(\Real)$. In addition, $\bar U=\bar u\circ y$. The function $\chi$ is a smooth increasing function that vanishes for large negative arguments and equals one for large positive arguments. 

Next we find that the system obeys the following system of ordinary differential equations 
 \begin{align*}
  y_t& =U , \quad U_t=-Q(X),\\
  y_{t,\xi}&=\chi_{\{\tau(\xi)>t\}}U_\xi,\\
  U_{t,\xi}&= \chi_{\{\tau(\xi)>t\}}\left(\frac12 h+(U^2+\frac12 k^2-P(X))y_\xi+k\bar r\right),\\
  h_t& = \chi_{\{\tau(\xi)>t \}}2(U^2+\frac12 k^2-P(X))U_\xi, \\
  \bar r_t&=\chi_{\{\tau(\xi)>t\}}\left(-kU_\xi\right),\\
  c_t&=0,\\
  k_t&=0,
\end{align*}
where $P(X)-U^2-\frac12 k^2$ and $Q(X)$ are given by \eqref{eq:Plag1} and \eqref{eq:Qlag1}, respectively (observe the
subtle modifications in the dissipative case compared with \eqref{eq:Plag2} and \eqref{eq:Qlag2}, respectively).
Introduce $q$ and $w$ for $y_\xi$ and $U_\xi$, respectively.  We find that $X=(\bar U, c, q, w, h, \bar r, k)$
satisfies the system $X_t=\chi_{\{\tau(\xi)>t\}}F(X)$.  The function $X$ takes values in a specific Banach space
$\bar V$, see \eqref{eq:barV}.  This system poses two challenges: First of all, due to wave breaking, the right-hand
side is discontinuous, and thus existence and uniqueness of solutions cannot follow the standard path. This is the
key difficulty compared with the conservative case. Secondly, the system possesses a number of additional constraints
in order to be consistent with the original Eulerian formulation. For example, we need to make sure that $q=y_\xi$
and $w=U_\xi$ are satisfied for positive $t$.  We will also need $y_\xi h=U_\xi^2+\bar r^2$. This is secured by only
using initial data from a carefully selected set that makes sure that all additional requirements are preserved for
the solution. The set is denoted $\G$ and is given in Definition \ref{def:G}.  A key result is the proof of global
existence of a solution, $X(t)=S_t(X_0)$ with initial data $X_0$, in the Lagrangian variables,
Theorem~\ref{th:global}.  Next we have to analyze stability of the solution in Lagrangian coordinates.  The problem
here is to identify a metric that separates conservative solutions (say $X^c(t)$) and dissipative solutions (say
$X^d(t)$) near wave breaking. The key is the behavior of the derivative of the characteristics, $y_\xi$.  At wave
breaking $y_\xi$ vanishes. For dissipative solutions $y_\xi$ remains constant for all later times, while for the
conservative solutions it becomes positive immediately after.  For conservative solutions, the metric is induced by
the Euclidean norm of the Banach space $\bar V$. For dissipative solutions, we need a metric which in addition can
separate $X^d$ and $X^c$ after collision. Indeed, if we denote by $t_c$ the collision time, the distance
between $X^d(t_c+\epsilon,\xi)$ and $X^c(t_c+\epsilon,\xi)$ with respect with this metric has to be big, for any
$\epsi>0$.  This is taken care of by the introduction of a function $g$ in Definition~\ref{def:Omega}, and the
corresponding metric $d_\Real$ in Definition~\ref{def:metric1}.  The metric yields the flow Lipschitz continuous in
the sense that
\begin{equation}
 d_\Real(X(t),X(\bar t))\leq C_T(M)\abs{\bar t-t}, \quad t,\bar t \leq T,
\end{equation}
for initial data in the set $B_M$, given by \eqref{eq:defBMfix},  see Lemma  \ref{lem:Lip}.

Having obtained the solution in Lagrangian coordinates, the next task is to transfer the solution back to Eulerian
variables $(u,\rho)$.  Here we are confronted with relabeling issue; there are several distinct solutions in
Lagrangian variables corresponding to one and the same Eulerian functions $(u,\rho)$. This is the reminiscent of the
fact that there are many distinct ways to parametrize the graph of a given function. We identify the functions that
give the same Eulerian solution, and show that the semigroup in Lagrangian coordinates respects the relabeling in the
sense that $S_t(X\circ f)=S_t(X)\circ f$ where $S_t$ denotes the semigroup of solutions and $f$ denotes the label.
We define, for any $\xi$ such that $x=y(\xi)$ (cf.~Theorem \ref{th:umudef}),
\begin{align*}
u(x)&=U(\xi),\\
\mu&=y_\#(h(\xi)\,d\xi),\\
\bar \rho(x)\,dx&=y_\#(\bar r(\xi)\,d\xi),\\
\rho(x)&=k+\bar\rho(x).
\end{align*}
We measure the distance between two Eulerian solutions by their corresponding Lagrangian distance, see \eqref{eq:meter}.

The interplay between dissipative and conservative solutions is interesting. As shown in \cite{ConstantinIvanov:2008,GHR4} a positive density $\rho_0$ regularizes the function $u$. If $\rho_0$ is positive on the whole line, no wave breaking will take place, and conservative and dissipative solutions coincide. On the other hand, if $\rho_0$ is identically zero, then $u$ will satisfy the scalar CH equation, and we will have wave breaking generically.  A local version of this result is that if $\rho_0$ is positive on an interval, then the solution will remain regular, i.e., no wave breaking will take place in the  interval bounded by the corresponding characteristics. In \cite{GHR4} we have shown that one can obtain conservative solutions of the CH equation by considering solutions of the 2CH system with positive density $\rho$. If one lets the initial density approach zero appropriately, then the solution $u$ will converge to the conservative solution of the CH equation.  However, we do show  continuity results for dissipative solutions of the 2CH system, cf.~Lemmas~\ref{lemma:kont} and \ref{lemma:kont1}.
These results are discussed in Section~\ref{sec:stabil}.
%---------- section
\section{Eulerian setting}\label{sec:euler}

We consider the Cauchy problem for the two component Camassa--Holm system with arbitrary 
$\kappa\in\Real$ and $\eta\in(0,\infty)$, given by
\begin{subequations}
\label{eq:chsys2}
\begin{align}
\label{eq:chsys21}
u_t-u_{txx}+\kappa u_x+3uu_x-2u_xu_{xx}-uu_{xxx}+\eta\rho\rho_x&=0,\\ \label{eq:chsys22}
\rho_t+(u\rho)_x&=0, 
\end{align}
\end{subequations}
with initial data $u|_{t=0}=u_0$ and $\rho|_{t=0}=\rho_0$.
We are interested in global solutions for initial data $u_0$ with nonvanishing and possibly distinct
limits at infinity, that is, 
\begin{equation}
  \label{eq:nonvanlim}
  \lim_{x\to-\infty} u_0(x)=u_{-\infty}\quad\text{ and }\quad\lim_{x\to\infty} u_0(x)=u_{\infty}.
\end{equation}
Furthermore we assume that the initial density has equal asymptotics which need not to be zero, that is,
\begin{equation}\label{eq:nonvanlimR}
 \lim_{x\to\pm\infty} \rho_0(x)=\rho_{\infty}.
\end{equation}
More precisely, we introduce the spaces
\begin{equation}
H_{\infty}(\Real)=\{v\in H^1_{\rm loc}(\Real) \mid 
v(x)=\bar v(x)+v_{-\infty}\chi(-x)+v_\infty \chi(x), \, \bar v\in H^1(\Real), v_{\pm\infty}\in\Real\}, \label{eq:Hinf} 
\end{equation}
where $\chi$ denotes a smooth partition function 
with support in $[0,\infty)$ such that $\chi(x)=1$ for $x\geq 1$ and
$\chi'(x)\geq 0$ for $x\in\Real$, and
\begin{equation}
L^2_{\rm const}(\Real)= \{g\in L^1_{\rm loc}(\Real) \mid g(x)=g_\infty+\bar g(x), \, \bar g\in L^2(\Real),  g_\infty\in \Real\}.\label{eq:Lconst} 
\end{equation}
Subsequently, we will assume that
\begin{equation}
u_0\in H_{\infty}(\Real), \quad \rho_0\in L^2_{\rm const}(\Real).
\end{equation}

Introducing the
mapping $I_{\chi}$ from $H^1(\Real)\times\Real^2$
into $H_{\text{loc}}^1(\Real)$ given by
\begin{equation*}
  I_{\chi}(\bar u,c_-,c_+)(x)=\bar u(x)+c_-\chi(-x)+c_+\chi(x)
\end{equation*}
for any $(\bar u,c_-,c_+)\in
H^1(\Real)\times\Real^2$, yields that any initial
condition $u_0\in H_\infty(\Real)$ is defined by
an element in $H^1(\Real)\times \Real^2$ through
the mapping $I_\chi$.  Hence we see that
$H_{\infty}(\Real)$ is the image of
$H^1(\Real)\times\Real^2$ by $I_\chi$, that is,
$H_{\infty}(\Real)=I_{\chi}(H^1(\Real)\times\Real^2)$. The
linear mapping $I_\chi$ is injective. We equip
$H_\infty(\Real)$ with the norm
\begin{equation}
  \label{eq:normHinfty}
  \norm{u}_{H_\infty}=\norm{\bar u}_{H^1}+\abs{c_-}+\abs{c_+}
\end{equation}
where $u=I_\chi(\bar u,c_-,c_+)$.  Then
$H_\infty(\Real)$ is a Banach space. Given another
partition function $\tilde\chi$, we define the
mapping $(\tilde{\bar u},\tilde c_-,\tilde
c_+)=\Psi(\bar u,c_-,c_+)$ from $H^1(\Real)\times\Real^2$
to $H^1(\Real)\times\Real^2$ as $\tilde c_-=c_-$, $\tilde
c_+=c_+$ and
\begin{equation}
  \label{eq:defPsi}
  \tilde{\bar u}(x)=\bar u(x)+c_-(\chi(-x)-\tilde\chi(-x))+c_+(\chi(x)-\tilde\chi(x)).
\end{equation}
The linear mapping $\Psi$ is a continuous
bijection. Since
\begin{equation*}
  I_{\chi}=I_{\tilde\chi}\circ \Psi,
\end{equation*}
we can see that the definition of the Banach space
$H_\infty(\Real)$ does not depend on the choice of
the partition function $\chi$. The norm defined by
\eqref{eq:normHinfty} for different partition
functions $\chi$ are all equivalent.

Similarly, one can associate to any element $\rho\in L^2_{\rm const}(\Real)$ the unique pair $(\bar\rho, k)\in L^2(\Real)\times \Real$ through the mapping $J$ from $L^2(\Real)\times\Real$ to $L^2_{\rm const}(\Real)$ which is defined as 
\begin{equation}
 J(\bar\rho,k)=\bar\rho+k.
\end{equation}
In fact $J$ is bijective from $L^2(\Real)\times\Real$ to $L^2_{\rm const}(\Real)$, which allows us to equip $L^2_{\rm const}(\Real)$ with the norm 
\begin{equation}\label{normL2inf}
 \norm{\rho}_{L^2_{\rm const}}=\norm{\bar\rho}_{L^2}+\vert k\vert,
\end{equation}
where we decomposed $\rho$ according to $\rho=J(\bar\rho,k)$. Thus $L^2_{\rm const}(\Real)$ together with the norm defined in \eqref{normL2inf} is a Banach space.

Note that for smooth solutions, we have the
following conservation law
\begin{equation}
  \label{eq:conslaw}
  (u^2+u_x^2+\eta\rho^2)_t+(u(u^2+u_x^2+\eta\rho^2))_x=(u^3+\kappa u^2-2Pu)_x.
\end{equation}

Moreover, if $(u(t,x),\rho(t,x))$ is a solutions of the two-component Camassa--Holm system \eqref{eq:chsys2}, then, for any constant $\alpha\in\Real$ we easily find that 
\begin{equation}
 v(t,x)=u(t,x-\alpha t)+\alpha,\quad \text{ and }\quad \tau(t,x)=\sqrt{\eta}\rho(t,x-\alpha t), 
\end{equation}
solves the two-component Camassa--Holm system with
$\kappa$ replaced by $\kappa-2\alpha$ and
$\eta=1$. Therefore, without loss of generality,
we assume in what follows, that $\lim_{x\to
  -\infty}u_0(x)=0$ and $\eta=1$. In addition, we
only consider the case $\kappa=0$ as one can make
the same conclusions for $\kappa\not =0$ with
slight modifications.

\section{Lagrangian setting}\label{sec:lag}

In this section we will introduce the set of Lagrangian coordinates we want to work with and the corresponding Banach spaces. 

\subsection{Reformulation of the 2CH system in Lagrangian coordinates}

The 2CH system with $\kappa=0$ can be rewritten as the following system\footnote{For $\kappa$ nonzero \eqref{ch2} is simply 
replaced by $P-P_{xx}=u^2+\kappa u+\frac12 u_x^2+\frac12 \rho^2$.}

\begin{subequations}
\begin{align}\label{ch1}
  u_t+uu_x& +P_x=0,\\ \label{ch3}
\rho_t&+(u\rho)_x=0,\\ \label{ch2}
 P-P_{xx}&=u^2+\frac12 u_x^2+\frac12\rho^2,
\end{align}
\end{subequations}
where $P-u^2-\frac12k^2$ and $P_x$ are given by 
\begin{align}\label{rep:Peul}
 P(t,x)-u^2(t,x)-\frac12 k^2& =-2c\chi(x)\bar u(t,x)-\bar u^2(t,x)\\ \nn
& \quad+\frac12\int_\Real e^{-\vert x-z\vert}(2c\chi\bar u+\bar u^2+\frac12 u_x^2+\frac12 \bar \rho^2+k\bar\rho)(t,z)dz\\ \nn
& \quad +\frac12\int_\Real e^{-\vert x-z\vert }2c^2(\chi^{\prime 2}+\chi\chi'')(z)dz,
\end{align}
and
\begin{align}\label{eq:Pxeul}
 P_x(t,x)& =2c^2\chi\chi'(x)\\ \nn 
& \quad-\frac12\int_\Real \sgn{(x-z)}e^{-\vert x-z\vert}(2c\chi\bar u+\bar u^2+\frac12 u_x^2+\frac12 \bar\rho^2+k\bar\rho)(t,z) dz \\ \nn 
& \quad -\frac12 \int_\Real\sgn{(x-z)}e^{-\vert x-z\vert}2c^2(\chi^{\prime 2}+\chi\chi^{\prime\prime})(z) dz.
\end{align}

A close inspection of $P_x(t,x)$, like in \cite{GHR3}, reveals that the asymptotic behavior has to be preserved.  
Thus we write here and later 
\begin{equation}\label{eq:ubar}
 u(t,x)=\bar u(t,x)+c \chi(x), \quad \bar u\in H^1(\Real),
\end{equation}
and 
\begin{equation}\label{eq:rhobar}
 \rho(t,x)=\bar\rho(t,x)+k, \quad \rho\in L^2(\Real).
\end{equation}

Until wave breaking occurs for the first time every dissipative solution coincides with the conservative one, and hence they can be described in the same way. Therefore we summarize the derivation of the Lagrangian coordinates and the corresponding system of ordinary differential equations, which describe the conservative solutions here. For details we refer to \cite{GHR4}. Afterwards, in the next subsection, we will adapt the system describing the time evolution to the dissipative case.

Define the characteristics $y(t,\xi)$ as the solution of 
\begin{equation}
 y_t(t,\xi)=u(t,y(t,\xi))
\end{equation}
for a given $y(0,\xi)$. The Lagrangian velocity is given by $U(t,\xi)=u(t,y(t,\xi))$ and we find using \eqref{ch1} that 
\begin{equation}
 U_t(t,\xi)=-Q(t,\xi),
\end{equation}
where $Q(t,\xi)=P_x(t,\xi)$ is given by 
\begin{align}\label{eq:Qlag2}
Q(t,\xi)& =2c^2\chi(y(t,\xi))\chi^\prime (y(t,\xi))\\ \nn 
& \quad-\frac12 \int_\Real \sgn{(\xi-\eta)}e^{-\vert y(t,\xi)-y(t,\eta)\vert}(2c\chi\circ y \bar Uy_\xi+\bar U^2y_\xi+\frac12 h+k\bar r)(t,\eta)d\eta\\ \nn 
& \quad -\frac12 \int_\Real \sgn{(\xi-\eta)}e^{-\vert y(t,\xi)-y(t,\eta)\vert}2c^2(\chi^{\prime 2}+\chi\chi^{\prime\prime})(y(t,\eta))y_\xi(t,\eta) d\eta 
\end{align}
where
we have introduced 
\begin{align}
 h(t,\xi)&=u_x^2(t,y(t,\xi))y_\xi(t,\xi)+\bar\rho^2(t,y(t,\xi))y_\xi(t,\xi), \\
 r(t,\xi)&=\rho(t,y(t,\xi))y_\xi(t,\xi), \\
 U(t,\xi)&=\bar U(t,\xi)+c \chi(y(t,\xi)),\label{eq:Ubar}\\
\intertext{and} \label{eq:rbar}
r(t,\xi)&=\bar r(t,\xi)+ky_\xi(t,\xi).
\end{align}
The time evolution of $h(t,\xi)$ is given by 
\begin{equation}
 h_t(t,\xi)=2(U^2(t,\xi)+\frac12 k^2-P(t,\xi))U_\xi(t,\xi), 
\end{equation}
where, slightly abusing the notation, $P(t,\xi)-U^2(t,\xi)-\frac12 k^2=P(t,y(t,\xi))-U^2(t,\xi)-\frac12 k^2$ is given by 
\begin{equation}\label{eq:Plag2}
\begin{aligned}
 P(t,\xi)-U^2(t,\xi)&-\frac12 k^2\\
 &=-2c\chi(y(t,\xi))\bar U(t,\xi)-\bar U^2(t,\xi)\\ 
&\quad +\frac12 \int_\Real e^{-\vert y(t,\xi)-y(t,\eta)\vert}(2c\chi\circ y \bar U y_\xi+\bar U^2y_\xi+\frac12 h+k\bar r)(t,\eta)d\eta \\
 &\quad +\frac12\int_\Real e^{-\vert y(t,\xi)-y(t,\eta)\vert} 2c^2 (\chi^{\prime 2}+\chi\chi^{\prime\prime})(y(t,\eta))y_\xi(t,\eta)d\eta.
\end{aligned}
\end{equation}
Last but not least, according to \eqref{ch3}, $r(t,\xi)$ is preserved  with respect to time, i.e., $r_t=0$.

\subsection{Necessary adaptations for dissipative solutions}

Wave breaking for the 2CH system means that $u_x$
becomes unbounded, which is equivalent, in the Lagrangian setting, to saying that $y_\xi$ becomes zero.  Let
therefore $\tau(\xi)$ be the first time when
$y_\xi(t,\xi)$ vanishes, i.e.,
\begin{equation}\label{eq:taudef}
 \tau(\xi)=\begin{cases} 0 , & \text{ if } y_\xi(0,\xi)=0,\\
\sup \{t\in \Real_+\mid y_\xi(t^\prime,\xi)>0 \text{ for all } 0\leq t^\prime <t\}, & \text{ otherwise},
\end{cases}
\end{equation}
where $\Real_+=[0,\infty)$.
The dissipative solutions will then be described through the solutions of the following system of ordinary differential equations
\begin{subequations}\label{eq:sysdiss}
 \begin{align}
  y_t& =U , \quad U_t=-Q(X),\\
  y_{t,\xi}&=\chi_{\{\tau(\xi)>t\}}U_\xi,\\
  U_{t,\xi}&= \chi_{\{\tau(\xi)>t\}}\left(\frac12 h+(U^2+\frac12 k^2-P(X))y_\xi+k\bar r\right),\\
  h_t& = \chi_{\{\tau(\xi)>t \}}2(U^2+\frac12 k^2-P(X))U_\xi, \\
  \bar r_t&=\chi_{\{\tau(\xi)>t\}}\left(-kU_\xi\right),\\
  c_t&=0,\\
  k_t&=0,
\end{align}
\end{subequations}
where $P(X)-U^2-\frac12 k^2$ and $Q(X)$ are given by  (observe the subtle modifications in the dissipative case compared with 
\eqref{eq:Plag2} and \eqref{eq:Qlag2}, respectively)
\begin{equation}\label{eq:Plag1}
\begin{aligned}
  &P(t,\xi) -U^2(t,\xi)-\frac12 k^2\\ 
  &\qquad\qquad =-2c\chi(y(t,\xi))\bar U(t,\xi)-\bar U^2(t,\xi)\\ 
&\qquad\qquad\quad +\frac12 \int_{\tau(\eta)>t} e^{-\vert y(t,\xi)-y(t,\eta)\vert}(2c\chi\circ y \bar U y_\xi+\bar U^2y_\xi+\frac12 h+k\bar r)(t,\eta)d\eta \\ 
 &\qquad\qquad\quad +\frac12\int_\Real e^{-\vert y(t,\xi)-y(t,\eta)\vert} 2c^2 (\chi^{\prime 2}+\chi\chi^{\prime\prime})(y(t,\eta))y_\xi(t,\eta)d\eta
\end{aligned}
\end{equation}
and 
\begin{align}\label{eq:Qlag1}
Q(t,\xi)&=2c^2\chi(y(t,\xi))\chi^\prime(y(t,\xi))\\ \nn 
& \quad-\frac12 \int_{\tau(\eta)>t} \sgn{(\xi-\eta)}e^{-\vert y(t,\xi)-y(t,\eta)\vert}(2c\chi\circ y \bar Uy_\xi+\bar U^2y_\xi+\frac12 h+k\bar r)(t,\eta)d\eta\\ \nn
& \quad -\frac12 \int_\Real \sgn{(\xi-\eta)}e^{-\vert y(t,\xi)-y(t,\eta)\vert}2c^2(\chi^{\prime 2}+\chi\chi^{\prime\prime})(y(t,\eta))y_\xi(t,\eta) d\eta,
\end{align}
respectively. (The integrals over the real line in \eqref{eq:Plag1} and \eqref{eq:Qlag1} could be replaced by $\tau(\eta)>t$ since  $y_\xi(t,\eta)=0$ 
outside this domain.)

We introduce the following notation for the Banach spaces we will often use. Let 
\begin{equation*}
E= L^2(\Real)\cap L^\infty(\Real), 
\end{equation*}
together with the norm
\begin{equation*}
 \norm{f}_{E}=\norm{f}_{L^2}+\norm{f}_{L^\infty},
\end{equation*}
and 
\begin{align}
 W&= L^2(\Real)\times L^2(\Real)\times L^2(\Real)\times L^2(\Real), \notag\\
\bar W&= E\times E\times E\times E,\notag\\
V&= L^\infty(\Real) \times L^2(\Real)\times L^\infty(\Real)\times L^2(\Real)\times L^2(\Real)\times L^2(\Real)\times L^2(\Real)\times L^\infty(\Real),\notag\\
\bar V& = L^\infty(\Real)\times E\times L^\infty(\Real)\times E\times E\times E\times E \times L^\infty(\Real). \label{eq:barV}
\end{align}

For any function $f\in C([0,T], B)$ for $T\geq 0$ and $B$ a normed space, we denote 
\begin{equation*}
 \norm{f}_{L_T^1B}=\int_0^T\norm{f(t,\dott)}_B dt\quad \text{ and }\quad \norm{f}_{L_T^\infty B}=\sup_{t\in[0,T]}\norm{f(t,\dott)}_B.
\end{equation*}

\begin{definition}\label{def:Omega}
For $x=(x_1,x_2,x_3,x_4,x_5,x_6,x_7,x_8)\in \Real^8$, we define the functions $g_1,g_2,g\colon\Real^8\to \Real$ by
\begin{align*}
 g_1(x)&=\vert x_5\vert +2\vert x_7 x_8\vert +2x_4,\\
 g_2(x)&= x_4+x_6
\end{align*}
and 
\begin{equation}\label{eq:defg}
 g(x)=\begin{cases} 
g_1(x), &\quad \text{if } x\in \Omega_1,\\
 g_2(x), & \quad \text{otherwise},\\
 \end{cases}
\end{equation}
where  $\Omega_1$ is the set where $g_1\leq g_2$, $x_5$ is negative, and $x_7+x_8x_4=0$, thus
\begin{equation*}
 \Omega_1=\{ x\in \Real^8\mid \vert x_5\vert+2\vert x_7x_8\vert +2x_4\leq x_4+x_6 \text{, } x_5\leq 0, \text{ and } x_7+x_8x_4=0\}.
\end{equation*}
Furthermore,  we will split up $\Omega_1^c$ as follows.
$\Omega_2$ is the complement of $\Omega_1$ restricted to the set where $x_7+x_8x_4=0$, that is,  
\begin{equation*}
 \Omega_2=\Omega_1^c\cap \{x\in \Real^8\mid x_7+x_8x_4=0\},
\end{equation*}
and $\Omega_3$ is the set of all points such that $x_7+x_8x_4\not =0$, that is, 
\begin{equation*}
 \Omega_3=\{ x\in \Real^8\mid x_7+x_8x_4\not =0\}.
\end{equation*}
Note the following obvious relations:
\begin{equation}
\begin{aligned}
\Omega_1\cap\Omega_2&=\emptyset, &  \Omega_1\cap\Omega_3&=\emptyset, & \Omega_2\cap\Omega_3&=\emptyset,\\
& &   \Omega_1\cup\Omega_2\cup\Omega_3&=\Real^8. &&
\end{aligned}
\end{equation}
%You can and should think of $x=(x_1,x_2,x_3,x_4,x_5,x_6,x_7,x_8)$ as $X=(y, \bar U, c, y_\xi, U_\xi, h, \bar r, k)$.
\end{definition}
See Figure \ref{fig:omega_omrl}.
As long as we are working in Lagrangian coordinates we will identify $x=(x_1,x_2,x_3,x_4,x_5,x_6,x_7,x_8)$ with $X=(y, \bar U, c, y_\xi, U_\xi, h, \bar r, k)$.

\begin{definition} \label{def:G}
 The set $\G$ consists of all $(\zeta, U, h,r)$ such that 
\begin{subequations}\label{eq:lagcoord}
\begin{align}
\label{eq:lagcoord1}
&X=(\zeta, \bar U, c, \zeta_\xi, U_\xi, h, \bar r, k)\in \bar V,\\
\label{eq:lagcoord2}
&g(y,\bar U,c, y_\xi,U_\xi,h, \bar r, k)-1\in E,\\
\label{eq:lagcoord3}
&y_\xi\geq 0, h\geq 0 \text{ almost everywhere}, \\
\label{eq:lagcoord4}
&\lim_{\xi\to -\infty} \zeta(\xi)=0,\\
\label{eq:lagcoord5}
&\frac{1}{y_\xi+h}\in L^\infty(\Real),\\ 
\label{eq:lagcoord6}
&y_\xi h=U_\xi^2+\bar r^2 \text{ almost everywhere},
\end{align}
\end{subequations}
where we denote $y(\xi)=\zeta(\xi)+\xi$.
\end{definition}
The condition \eqref{eq:lagcoord4} will be valid as long as the solutions exist since in that case we must have  $\lim_{\xi\to-\infty}U(t,\xi)=0$ by construction. In addition it should be noted that, due to the definition of $g(X)$,  \eqref{eq:lagcoord2} is valid for any $X$ that satisfies \eqref{eq:lagcoord1}.

Making the identifications $y_\xi=q$ and $w=U_\xi$, we obtain 
\begin{subequations}\label{ODEsys}
 \begin{align}
  y_t& =U , \quad U_t=-Q(X),\\
  q_{t}&=\chi_{\{\tau(\xi)>t\}}w,\\
  w_{t}&= \chi_{\{\tau(\xi)>t\}}\left(\frac12 h+(U^2+\frac12 k^2-P(X))q+k\bar r\right),\\
  h_t& = \chi_{\{\tau(\xi)>t\} }2(U^2+\frac12 k^2-P(X))w,\\
  \bar r_t&=\chi_{\{\tau(\xi)>t\}}\left(-kw\right),\\
  c_t&=0,\\
  k_t&=0,
\end{align}
\end{subequations}
where $P(X)-U^2-\frac12 k^2$ and $Q(X)$ are given by 
 \begin{align}
  P(t,\xi) -U^2&(t,\xi)-\frac 12 k^2 \nn \\
  & =-2c\chi(y(t,\xi))\bar U(t,\xi)-\bar U^2(t,\xi) \nn  \\
&\quad +\frac12 \int_{\tau(\eta)>t} e^{-\vert y(t,\xi)-y(t,\eta)\vert}(2c\chi\circ y \bar Uq+\bar U^2q+\frac12 h+k\bar r)(t,\eta)d\eta \label{eq:Plag3} \\ \nn
 &\quad +\frac12\int_\Real e^{-\vert y(t,\xi)-y(t,\eta)\vert} 2c^2 (\chi^{\prime 2}+\chi\chi^{\prime\prime})(y(t,\eta))q(t,\eta)d\eta
\end{align}
and 
\begin{align}
Q(t,\xi)& =2c^2\chi(y(t,\xi))\chi^\prime(y(t,\xi))\nn \\ 
& \quad-\frac12 \int_{\tau(\eta)>t} \sgn{(\xi-\eta)}e^{-\vert y(t,\xi)-y(t,\eta)\vert}\nn\\
&\qquad\qquad\qquad\qquad \times(2c\chi\circ y \bar Uq+\bar U^2q+\frac12 h+k\bar r)(t,\eta)d\eta\label{eq:Qlag3}\\ \nn
& \quad -\frac12 \int_\Real \sgn{(\xi-\eta)}e^{-\vert y(t,\xi)-y(t,\eta)\vert}2c^2(\chi^{\prime 2}+\chi\chi^{\prime\prime})(y(t,\eta))q(t,\eta) d\eta,
\end{align}
respectively.

The definition of $\tau$ given by
\eqref{eq:taudef} (after replacing $y_\xi$ by the
corresponding variable $q$) is not appropriate for
$q\in C([0,T],\Linf)$, and, in addition, it is not
clear from this definition whether $\tau$ is
measurable or not. That is why we replace this
definition by the following one. Let $\{t_i\}_{i=1}^\infty$ be
a dense countable subset of $[0,T]$.  Define
\begin{equation*}
  A_t=\bigcup_{n\geq1}\bigcap_{t_i\leq
    t}\Big\{\xi\in\Real\mid q(t_i,\xi)>\frac1n\Big\}. 
\end{equation*}
The sets $A_t$ are measurable for all $t$, and we
have $A_{t'}\subset A_t$ for $t\leq t'$. We
consider a dyadic partition of the interval
$[0,T]$ (that is, for each $n$, we consider the set
$\{2^{-n}iT\}_{i=0}^{2^n}$) and set
\begin{equation*}
  \tau^n(\xi)=\sum_{i=0}^{2^n}\frac{iT}{2^n}\chi_{i,n}(\xi)
\end{equation*}
where $\chi_{i,n}$ is the indicator function of the
set $A_{2^{-n}iT}\setminus A_{2^{-n}(i+1)T}$. The function $\tau^n$ is by construction
measurable. One can check that $\tau^n(\xi)$ is
increasing with respect to $n$, it is also bounded
by $T$. Hence, we can define 
\begin{equation*}
  \tau(\xi)=\lim_{n\to\infty}\tau^n(\xi),
\end{equation*}
and $\tau$ is a measurable function. The next lemma gives the main property of $\tau$.

\begin{lemma}\label{lem:2.3}
 If, for every
$\xi\in\Real$, $q(t,\xi)$ is positive and continuous with
respect to time, then 
\begin{equation}\label{eq:sectau}
 \tau(\xi)=\begin{cases} 0 , & \text{ if } y_\xi(0,\xi)=0,\\
\sup \{t\in \Real_+\mid y_\xi(t^\prime,\xi)>0 \text{ for all } 0\leq t^\prime <t\}, & \text{ otherwise},
\end{cases}
\end{equation}
%\begin{equation}
%  \label{eq:sectau}
%  \tau(\xi)=\sup\{t\in\Real^+\mid q(t',\xi)>0\text{ for all }0\leq t'<t\}, 
%\end{equation}
that is, we retrieve the definition \eqref{eq:taudef}.
\end{lemma}
\begin{proof}(From \cite{HolRay:09}.)
We denote by
$\bar\tau(\xi)$ the right-hand side of
\eqref{eq:sectau}, and we want to prove that
$\bar\tau=\tau$. We claim that
\begin{equation}
  \label{eq:tauest}
  \text{for all $t<\bar\tau(\xi)$, we have  $\xi\in A_t$, 
    and for all $t\geq\bar\tau(\xi)$, we have  $\xi\notin A_t$.} 
\end{equation}
If $t<\bar\tau(\xi)$, then
$\inf_{t'\in[0,t]}q(t',\xi)>0$ because $q$ is
continuous in time and positive. Hence, there
exists an $n$ such that
$\inf_{t'\in[0,t]}q(t',\xi)>\frac1n$, and we find that
$\xi\in\bigcap_{t_i\leq t}\left\{\xi\in\Real\mid
  q(t_i,\xi)>\frac1n\right\}$ in order that $\xi\in
A_t$. If $t\geq\bar\tau(\xi)$, then there exists a
sequence $t_{i(k)}$ of elements in the dense
family $\{t_i\}$ of $[0,T]$ such that
$t_{i(k)}\leq\bar\tau\leq t$ and
$\lim_{k\to\infty}t_{i(k)}=\bar\tau$. Since
$q(t,\xi)$ is continuous,
$\lim_{k\to\infty}q(t_{i(k)},\xi)=q(\bar\tau(\xi),\xi)=0$
and for any integer $n>0$, there exists a $k$ such
$q(t_{i(k)},\xi)\leq\frac1n$ and $t_{i(k)}\leq t$. 
Hence, for any $n>0$, $\xi\notin\bigcap_{t_i\leq
  t}\left\{\xi\in\Real\ |\
  q(t_i,\xi)>\frac1n\right\}$ and therefore
$\xi\notin A_t$. When $\bar \tau(\xi)>0$, for any
$n>0$, there exists $0\leq i\leq 2^n-1$ such that
$2^{-n}iT<\bar\tau\leq 2^{-n}(i+1)T$. 
From \eqref{eq:tauest}, we infer that $\xi\in
A_{2^{-n}iT}\setminus
A_{2^{-n}(i+1)T}$. Hence,
$\tau^n(\xi)=2^{-n} iT$, so that
\begin{equation*}
  \bar\tau(\xi)-\frac{T}{2^n}\leq\tau^n(\xi)\leq\bar\tau(\xi)+\frac{T}{2^n}. 
\end{equation*}
Letting $n$ tend to infinity, we conclude that
$\tau(\xi)=\bar\tau(\xi)$. If $\bar\tau(\xi)=0$,
then $\xi\notin A_t$ for all $t\geq0$ and
$\tau^n(\xi)=0$ for all $n$. Hence,
$\tau(\xi)=\bar\tau(\xi)=0$. 
\end{proof}

So far we have identified $q$ with $y_\xi$. However, $y_\xi$ does not decay fast enough at infinity  to belong to $L^2(\Real)$, but $y_\xi-1=\zeta_\xi$ will be in $L^2(\Real)$ and we therefore introduce $v=q-1$. In the case of conservative solutions, we know that $Q(X)$ and $P(X)-U^2-\frac12 k^2$ are Lipschitz continuous on bounded sets and that $Q(X)$ and $P(X)-U^2-\frac12 k^2$ can be bounded by a constant depending on the bounded set. A slightly different result is true when describing dissipative solutions. Let 
\begin{equation}\label{eq:defBMfix}
 B_M=\{ X\in \bar V\mid \norm{X}_{\bar V}+\norm{\frac{1}{q+h}}_{L^\infty}\leq M \text{, } qh=w^2+\bar r^2 \text{, } q\geq 0\text {, and } h\geq 0 \text{ a.e.}\}.
\end{equation}
\begin{remark}
According to our system of ordinary differential equations \eqref{ODEsys} it seems natural to impose for the solution space that if wave breaking occurs, then the functions $q$, $w$, $h$, and $\bar r$ should remain unchanged afterwards, that means $q(t,\xi)=0$, $w(t,\xi)=0$, and $\bar r(t,\xi)=0$ for all $t\geq \tau(\xi)$  and $h(t,\xi)=h(\tau(\xi),\xi)$. Moreover, also the asymptotic behavior is preserved, i.e.,  $c(t)=c(0)$ and $k(t)=k(0)$. In what follows we will always assume that these properties are fulfilled for any $X\in C([0,T], B_M)$ without stating it explicitly. 
\end{remark}

In addition it should be pointed out that for any $X\in C([0,T],B_M)$ the set of all points which enjoy wave breaking within a finite time interval $[0,T]$ is bounded, since
\begin{equation}
\meas(\{\xi\in\Real\mid q(T,\xi)=0\})\leq \int_\Real \frac{h}{q+h}(T,\xi)d\xi\leq \norm{\frac{1}{q(T)+h(T)}}_{L^\infty}\norm{h}_{L^1}\leq C(M),
\end{equation}
where $C(M)$ denotes some constant only depending on $M$.

\begin{lemma}\label{lem:PQ}
 (i) For all $X\in C([0,T],B_M)$, we have 
\begin{equation}
 \norm{Q(X)}_{L^\infty_T E}+\norm{P(X)-U^2-\frac12 k^2}_{L^\infty_TE}\leq C(M)
\end{equation}
for a constant $C(M)$ which only depends on $M$.\\
(ii) 
For any $X$ and $\tilde X$ in $C([0,T],B_M)$, we have
\begin{align}
 \norm{Q(X)-Q(\tilde X)}_{L^1_T E}&+\norm{\big(P(X)-U^2-\frac12 k^2\big)-\big(P(\tilde X)-\tilde U^2-\frac12 \tilde k^2\big)}_{L^1_TE} \nn\\
& \leq C(M) \Big(T\norm{X-\tilde X}_{L^\infty_T \hat V}   \label{eq:lt1bdqp} \\  \nn
&\quad+\int_\Real \big(\int_{\tau}^{\tilde\tau} \tilde h(t,\xi) \chi_{\{\tau<\tilde\tau\}}(\xi)dt+\int_{\tilde\tau}^\tau h(t,\xi) \chi_{\{\tilde\tau<\tau\}}(\xi) dt\big) d\xi\Big),
\end{align}
where 
\begin{align}
 \norm{X-\tilde X}_{\hat V}&=\norm{y-\tilde y}_{L^\infty}+\norm{\bar U-\bar {\tilde {U}}}_E+\vert c-\tilde c\vert+\norm{q-\tilde q}_{L^2}\\ \nn 
& \quad+\norm{w-\tilde w}_{L^2}+\norm{(h-\tilde h)\chi_{\{\tau(\xi)>t\}}\chi_{\{\tilde\tau(\xi)>t\}}}_{L^2}+\norm{\bar r-\bar{\tilde{r}}}_{L^2}+\vert k-\tilde k\vert.
\end{align}
Here $C(M)$ denotes a constant which only depends on $M$.
\end{lemma}

\begin{proof}
 We will only establish the estimates for $P(X)-U^2-\frac12 k^2$ as the ones for $Q(X)$ can be obtained using the same methods with only slight modifications. 
The main tool for proving the stated estimates will be Young's inequality which we recall here for the sake of completeness.
For any $f\in L^p(\Real)$ and $g\in L^q(\Real)$ with $1\leq p,q,r\leq \infty$, we have
\begin{equation}
 \norm{f\star g}_{L^r}\leq \norm{f}_{L^p}\norm{g}_{L^q}, \quad \text{ if } \quad 1+\frac{1}{r}=\frac{1}{p}+\frac{1}{q}.
\end{equation}

(\textit{i}): 
Let $f(\xi)=\chi_{\{\xi>0\}}e^{-\xi}$. Then we have 
\begin{align*}
 &\norm{-\frac{e^{-\zeta(t,\xi)}}{2}(f\star [\chi_{\{\tau(\xi)>t\}}e^\zeta(2c\chi\circ y\bar U q+\bar U^2q+\frac12 h+k\bar r)])(t,\xi)}_{L^\infty_TE}\\ 
& \qquad\leq\frac12 e^{\norm{\zeta}_{L^\infty_TL^\infty}}\norm{(f\star [\chi_{\{\tau(\xi)>t\}}e^\zeta(2c\chi\circ y \bar U q+ \bar U^2q+\frac12 h+k\bar r)])(t,\xi)}_{L^\infty_TE}\\ 
& \qquad\leq C(M)( \norm{f}_{L^1}+\norm{f}_{L^2})\norm{e^\zeta(2c\chi\circ y \bar Uq+\bar U^2q+\frac12 h+k\bar r)}_{L^\infty_TL^2}\\
& \qquad\leq C(M).
\end{align*}
Similarly, it follows that 
\begin{equation*}
 \norm{\frac{e^{\zeta(t,\xi)}}{2}\big([\chi_{\{\xi<0\}}e^{\xi}]\star [\chi_{\{\tau(\xi)>t\}}e^{-\zeta}(2c\chi\circ y \bar U q+\bar U^2q+\frac12 h+k\bar r)]\big)(t,\xi)}_{L^\infty_TE}\leq C(M).
\end{equation*}
Analogously one can investigate the other integral term. Indeed, since  $y(\xi)=\xi+\zeta(\xi)$, we have $\xi=y(\xi)-\zeta(\xi)$. The support of $\chi'$ is contained in $[0,1]$ and this means that the support of $\chi'\circ y$ is contained in the set $\{\xi\in\Real\mid 0\leq y(\xi)\leq 1\}$. Inserting this into $\xi=y(\xi)-\zeta(\xi)$, we get that $\supp(\chi'\circ y)\subset \{\xi\in\Real\mid -\norm{\zeta}_{L^\infty}\leq \xi\leq 1+\norm{\zeta}_{L^\infty}\}$. 
 Using that we obtain  
\begin{align}\label{est:chiprime}
 \norm{\chi'\circ y(t,\dott)}_{L^2}^2&=\int_{-\norm{\zeta(t,\dott)}_{L^\infty}}^{1+\norm{\zeta(t,\dott)}_{L^\infty}}(\chi'\circ y)^2(t,\xi)d\xi\leq\norm{\chi'}_{L^\infty}^2\int_{-\norm{\zeta(t,\dott)}_{L^\infty}}^{1+\norm{\zeta(t.\dott)}_{L^\infty}} d\xi  \nn\\
 &\leq C(M),
\end{align} 
together with the fact that a similar estimate holds for the $L^2(\Real)$-norm of $\chi\circ y\chi^{\prime\prime}\circ y$. It follows immediately that $\bar U^2$ and $c\chi\circ y \bar U$ both belong to $L^2(\Real)$ and that they can be bounded by a constant only depending on $M$. This  finishes the proof of the first part.

(\textit{ii}):
The only term which cannot be investigated like in (\textit{i}) is given by the integral term with domain of integration $\{\xi\mid \tau(\xi)>t\}$. 
Let $f(\xi)=\chi_{\{\xi>0\}}e^{-\xi}$ as before, and write $z=e^\zeta\frac12 h+e^\zeta (2c\chi\circ y\bar U q+\bar U^2 q+k\bar r)$. Then we can write 
\begin{align}\nn
f\star(\chi_{\{\tau(\xi)>t\}}z-\chi_{\{\tilde\tau(\xi)>t\}}\tilde z)&= f\star \big(\chi_{\{\tau(\xi)>t\}}\chi_{\{\tau(\xi)<\tilde\tau(\xi)\}}(z-\tilde z)\big)\\ \nn
&\quad +f\star\big((\chi_{\{\tau(\xi)>t\}}-\chi_{\{\tilde\tau(\xi)>t\}})\chi_{\{\tau(\xi)<\tilde\tau(\xi)\}}\tilde z\big)\\
&\quad +f\star \big(\chi_{\{\tilde\tau(\xi)>t\}}\chi_{\{\tau(\xi)\geq\tilde\tau(\xi)\}}(z-\tilde z)\big)\\ \nn
&\quad + f\star \big((\chi_{\{\tau(\xi)>t\}}-\chi_{\{\tilde\tau(\xi)>t\}})\chi_{\{\tau(\xi)\geq\tilde\tau(\xi)\}}z\big).
\end{align}
We estimate each of these terms separately. The first and the third term are similar, thus we only treat the first one. We obtain 
\begin{align*}
 \norm{f\star\big(\chi_{\{\tau(\xi)>t\}}\chi_{\{\tau(\xi)<\tilde\tau(\xi)\}}(z-\tilde z)\big)}_{E} & \leq (\norm{f}_{L^1}+\norm{f}_{L^2})\norm{\chi_{\{\tau(\xi)>t\}}\chi_{\{\tilde\tau(\xi)>t\}}(z-\tilde z)}_{L^2}\\ \nn 
& \leq C(M) \norm{X-\tilde X}_{\hat V}.
\end{align*}
The second term  can be treated in much the same way as the fourth one. We have 
$(\chi_{\{\tau(\xi)>t\}}-\chi_{\{\tilde\tau(\xi)>t\}})\chi_{\{\tau(\xi)<\tilde\tau(\xi)\}}=-\chi_{\{\tau(\xi)\leq t<\tilde\tau(\xi)\}}$. Introduce
$z=z_1+z_2$ with $z_1=e^\zeta\frac12 h$ and $z_2=e^\zeta (2c\chi\circ y\bar U q+\bar U^2 q+k\bar r)$. Then 
\begin{align*}
 &\norm{f\star\big((\chi_{\{\tau(\xi)>t\}}-\chi_{\{\tilde\tau(\xi)>t\}})\chi_{\{\tau(\xi)<\tilde\tau(\xi)\}}\tilde z_1\big)}_{L^1_TE}\\
& \qquad \leq \norm{f\star \big(-\chi_{\{\tau(\xi)\leq t<\tilde\tau(\xi)\}}\tilde z_1\big)}_{L^1_TE} \\ 
& \qquad \leq C(M)(\norm{f}_{L^\infty}+\norm{f}_{L^2})\norm{\chi_{\{\tau(\xi)\leq t<\tilde\tau(\xi)\}}e^{\tilde\zeta}\tilde h}_{L^1_TL^1}\\ 
& \qquad \leq C(M) \int_\Real \big( \int_{\tau(\xi)}^{\tilde\tau(\xi)}\tilde h(t,\xi)\chi_{\{ \tau(\xi)\leq t<\tilde\tau(\xi)\}}(\xi)dt\big)d\xi
\end{align*}
after applying Fubini's theorem in the last step, which is possible since the set of points which enjoy wave breaking within the time interval $[0,T]$ is bounded. Finally 
\begin{align*}
& \norm{f\star\big((\chi_{\{\tau(\xi)>t\}}-\chi_{\{\tilde\tau(\xi)>t\}})\chi_{\{\tau(\xi)<\tilde\tau(\xi)\}}\tilde z_2\big)}_{E}\\ 
& \qquad \leq \norm{f\star \big(-\chi_{\{\tau(\xi)\leq t<\tilde\tau(\xi)\}}\tilde z_2\big)}_{E} \\
& \qquad \leq (\norm{f}_{L^1}+\norm{f}_{L^2})\norm{\chi_{\{\tau(\xi)\leq t<\tilde\tau(\xi)\}}\tilde z_2}_{L^2}\\ 
& \qquad \leq C(M) \norm{\chi_{\{\tau(\xi)\leq t <\tilde\tau(\xi)\}}e^{\tilde\zeta} \big(2\tilde c\chi\circ \tilde y\bar{\tilde{U}}(q-\tilde q)+\bar{\tilde{U}}^2(q-\tilde q)+\tilde k(\bar r-\bar{\tilde{r}})\big)}_{L^2}\\ 
&\qquad \leq C(M) \norm{X-\tilde X}_{\hat V}.
\end{align*}
\end{proof}

To show the short-time existence of solutions we will use an iteration argument for the following system of ordinary differential equations. 
Denote generically $(\zeta, \bar U, c, q,w,h,\bar r, k)$ by $X$ and $(q,w,h,\bar r)$ by $Z$, thus $X=(\zeta, \bar U,c,Z,k)$. Then we define the mapping 
$$
\mathcal{P}\colon C([0,T], \bar V)\to C([0,T], \bar V)
$$ 
as follows: Given $X$ in $C([0,T], B_M)$, we can compute $P(X)-U^2-\frac12 k^2$ and $Q(X)$ using \eqref{eq:Plag3} and \eqref{eq:Qlag3}. Then $\tilde X=\mathcal{P}(X)$ is given as the modified solution with $\tilde X(0)=X(0)$ of the following system of ordinary differential equations
\begin{subequations}\label{eq:sysfix2}
 \begin{align}
 \tilde \zeta_t(t,\xi)&=\tilde U(t,\xi), \quad  \tilde U_t(t,\xi)=-Q(X)(t,\xi),  \label{eq:sysfix2A}\\
  \tilde q_t(t,\xi)&=\tilde w(t,\xi),\\
  \tilde w_t(t,\xi)&= \frac12\tilde h(t,\xi)+(U^2(t,\xi)+\frac12 k^2-P(X)(t,\xi))\tilde q(t,\xi)+k_0\bar{\tilde{r}}(t,\xi),\\
  \tilde h_t(t,\xi)& = 2(U^2(t,\xi)+\frac12 k^2-P(X)(t,\xi))\tilde w(t,\xi), \\
  \bar{\tilde{r}}_t(t,\xi)&= -k_0\tilde w(t,\xi),\\
  \tilde c_t&=0,\\
  \tilde k_t&=0.
 \end{align}
\end{subequations}
Next, we modify $\tilde X$ as follows: Determine the function $\tilde \tau(\xi)$ according to \eqref{eq:taudef} (with $X$ replaced by $\tilde X$). Subsequently, we modify the function $\tilde X$ by setting
 \begin{subequations}
  \begin{align}
  \tilde q(t,\xi)&=\tilde q(\tilde \tau(\xi),\xi), \quad &\tilde w(t,\xi)=\tilde w(\tilde \tau(\xi),\xi), &\phantom{hei}\\
   \tilde h(t,\xi)&=\tilde h (\tilde \tau(\xi),\xi), \quad &\bar{\tilde{r}}(t,\xi) = \bar{\tilde{r}}(\tilde \tau(\xi),\xi), \quad &t\ge \tilde \tau(\xi).
\end{align} 
\end{subequations}
 Observe that the first two components of $\tilde X$, given by \eqref{eq:sysfix2A}, remain unmodified. 
We write $\tilde Z_t=\chi_{\{\tilde\tau>t\}}F(X)\tilde Z$.  Thus  we set $\tilde Z(t,\xi)=\tilde Z(\tilde\tau(\xi),\xi)$ for $t>\tilde\tau(\xi)$.  We will in the following, to keep the notation reasonably simple, often write $X$ for the vector $( \zeta,U,v,w,h,r)$, or even  $(y,U,v,w,h,r)$, and correspondingly for $\tilde X$.

We will frequently consider the following spaces. 
For $X=(\zeta, \bar U, c, q, w, h,\bar r, k)$ and $Z=(q,w,h,\bar r)$, we define
\begin{align*}
\norm{Z}_{W}& = \norm{v}_{L^2}+\norm{w}_{L^2}+\norm{h}_{L^2}+\norm{\bar r}_{L^2},\\ 
\norm{Z}_{\bar W}& =\norm{v}_E+\norm{w}_E+\norm{h}_E+\norm{\bar r}_E,\\
\norm{X}_{V}& = \norm{\zeta}_{L^\infty}+\norm{\bar U}_{L^2}+\vert c\vert +\norm{v}_{L^2}+\norm{w}_{L^2}+\norm{h}_{L^2}+\norm{\bar r}_{L^2}+\vert k\vert, \\ 
\norm{X}_{\bar V}& =\norm{\zeta}_{L^\infty}+\norm{\bar U}_E+\vert c\vert +\norm{v}_E+\norm{w}_E+\norm{h}_E+\norm{\bar r}_E+\vert k\vert.
\end{align*}

The following set 
\begin{equation}\label{eq:defKgamma}
 \kappa_{1-\gamma}=\{\xi\in\Real\mid \frac{h_0}{q_0+h_0}(\xi)\geq 1-\gamma\text{, } w_0(\xi)\leq 0, \text{ and } \bar r_0(\xi)+k_0q_0(\xi)=0\}, \quad \gamma\in [0,\frac12],
\end{equation}
will play a key role in the context  of wave breaking. (For a motivation on the set $ \kappa_{1-\gamma}$, please see the paragraph  before Lemma~\ref{lem:G2}.) In particular, we have that 
\begin{equation}
 \meas(\kappa_{1-\gamma})\leq \frac{1}{1-\gamma}\int_\Real \frac{h_0}{q_0+h_0}(\xi)d\xi\leq \frac{1}{1-\gamma}\norm{\frac{1}{q_0+h_0}}_{L^\infty}\norm{h_0}_{L^1},
\end{equation}
and therefore the set $\kappa_{1-\gamma}$ has finite measure if we choose $\gamma\in [0,\frac12]$, and, in particular, $\meas(\kappa_{1-\gamma})\leq C(M)$.

\begin{lemma}\label{lem:G}
 Given $X_0\in \mathcal{G} \cap B_{M_0}$ for some constant $M_0$, given $X=(\zeta,U,v,w,h,r)\in C([0,T], B_M)$, we denote by $\tilde X=(\tilde \zeta,\tilde U,\tilde v,\tilde w,\tilde h,\tilde r)=\mathcal{P}(X)$ with initial data $X_0$. Let $\bar M=\norm{Q(X)}_{L^\infty_TL^\infty}+\norm{P(X)-U^2-\frac12 k^2}_{L^\infty_TL^\infty}+M_0$. Then the following statements hold:

(\textit{i}) For all $t$ and almost all $\xi$ 
\begin{equation}\label{eq:G1}
 \tilde q(t,\xi)\geq 0, \quad \tilde h(t,\xi)\geq 0,
\end{equation}
 and 
\begin{equation}\label{eq:G2}
 \tilde q\tilde h =\tilde w^2+\bar{\tilde{r}}^2.
\end{equation}
Thus, $\tilde q(t,\xi)=0$ implies $\tilde w(t,\xi)=0$ and $\bar{\tilde{r}}(t,\xi)=0$. We recall the notation $\tilde q=\tilde v+1$.

(\textit{ii}) We have 
\begin{equation}\label{eq:G3}
 \norm{\frac{1}{\tilde q+\tilde h}(t,\dott)}_{L^\infty}\leq 2e^{C(\bar M)T}\norm{\frac{1}{q_0+h_0}}_{L^\infty},
\end{equation}
and 
\begin{equation}
\label{eq:G3b}
 \norm{(\tilde q+\tilde h)(t,\dott)}_{L^\infty}\leq 2e^{C(\bar M)T}\norm{q_0+h_0}_{L^\infty},
\end{equation}
for all $t\in[0,T]$ and a constant $C(\bar M)$ which depends only on $\bar M$. In particular, $\tilde q+\tilde h$ remains bounded strictly away from zero.

(\textit{iii}) There exists a $\gamma\in(0,\frac12)$ depending only on $\bar M$ and $T$ such that if $\xi\in \kappa_{1-\gamma}$, then $\tilde X(t,\xi)\in \Omega_1$ for all $t\in [0,T]$, $\frac{\tilde q}{\tilde q+\tilde h}(t,\xi)$ is a decreasing function and $\frac{\tilde w}{\tilde q+\tilde h}(t,\xi)$ is an increasing function with respect to time, and therefore we have 
\begin{equation}
 \frac{w_0}{q_0+h_0}(\xi)\leq\frac{\tilde w}{\tilde q+\tilde h}(t,\xi)\leq 0\quad \text{and} \quad 0\leq \frac{\tilde q}{\tilde q+\tilde h}(t,\xi)\leq \frac{q_0}{q_0+h_0}(\xi).
\end{equation}
In addition for $\gamma$ sufficiently small, depending only on $\bar M$ and $T$, we have 
\begin{equation}\label{eq:G5}
 \kappa_{1-\gamma}\subset \{\xi\in\Real\mid 0\leq \tilde \tau(\xi)<T\}.
\end{equation}

(\textit{iv})
Moreover, for any given $\gamma\in(0,\frac12)$, there exists $\hat T>0$ such that 
\begin{equation}\label{eq:G4}
 \{\xi\in\Real\mid 0<\tilde\tau(\xi)<\hat T\}\subset \kappa_{1-\gamma}.
\end{equation}
\end{lemma}

\begin{proof}
 (\textit{i}) Since $X_0\in\G$,  equations
  \eqref{eq:G1} and \eqref{eq:G2} hold for almost every
  $\xi\in\Real$ at $t=0$. We consider such a $\xi$
  and will drop it in the notation. From
  \eqref{eq:sysfix2}, we have, on the one hand,
  \begin{equation*}
    (\tilde q\tilde h)_t=\tilde q_t\tilde h+\tilde q\tilde h_t=\tilde w\tilde h+2(U^2+\frac12 k^2-P(X))\tilde q\tilde w,
  \end{equation*}
  and, on the other hand,
  \begin{equation*}
    (\tilde w^2+\bar{\tilde{r}}^2)_t=2\tilde w\tilde w_t+2\bar{\tilde{r}}\bar{\tilde{r}}_t= \tilde w\tilde h+2(U^2+\frac12 k^2-P(X))\tilde q\tilde w.
  \end{equation*}
  Thus, $(\tilde q\tilde h-\tilde w^2-\bar{\tilde{r}}^2)_t=0$, and since $\tilde q(0)\tilde
  h(0)=\tilde w^2(0)+\bar{\tilde{r}}^2(0)$, we
  have $\tilde q(t)\tilde h(t)=\tilde w^2(t)+\bar{\tilde{r}}^2(t)$ for all $t\in[0,T]$. We have
  proved \eqref{eq:G2}. From the definition of
  $\tilde\tau$, we have that $\tilde q(t)>0$ on
  $[0,\tilde\tau(\xi))$ and by the definition of $\tilde q$,
  we have $\tilde q(t)=0$ for $t\geq\tilde\tau(\xi)$. 
  Hence, $\tilde q(t)\geq0$ for $t\geq0$. From
  \eqref{eq:G2}, it follows that, for
  $t\in[0,\tilde\tau(\xi))$, $\tilde h(t)=\frac{\tilde w^2+\bar{\tilde{r}}^2}{\tilde q}(t)$ and
  therefore $\tilde h(t)\geq0$. By continuity
  (with respect to time) of $\tilde h$, we have
  $\tilde h(\tilde\tau(\xi))\geq0$ and, since the
  variable does not change for $t\geq\tilde\tau(\xi)$,
  we have $\tilde h(t)\geq0$ for all $t\geq0$. 

  (\textit{ii}) We consider a fixed $\xi$ that we
  suppress  in the notation. We denote the Euclidean norm of $\tilde
  Z=(\tilde q,\tilde w,\tilde h,\bar{\tilde{r}})$ by $\abs{\tilde
    Z}_2=(\tilde q^2+\tilde w^2+\tilde
  h^2+\bar{\tilde{r}}^2)^{1/2}$. Since $\tilde Z_t=F(X)\tilde Z$, we have
  \begin{align*}
    \frac{d}{dt}\abss{\tilde Z}_2^{-2}&=-2\abss{\tilde Z}_2^{-4} \tilde Z\cdot\frac{d\tilde Z}{dt}
    =-2\abss{\tilde Z}_2^{-4}\tilde Z\cdot F(X)\tilde Z\\
    &\leq C(\bar M)\abss{\tilde Z}_2^{-2}, 
  \end{align*}
 for a constant $C(\bar M)$ which
  depends only $\bar M$. 
  Applying Gronwall's lemma, we obtain
  $\abs{\tilde Z(t)}_2^{-2}\leq
  e^{C(\bar M)T}\abs{\tilde Z(0)}_2^{-2}$. Hence,
  \begin{equation}
    \label{eq:mindec}
    \frac{1}{\tilde q^2+\tilde w^2+\tilde h^2+\bar{\tilde{r}}^2}(t)\leq e^{C(\bar M)T}\frac{1}{q_0^2+w_0^2+h_0^2+\bar r_0^2}.
  \end{equation}
  Using \eqref{eq:G2}, we have
  \begin{equation*}
    \tilde q^2+\tilde w^2+\tilde h^2+\bar{\tilde{r}}^2=\tilde q^2+\tilde q\tilde h+\tilde h^2.
  \end{equation*}
   Hence, \eqref{eq:mindec} yields
  \begin{equation*}
    \frac1{(\tilde q+\tilde h)^2}(t)\leq\frac{1}{\tilde q^2+\tilde q\tilde h+\tilde h^2}(t)\leq e^{C(\bar M)T}\frac{1}{q_0^2+q_0h_0+h_0^2}\leq 2e^{C(\bar M)T}\frac{1}{(q_0+h_0)^2}. 
  \end{equation*}
The second claim can be shown similarly.

(\textit{iii}) Let us consider a given $\xi\in \kappa_{1-\gamma}$. We are going to determine an upper bound on $\gamma$ depending only on $\bar M$ and $T$ such that the conclusions of (\textit{iii}) hold. For $\gamma$ small enough we have $X_0(\xi)\in\Omega_1$ as otherwise $g(X_0(\xi))=q_0(\xi)+h_0(\xi)$ and
\begin{equation*}
 1=\frac{g(X_0(\xi))}{q_0(\xi)+h_0(\xi)}<\frac{-w_0(\xi)-2k_0\bar{r}_0(\xi)+2q_0(\xi)}{q_0(\xi)+h_0(\xi)}\leq(1+2\vert k_0\vert) \sqrt{\gamma}+2\gamma
\end{equation*}
would lead to a contradiction since $\vert k_0\vert \leq M_0$.

We claim that there exists a constant $\gamma(\bar M, T)$ depending only on $\bar M$ and $T$ such that for all $\gamma\leq\gamma (\bar M, T)$, $\xi\in\Real$, and $t\in [0,T]$,
\begin{equation}
 \label{eq:gammp1}
\frac{\tilde q}{\tilde q +\tilde h}(t,\xi)\leq \gamma\text{ and }\tilde w(t,\xi)=0 \text{ implies }\tilde q(t,\xi)=0
\end{equation}
and 
\begin{equation}
 \label{eq:gammp2}
\frac{\tilde q}{\tilde q +\tilde h}(t,\xi)\leq \gamma \text{ implies } \left(\frac{\tilde w}{\tilde q+\tilde h}\right)_t(t,\xi)\geq 0. 
\end{equation}
  We consider a fixed $\xi\in\Real$ and suppress
  it in the notation.  If $\tilde w(t)=0$, then
  \eqref{eq:G2} yields $\tilde q(t)\tilde h(t)=\bar{\tilde{r}}^2(t)=k_0^2\tilde q^2(t)$, where we used that $\tilde k(t)=\tilde k(0)=k_0$ and $\bar {\tilde{r}}(t)=-\tilde k(t)\tilde q(t)=-k_0\tilde q(t)$. Thus either $\tilde q(t)=0$ or $\tilde h(t)=k_0^2\tilde q(t)$. Assume that
  $\tilde q(t)\neq 0$, then $\tilde h(t)=k_0^2\tilde q(t)$. Hence $1-\gamma\leq \frac{\tilde h(t)}{\tilde q(t)+\tilde h(t)}=k_0^2\frac{\tilde q(t)}{\tilde q(t)+\tilde h(t)}\leq C(\bar M,T)\gamma$, and
  we are led to a contradiction if we choose $\gamma$ small enough. Hence, $\tilde
  q(t)=0$, and we have proved \eqref{eq:gammp1}.\\ If
  $\frac{\tilde q}{\tilde q+\tilde h}(t)\leq\gamma$, we have
  \begin{align}
    \label{eq:lbdwt}
    \Big(\frac{\tilde w}{\tilde q+\tilde h}\Big)_t 
    &=\frac12+(U^2+\frac12 k^2-P(X)-\frac12)\frac{\tilde q}{\tilde q+\tilde h}+k_0\frac{\bar{\tilde{r}}}{\tilde q+\tilde h}\\
    &\qquad -(2U^2+k^2-2P(X)+1)\frac{\tilde w^2}{(\tilde q+\tilde h)^2} \notag \\
    & \geq \frac12 -C(\bar M,T)\frac{\tilde q}{\tilde q+\tilde h}-k_0^2\frac{\tilde q}{\tilde q+\tilde h}-C(\bar M,T) \frac{\tilde q\tilde h}{(\tilde q+\tilde h)^2} \notag \\
   & \geq \frac12-C(\bar M,T)\frac{\tilde q}{\tilde q+\tilde h}\notag \\
    &\geq \frac12-C(\bar M,T)\gamma, \nn
  \end{align}
  where we used that $\tilde k(t)=\tilde k(0)=k_0$ and $\bar {\tilde{r}}(t)=-\tilde k(t)\tilde q(t)=-k_0\tilde q(t)$.
  (Recall that we allow for a redefinition of $C(\bar M,T)$.)
  By choosing
  $\gamma(\bar M,T)\leq(4C(\bar M,T))^{-1}$, we get
  $\Big(\frac{\tilde w}{\tilde q+\tilde h}\Big)_t\geq0$, and we have proved
  \eqref{eq:gammp2}.\\ For any
  $\gamma\leq\gamma(\bar M,T)$, we consider a given $\xi$
  in $\kappa_{1-\gamma}$ and again suppress it in the
  notation.  We define
  \begin{equation*}
    t_0=\sup\{t\in[0,\tilde\tau]\mid \frac{\tilde q}{\tilde q+\tilde h}(t')<2\gamma \text{ and } \tilde w(t')<0 \text{ for all }t'\leq t\}. 
  \end{equation*}
  Let us prove that $t_0=\tilde \tau$. Assume the
  opposite, that is, $t_0<\tilde \tau$.  Then, we
  have either $\frac{\tilde q}{\tilde q+\tilde h}(t_0)=2\gamma$ or $\tilde
  w(t_0)=0$.  We have $\Big(\frac{\tilde q}{\tilde q+\tilde h}\Big)_t\leq0$
  on $[0,t_0]$ and $\frac{\tilde q}{\tilde q+\tilde h}(t)$ is decreasing on
  this interval.  Hence, $\frac{\tilde q}{\tilde q+\tilde h}(t_0)\leq \frac{\tilde
  q}{\tilde q+\tilde h}(0)\leq\gamma$, and therefore we must have
  $\tilde w(t_0)=0$. 
  Then, \eqref{eq:gammp1}
  implies $\tilde q(t_0)=0$, and therefore
  $t_0=\tilde\tau$, which contradicts our
  assumption. From \eqref{eq:lbdwt}, we get, for
  $\gamma$ sufficiently small,
  \begin{equation*} 
    0=\frac{\tilde w}{\tilde q+\tilde h}(\tilde\tau)\geq\frac{\tilde w}{\tilde q+\tilde h}(0)+C(\bar M, T)\tilde\tau,
  \end{equation*}
  and therefore  $\tilde\tau\leq \frac{\sqrt{\gamma}}{ C(\bar
  M,T)}$. By taking $\gamma$ small enough we can
  impose $\tilde\tau<T$, which proves
  \eqref{eq:G5}. It is clear from
  \eqref{eq:gammp2} that $\frac{\tilde w}{\tilde q+\tilde h}$ is increasing. 
  Assume that $\tilde X(t,\xi)$ leaves $\Omega_1$
  for some $t$.  Then, we get
  \begin{equation*}
    1 =\frac{\tilde q(t)+\tilde h(t)}{\tilde q(t)+\tilde h(t)}\leq \frac{\abs{\tilde w(t)}+2\abs{k_0\bar{\tilde{r}}(t)}+\tilde 2q(t)}{\tilde q(t)+\tilde h(t)}\leq (1+2\vert k_0\vert)\sqrt{\gamma}+2\gamma
  \end{equation*}
  and, by taking
  $\gamma$ small enough, we are led to a
  contradiction. 

(\textit{iv}) Without loss of generality we assume $\hat T\leq 1$. From (\textit{iii}) we know that there exists a $\gamma'$ only depending on $\bar M$ and $T$ such that for $\xi\in\kappa_{1-\gamma'}$, $X(t,\xi)\in\Omega_1$ and in particular we have that the function $\frac{\tilde q}{\tilde q+\tilde h}$ is decreasing and $\frac{\tilde w}{\tilde q+\tilde h}$ is an increasing function both with respect to time on $[0,T]$. Let $\bar \gamma \leq \min (\gamma,\gamma')$. We consider a fixed $\xi\in\Real$ such that $\tilde \tau(\xi)<\hat T$ (which means implicitly $\tilde r(t)=0$ for all $t$), but $\xi\not\in \kappa_{1-\bar\gamma}$. Let us introduce 
\begin{equation}
 t_0=\inf\{t\in[0,\tilde\tau)\mid \frac{\tilde h}{\tilde q+\tilde h}(\bar t)\geq 1-\bar\gamma \text{ and } \tilde w(\bar t)\leq 0\text{ for all }\bar t\in[t,\tilde\tau)\}.
\end{equation}
Since $\tilde w_t(\tilde\tau)=\frac12 \tilde h(\tilde\tau)>0$ and $\tilde w(\tilde\tau)=\tilde q(\tilde\tau)=\bar{\tilde{r}}(\tilde\tau)=0$, the definition of $t_0$ is well-posed when $\tilde\tau>0$, and we have $t_0<\tilde\tau$. By assumption $t_0>0$ and $\tilde w(t_0)=0$ or $\frac{\tilde h}{\tilde q+\tilde h}(t_0)= 1-\bar\gamma$. We cannot have $\tilde w(t_0)=0$, since it would imply, see \eqref{eq:gammp1}, that $\tilde q(t_0)=0$ and therefore $t_0=\tilde\tau$ which is not possible. Thus we must have $\frac{\tilde h}{\tilde q+\tilde h}(t_0)=1-\bar\gamma$ and in particular $\frac{\tilde q}{\tilde q+\tilde h}(t_0)=\bar\gamma$. According to the choice of $\bar\gamma$ we have that $\frac{\tilde q}{\tilde q+\tilde h}(t)\leq \bar\gamma$ for all $t\geq t_0$ and $\frac{\tilde w}{\tilde q+\tilde h}(t)$ is increasing. Then we have 
 \begin{align}\notag
        \Big(\frac{\tilde w}{\tilde q+\tilde h}\Big)_t 
    &=\frac12+(U^2+\frac12 k^2-P(X)-\frac12)\frac{\tilde q}{\tilde q+\tilde h}+k_0\frac{\bar{\tilde{r}}}{\tilde q+\tilde h}\\
    &\qquad-(2U^2+k^2-2P(X)+1)\frac{\tilde w^2}{(\tilde q+\tilde h)^2} \notag \\
   & \geq \frac12-C(\bar M,T)\frac{\tilde q}{\tilde q+\tilde h}\notag \\
    &\geq \frac12-C(\bar M,T)\bar\gamma, \nn
  \end{align}
which yields for $0\leq t_0\leq t'\leq 1$
\begin{align*}
 \frac{\tilde w}{\tilde q+\tilde h}(t')\geq \frac{\tilde w}{\tilde q+\tilde h}(t_0)+(t'-t_0)(\frac12-C(\bar M,1)\bar \gamma).
\end{align*}
Since $\frac{\tilde w}{\tilde q+\tilde h}(t_0)=-\sqrt{\bar\gamma(1-\bar\gamma)}$, we choose $\hat T$ such that $0> -\sqrt{\bar\gamma(1-\bar\gamma)}+\hat T(\frac12-C(\bar M,1)\bar \gamma)$. Thus $\frac{\tilde w}{\tilde q+\tilde h}(\hat T)\not =0$ and therefore all points which enjoy wave breaking before $\hat T$ are contained in $\kappa_{1-\bar\gamma}$, since any point entering $\kappa_{1-\bar\gamma}$ at a later time cannot reach the origin within the time interval $[0,\hat T]$ according to the last estimate. 
\end{proof}

\begin{lemma}\label{lem:estshort}
 Given $M>0$, there exists $\bar T$ and $\bar M$ such that for all $T\leq \bar T$ and any initial data $X_0\in\mathcal{G}\cap B_M$, $\mathcal{P}$ is a mapping from $C([0,T],B_{\bar M})$ to $C([0,T], B_{\bar M})$.
\end{lemma}

\begin{proof}
 To simplify the notation, we will generically denote by $K(M)$ and $C(\bar M)$ increasing functions of $M$ and $\bar M$, respectively. Without loss of generality, we assume $\bar T\leq 1$. 

Let $X\in C([0,T], B_{\bar M})$ for a value of $\bar M$ that will be determined at the end as a function of $M$. We assume without loss of generality $\bar M\geq M$.  Let $\tilde X=\mathcal{P}(X)$. From Lemma~\ref{lem:PQ}, we have 
\begin{equation}
 \norm{Q(X)}_{L^\infty_TE}\leq C(\bar M), \quad \norm{P(X)-U^2-\frac12 k^2}_{L^\infty_TE}\leq C(\bar M).
\end{equation}
Since $\tilde U_t=-Q(X)$ and  $U_0=\bar U_0+c_0\chi\circ y_0$, we get 
\begin{equation*}
 \norm{\tilde U}_{L^\infty_TL^\infty}\leq \norm{U_0}_{L^\infty}+T\norm{Q(X)}_{L^\infty_TL^\infty}\leq M +TC(\bar M).
\end{equation*}
We use that $\tilde U=\bar{\tilde{U}}+c_0\chi\circ \tilde y$ to deduce that
\begin{equation}\label{est:est1}
 \norm{\bar{\tilde{U}}}_{L^\infty_TL^\infty}\leq K(M)+TC(\bar M).
\end{equation}
Since, $\tilde \zeta_t=\tilde U$, we get
\begin{equation}\label{est:est2}
 \norm{\tilde\zeta}_{L^\infty_TL^\infty}\leq \norm{\zeta_0}_{L^\infty}+T\norm{\tilde U}_{L^\infty_TL^\infty}\leq M+TC(\bar M).
\end{equation}
Moreover, $\bar{\tilde{U}}_t=-Q(X)-c_0\chi'\circ\tilde y\tilde U$, 
we have 
\begin{equation}\label{est:est3}
\begin{aligned}
 \norm{\bar{\tilde{U}}}_{L^\infty_TL^2}&\leq \norm{\bar U_0}_{L^2}+T(\norm{Q(X)}_{L^\infty_TE}+\vert c_0\vert\norm{\tilde U}_{L^\infty_TL^\infty}\norm{\chi'\circ \tilde y}_{L^\infty_TL^2})\\
 &\leq K(M)+TC(\bar M),
\end{aligned}
\end{equation}
by \eqref{est:chiprime}.

From \eqref{eq:sysfix2}, by the Minkowsky inequality for integrals, we get 
\begin{subequations}\label{est:L1}
 \begin{align}
  \norm{\tilde v(t,\dott)}_E&\leq \norm{v_0}_E+\int_0^t \norm{\tilde w(t',\dott)}_Edt',\\ 
  \norm{\tilde w(t,\dott)}_E&\leq \norm{w_0}_E+T\norm{P(X)-U^2-\frac12 k^2}_{L^\infty_TE}\\ \nn
 & \quad +\int_0^t \Big(\frac12\norm{\tilde h(t',\dott)}_E+\norm{U^2+\frac12 k^2-P(X)}_{L^\infty_TE}\norm{\tilde v(t',\dott)}_E\nn\\
 &\qquad\qquad\qquad\qquad+\vert k_0\vert \norm{\tilde w(t',\dott)}_E\Big)dt',\nn \\
\norm{\tilde h(t,\dott)}_E& \leq \norm{h_0}_E+2\int_0^t\norm{U^2+\frac12 k^2-P(X)}_{L^\infty_TE}\norm{\tilde w(t',\dott)}_Edt',\\
\norm{\bar{\tilde{r}}(t,\dott)}_E&\leq \norm{\bar r_0}_E+\int_0^t \vert k_0\vert \norm{\tilde w(t',\dott)}_E dt'.
 \end{align}
\end{subequations}
These inequalities imply that 
\begin{equation}
 \norm{\tilde Z(t,\dott)}_{\bar W}\leq K(M)+TC(\bar M)+C(\bar M)\int_0^t \norm{\tilde Z(t',\dott)}_E dt',
\end{equation}
and, applying Gronwall's inequality yields
\begin{equation}\label{est:est4}
 \norm{\tilde Z}_{L^\infty_T\bar W}\leq (K(M)+TC(\bar M))e^{C(\bar M)T}.
\end{equation}

Gathering \eqref{est:est1}, \eqref{est:est2}, \eqref{est:est3}, and \eqref{est:est4}, we get 
\begin{equation}\label{est:est10}
 \norm{\tilde X}_{L^\infty_T\bar V}\leq (K(M)+TC(\bar M))e^{C(\bar M)T}.
\end{equation}

From \eqref{eq:G3} we get 
\begin{equation*}
  \norm{\frac{1}{\tilde q+\tilde h}}_{L^\infty_TL^\infty}\leq K(M) e^{C(\bar M)T}.
\end{equation*}
Thus we finally obtain 
\begin{equation}
 \norm{\tilde X}_{L^\infty_T \bar V}+\norm{\frac{1}{\tilde q+\tilde h}}_{L^\infty_TL^\infty}\leq (K(M)+TC(\bar M))e^{C(\bar M)T}
\end{equation}
for some constants $K(M)$ and $C(\bar M)$ that only depend on $M$ and $\bar M$, respectively. We now set $\bar M=2K(M)$.  Then we can choose $T$ so small that $(K(M)+C(\bar M)T)e^{C(\bar M)T}\leq 2K(M)=\bar M$ and therefore  $\norm{\tilde X}_{L^\infty_T \bar V}+\norm{\frac{1}{\tilde q+\tilde h}}_{L^\infty_TL^\infty}\leq \bar M$.
\end{proof}

Given $X_0\in \mathcal{G}\cap B_M$, there exists $\bar M$ which depends only on $M$ such that $\mathcal{P}$ is a mapping from $C([0,T], B_{\bar M})$ to $C([0,T], B_{\bar M})$ for $T$ small enough. Therefore we set
\begin{equation}
\Ima(\mathcal{P})=\{\mathcal{P}(X)\mid X\in C([0,T],B_{\bar M})\}.
\end{equation}

We define the \textit{discontinuity residual} as
\begin{equation*}
  \Gamma(X,\tilde X)=\int_\Real \Big(\int_{\tau}^{\tilde\tau} \tilde h(t,\xi) \chi_{\{\tau<\tilde\tau\}}(\xi)dt+\int_{\tilde\tau}^\tau h(t,\xi) \chi_{\{\tilde\tau<\tau\}}(\xi) dt\Big) d\xi.
\end{equation*}
Here it should be noted that $\Gamma(X,\tilde X)$ describes the distance between $\tau$ and $\tilde\tau$ as the following estimate shows,
\begin{equation*}
 \int_{\tilde \tau}^{\tau} h(t,\xi)dt\leq C(\bar M)(\tau-\tilde \tau)\leq C(\bar M)\left(\int_{\tilde\tau}^{\tau} h(t,\xi)dt+\int_{\tilde\tau}^{\tau} (q(t,\xi)-\tilde q(t,\xi)) dt\right).
\end{equation*}
According to Lemma~\ref{lem:PQ}, we have,
\begin{multline}\label{est:PQGamma}
  \norm{Q(X)-Q(\tilde X)}_{L^1_T
    E}+\norm{\big(P(X)-U^2-\frac12 k^2\big)-\big(P(\tilde X)-\tilde
    U^2-\frac12 \tilde k^2\big)}_{L^1_TE}\\\leq C(\bar M) \Big(T\norm{X-\tilde
    X}_{L^\infty_T \bar V}+\Gamma(X,\tilde
  X)\Big).
\end{multline}

In the next lemma we establish some estimates for
$\Gamma(X,\tilde X)$,
$\Gamma(\mathcal{P}(X),\mathcal{P}(\tilde X))$ and
a quasi-contraction property for $\mathcal{P}$.
\begin{lemma} \label{lem:contrgamma} Given $X$, $\tilde X\in \Ima(\mathcal{P})$ and $\gamma\in(0,\frac12)$ there exists $T>0$ depending on $\bar M$ such that the following inequalities hold\\
  (\textit{i})
  \begin{equation}
    \label{eq:contGamma}
    \Gamma(X,\tilde X)\leq C(\bar M)\norm{X-\tilde X}_{L^\infty_T\bar V},
  \end{equation}
  (\textit{ii}) 
  \begin{equation}
    \label{eq:contracGamma}
    \Gamma(\mathcal{P}(X),\mathcal{P}(\tilde X))\leq C(\bar M)\left(T(\norm{\mathcal{P}(X)-\mathcal{P}(\tilde X)}_{L^\infty_T\bar V}+\norm{X-\tilde X}_{L^\infty_T\bar V}) + \gamma\Gamma(X,\tilde X)\right),
  \end{equation}
  (\textit{iii}) 
  \begin{align}
    \label{eq:contracP}
    \norm{\mathcal{P}(X)-\mathcal{P}(\tilde X)}_{L^\infty_T\bar V}&\leq
    C(\bar M)\left(T\norm{X-\tilde X}_{L^\infty_T\bar V}+\Gamma(X,\tilde X)\right),
  \end{align}
where $C(\bar M)$ denotes some constant which only depends on $\bar M$.
\end{lemma}

\begin{proof}
Denote by $X_2=\mathcal{P}(X)$ and $\tilde X_2=\mathcal{P}(\tilde X)$. Given $\gamma>0$ we know from Lemma~\ref{lem:G} (\textit{iv}) that there exists $T$ small enough such that $\{\xi\in\Real\mid \tau_2(\xi)<T \text{ or } \tilde\tau_2(\xi)<T\}\subset \kappa_{1-\gamma}$ and we consider such $T$. Without loss of generality we can assume $T\leq 1$ and $\gamma\leq \gamma(\bar M,1)$.

(\textit{i}) Let us now consider $\xi\in\kappa_{1-\gamma}$ such that $\tau(\xi)\not =\tilde\tau(\xi)$. Without loss of generality we assume $\tau(\xi)<\tilde\tau(\xi)$. At time $t=0$, $X$ and $\tilde X$ coincide and therefore we cannot have $\tau(\xi)=0$ because it would imply $\tilde\tau(\xi)=0$. Hence $0<\tau(\xi)<\tilde\tau(\xi)\leq T$. Since $X(t,\xi)$ and $\tilde X(t,\xi)$ both belong to the $\Ima(\mathcal{P})$ and $\xi\in \kappa_{1-\gamma}$, we get that $X(t,\xi)$, $\tilde X(t,\xi)\in \Omega_1$ and especially $w(t,\xi)\leq 0$ and $\tilde w(t,\xi)\leq 0$. Thus we get from \eqref{eq:sysfix2}, if $\tilde X=\mathcal{P}(\hat X)$, that for $t\in [\tau(\xi),\tilde\tau(\xi)]$,
\begin{align}
 0\geq \tilde w(t,\xi)&=\tilde w(\tau(\xi),\xi)+\frac12 \int_{\tau}^{t} \tilde h(t',\xi)dt'\\ \nn 
& \quad+\int_{\tau}^{t} (\hat U^2+\frac12 \hat k^2-P(\hat X))\tilde q(t',\xi)dt'+\int_{\tau}^t k_0\bar{\tilde{r}}(t',\xi)dt'.
\end{align}
Thus, since $\norm{\hat U^2+\frac12 \hat k^2-P(\hat X)}_{L^\infty_TE}\leq C(\bar M)$ and $q(t,\xi)=w(t,\xi)=\bar r(t,\xi)=0$ for $t\geq \tau(\xi)$, we have 
\begin{align}\label{est:estesth}
  \frac12 &\int_{\tau}^{\tilde \tau} \tilde
  h(t,\xi)dt\\ \nn & \leq -\tilde
  w(\tau(\xi),\xi)+C(\bar M)\int_{\tau}^{\tilde
    \tau}(\tilde q(t,\xi)+\vert \bar{\tilde{r}}(t,\xi)\vert) dt\\ \nn & \leq
  w(\tau(\xi),\xi)-\tilde
  w(\tau(\xi),\xi)+TC(\bar M)(\norm{q-\tilde
    q}_{L^\infty_TL^\infty}+\norm{\bar r-\bar{\tilde{r}}}_{L^\infty_TL^\infty})\\ \nn &\leq
  C(\bar M)\norm{X-\tilde X}_{L^\infty_T\bar V}.
\end{align}
A similar inequality holds for $0<\tilde\tau(\xi)<\tau(\xi)\leq T$. Since $\meas(\kappa_{1-\gamma})\leq C(\bar M)$ and the only points that contribute to the integral are contained in $\kappa_{1-\gamma}$, we get 
\begin{align*}
 \Gamma(X,\tilde X)&=\int_\Real \big(\int_{\tau}^{\tilde\tau} \tilde h(t,\xi) \chi_{\{\tau<\tilde\tau\}}(\xi)dt+\int_{\tilde\tau}^\tau h(t,\xi) \chi_{\{\tilde\tau<\tau\}}(\xi) dt\big) d\xi\\ 
& \leq C(\bar M)\norm{X-\tilde X}_{L^\infty_T\bar V}.
\end{align*}
(\textit{ii}) We denote $X_2=\mathcal{P}(X)$ and
$\tilde X_2=\mathcal{P}(\tilde X)$. 
Let us now consider $\xi\in\kappa_{1-\gamma}$ such that $\tau_2(\xi)\not =\tilde\tau_2(\xi)$. Without loss of generality we assume $\tau_2(\xi)<\tilde\tau_2(\xi)$. At time $t=0$, $X_2$ and $\tilde X_2$ coincide and therefore we cannot have $\tau_2(\xi)=0$ because it would imply $\tilde\tau_2(\xi)=0$. Hence $0<\tau_2(\xi)<\tilde\tau_2(\xi)\leq T$. Since $X_2(t,\xi)$ and $\tilde X_2(t,\xi)$ both belong to the $\Ima(\mathcal{P})$ and $\xi\in \kappa_{1-\gamma}$, we get that $X_2(t,\xi)$, $\tilde X_2(t,\xi)\in \Omega_1$ and especially $w_2(t,\xi)\leq 0$ and $\tilde w_2(t,\xi)\leq 0$. Thus we get from \eqref{eq:sysfix2} that for $t\in [\tau_2(\xi),\tilde\tau_2(\xi)]$,
\begin{align}
 0\geq \tilde w_2(t,\xi)& =\tilde w_2(\tau_2(\xi),\xi)+\frac12 \int_{\tau_2}^{t} \tilde h_2(t',\xi)dt'\\ \nn
& \quad +\int_{\tau_2}^{t} (\tilde U^2+\frac12 k_0^2-P(\tilde X))\tilde q_2(t',\xi)dt' +\int_{\tau_2}^tk_0\bar{\tilde{r}}_2(t',\xi)dt',
\end{align}
since $\tilde k_2(t)=k_0=\tilde k(t)$ for all $t$.
From
  \eqref{eq:sysfix2}, we get following \eqref{est:estesth}
  \begin{align}\notag
    \int_{\tau_2}^{\tilde\tau_2}\tilde
    h_2(t,\xi)&\leq2( w_2(\tau_2,\xi)-\tilde
    w_2(\tau_2,\xi))+C(\bar M)\int_{\tau_2}^{\tilde\tau_2}(\vert q_2-\tilde
    q_2\vert(t,\xi)+\vert \bar r_2-\bar{\tilde{r}}_2\vert (t,\xi))dt\\
    \label{eq:inttautaut}
    &\leq 2\abs{w_2(\tau_2,\xi)-\tilde w_2
      (\tau_2,\xi)}+C(\bar M)T(\norm{q_2-\tilde
      q_2}_{L_T^\infty L^{\infty}}+\norm{\bar r_2-\bar{\tilde{r}}_2}_{L^\infty_TL^\infty}),
  \end{align}
  where we used that $w_2(\tau_2,\xi)=0$,
  $\tilde w_2(\tilde\tau_2,\xi)\leq 0$, and $q_2(t,\xi)=\bar r_2(t,\xi)=0$ for $t\in[\tau_2(\xi),\tilde\tau_2(\xi)]$. A
  corresponding inequality holds for the case
  $\tilde\tau_2(\xi)<\tau_2(\xi)$. We have, using again \eqref{eq:sysfix2} on the interval $[0,\tau_2(\xi)]$,
  \begin{equation}     \label{eq:bdUxitau}
  \begin{aligned}
     \vert w_2(\tau_2,\xi)-\tilde
     w_2(\tau_2,\xi)\vert &\leq \frac 12
     \int_0^{\tau_2} \vert h_2-\tilde h_2\vert
     (t,\xi)dt+\vert k_0\vert \int_0^{\tau_2}\vert \bar r_2-\bar{\tilde{r}}_2\vert (t,\xi)dt\\
          & \quad +\int_0^{\tau_2}
     \vert \big(U^2+\frac12 k_0^2-P(X)\big)-\big(\tilde U^2+\frac12 k_0^2-P(\tilde X)\big)\vert
     \vert q_2\vert (t,\xi)  dt \\
     &\quad +\int_0^{\tau_2} \vert \tilde
     U^2+\frac12 k_0^2-P(\tilde X)\vert \vert q_2-\tilde
     q_2\vert (t,\xi)dt\\ 
     & \leq \frac12
     \int_0^{\tau_2} \vert h_2-\tilde h_2\vert (t,\xi)dt \\
     &\quad+C(\bar M)\int_0^{\tau_2}(\vert q_2-\tilde q_2\vert (t,\xi)+\vert \bar r_2-\bar{\tilde{r}}_2\vert (t,\xi))dt \\
      & \quad + C(\bar M)\gamma\norm{\big(U^2+\frac12 k_0^2-P(X)\big)-\big(\tilde
       U^2+\frac12 k_0^2-P(\tilde X)\big)}_{L^1_TE} \\ &
     \leq C(\bar M)T(\norm{X_2-\tilde X_2}_{L^\infty_T\bar V}+\norm{X-\tilde X}_{L_T^\infty\bar
       V}) \\
       &\quad +C(\bar M)\gamma \Gamma(X,\tilde X), 
  \end{aligned}
  \end{equation}
where we used that $\frac{q_2}{q_2+h_2}\leq \gamma$ and $q_2=(q_2+h_2)\frac{q_2}{q_2+h_2}\leq C(\bar M)\gamma$.

  (\textit{iii}) First we estimate $\norm{Z_2-\tilde
    Z_2}_{L^\infty_T \bar
    W(\kappa_{1-\gamma}^c)}$.  For $\xi\in
  \kappa_{1-\gamma}^c$, we have
  $Z_{2,t}=F(X)Z_{2}$ and $\tilde
  Z_{2,t}=F(\tilde X)\tilde Z_2$ for all
  $t\in[0,T]$. Hence,
  \begin{multline}
    \label{eq:estdifzgc2}
    \norm{(Z_2-\tilde Z_2)(t,\dott)}_{\bar
      W(\kappa_{1-\gamma}^c)}\leq
    \int_0^t \norm{\left(F(X)-F(\tilde X)\right)Z_{2}(t',\dott)}_{\bar W(\kappa_{1-\gamma}^c)}\,dt'\\
    \ +\int_0^t\norm{F(\tilde X)(Z_2-\tilde
      Z_2)(t',\dott)}_{\bar
      W(\kappa_{1-\gamma}^c)}\,dt'.
  \end{multline}
  We have, since $k_2(t)=k(t)=k_0=\tilde k(t)=\tilde k_2(t)$ for all $t$,
  \begin{multline*}
    \left(F(X)-F(\tilde
      X)\right)Z_2=\Big(0,\big((U^2+\frac12 k_0^2-P(X))-(\tilde
    U^2+\frac12 k_0^2-P(\tilde
    X))\big)q_2,\\
    2\big((U^2+\frac12 k_0^2-P(X))-(\tilde
    U^2+\frac12 k_0^2-P(\tilde X))\big)w_2,0\Big),
  \end{multline*}
  and therefore
  \begin{equation}
    \label{eq:Fbd2}
    \norm{(F(X)-F(\tilde X))Z_2}_{L_T^1\bar W}
    \leq C(\bar M)\norm{\big(U^2+\frac12 k_0^2-P(X)\big)-\big(\tilde U^2+\frac12 k_0^2-
      P(\tilde X)\big)}_{L_T^1E}. 
  \end{equation}
  Applying Gronwall's lemma to
  \eqref{eq:estdifzgc2}, as $\norm{F( \tilde
    X)}_{L_T^\infty L^\infty}\leq C(\bar
  M)$, we get
  \begin{equation}
    \label{eq:gronapp2}
    \norm{Z_2-\tilde Z_2}_{\bar W(\kappa_{1-\gamma}^c)}\leq
    C(\bar M)\norm{(F(X)-F(\tilde X))Z_2}_{L_T^1\bar W}. 
  \end{equation}
  Hence, we get by \eqref{eq:Fbd2} that
  \begin{equation}
    \label{eq:ltl2est2}
    \norm{Z_2-\tilde Z_2}_{L_T^\infty
      \bar W(\kappa_{1-\gamma}^c)}\leq C(\bar M)\norm{\big(P(X)-U^2-\frac12 k_0^2\big)- \big(P(\tilde X)-\tilde U^2-\frac12 k_0^2\big)}_{L_T^1 E}. 
  \end{equation}
  Thus, we have by \eqref{est:PQGamma} that 
  \begin{equation}
    \label{eq:pliptau22}
    \norm{Z_2-\tilde Z_2}_{L_T^\infty \bar W(\kappa_{1-\gamma}^c)}\leq C(\bar M)\big(T\norm{X-\tilde X}_{L^\infty_T\bar V}+\Gamma(X,\tilde X)\big).
  \end{equation}
  To estimate $\norm{Z_2-\tilde Z_2}_{L^\infty_T
    \bar W(\kappa_{1-\gamma})}$, we fix $\xi\in
  \kappa_{1-\gamma}$ and assume without loss of
  generality that $0<\tau_2(\xi)<\tilde
  \tau_2(\xi)\leq T$. From Lemma \ref{lem:G}, we
  have that $\frac{\tilde q_2}{\tilde q_2+\tilde
    h_2}$ is positive decreasing and $\frac{\tilde
    w_2}{\tilde q_2+\tilde h_2}$ is negative
  decreasing so that
  \begin{equation}
    \label{eq:estforqwe02}
    \abs{\tilde q_2(t,\xi)}\leq C(\bar M)\abs{
      \tilde q_2(\tau_2,\xi)} \text{ and }\abs{
      \tilde w_2(t,\xi)}\leq C(\bar M)\abs{\tilde w_2(\tau_2,\xi)}
  \end{equation}
  for $t\in[\tau_2(\xi),T]$, and therefore
  \begin{equation}
    \label{eq:estforqwe2}
    \abs{\tilde q_2(t,\xi)-q_2(t,\xi)}\leq C(\bar M)\abs{
      \tilde q_2(\tau_2,\xi)-
      q_2(\tau_2,\xi)} 
  \end{equation}
  and
  \begin{equation}\label{eq:estforqwe2b}
    \abs{
      \tilde w_2(t,\xi)-
      w_2(t,\xi)}\leq C(\bar M)\abs{\tilde w_2(\tau_2,\xi)-w_2(\tau_2,\xi)}
  \end{equation}
  for $t\in[\tau_2(\xi),T]$ because $ q_2(t,\xi)=
  w_2(t,\xi)=0$ for $t\in[\tau_2(\xi),T]$. Since $\bar{\tilde{r}}_2(t,\xi)+k_0\tilde q_2(t,\xi)=0$, we know that $\sign(\bar{\tilde{r}}_2(t,\xi))=-\sign(k_0)$, and therefore $\bar{\tilde{r}}_{2,t}(t,\xi)=-k_0\tilde w_2(t,\xi)$ implies that  $\vert \bar{\tilde{r}}_{2}(t,\xi)\vert$ decreases on $[\tau_2(\xi),T]$. Thus 
  \begin{equation}\label{eq:estdifr2}
   \vert \bar{\tilde{r}}_2(t,\xi)\vert \leq \vert \bar{\tilde{r}}_2(\tau_2,\xi)\vert \quad\text{and}\quad \vert \bar{\tilde{r}}_2(t,\xi)-\bar r_2(t,\xi)\vert \leq \vert \bar{\tilde{r}}_2(\tau_2,\xi)-\bar r_2(\tau_2,\xi)\vert
  \end{equation}
for all $t\in [\tau_2(\xi),T]$  since $\bar r_2(t,\xi)=0$ for all $t\in [\tau_2(\xi),T]$.
  For
  $t\in[\tau_2(\xi),T]$, we have $\tilde
  h_{2,t}=2(\tilde U^2+\frac12 k_0^2-P(\tilde X))\tilde w_2$
  and $h_{2,t}=0$.  Hence,
  \begin{equation}
    \label{eq:estdifh2}
    \abs{(\tilde h_2-h_2)(t,\xi)}\leq\abs{(\tilde h_2-h_2)(\tau_2,\xi)}+C(\bar M)T\abs{(\tilde w_2-w_2)(\tau_2,\xi)},
  \end{equation}
  from \eqref{eq:estforqwe2b}. For
  $t\in[0,\tau_2(\xi)]$, we have $Z_{2,t}=F(X)Z_2$
  and $ \tilde Z_{2,t}=F(\tilde X)\tilde
  Z_2$. We proceed as in the previous step and in
  the same way as we obtained \eqref{eq:gronapp2},
  we now obtain
  \begin{equation*}
    \abss{(\tilde Z_2-Z_2)(\tau_2,\xi)}\leq
    C(\bar M)\norm{(F(X)-F(\tilde X))Z_{2}}_{L_T^1L^\infty},
  \end{equation*}
  and, after using \eqref{eq:Fbd2} together with \eqref{est:PQGamma}, we get
  \begin{equation}
    \label{eq:edifZ2}
    \abss{(Z_2-\tilde Z_2)(\tau_2,\xi)}\leq C(\bar M)\big(T\norm{X-\tilde X}_{L^\infty_T\bar V}+\Gamma(X,\tilde X)\big). 
  \end{equation}
  Combining \eqref{eq:estforqwe2},
  \eqref{eq:estforqwe2b}, \eqref{eq:estdifr2}, \eqref{eq:estdifh2} and
  \eqref{eq:edifZ2}, we get
  \begin{equation}
    \label{eq:edifZ22}
    \abss{(Z_2-\tilde Z_2)(t,\xi)}\leq C(\bar M)\big(T\norm{X-\tilde X}_{L^\infty_T\bar V}+\Gamma(X,\tilde X)\big)
  \end{equation}
  for all $t\in[0,T]$. Since $\meas(\kappa_{1-\gamma})\leq C(\bar M)$,
  \eqref{eq:edifZ22} implies
  \begin{equation}
    \label{eq:edifZfin2}
    \norm{Z_2-\tilde Z_2}_{L_T^\infty \bar W(\kappa_{1-\gamma})}\leq C(\bar M)\big(T\norm{X-\tilde X}_{L^\infty_T\bar V}+\Gamma(X,\tilde X)\big).
  \end{equation}
  Combining \eqref{eq:pliptau22} and
  \eqref{eq:edifZfin2}, we get
  \begin{equation}
    \label{eq:estdiffZ2}
    \norm{Z_2-\tilde Z_2}_{L_T^\infty
      \bar W}\leq C(\bar M)\big(T\norm{X-\tilde X}_{L^\infty_T\bar V}+\Gamma(X,\tilde X)\big).
  \end{equation}
  From \eqref{eq:sysfix2A}, we obtain
  \begin{equation} \label{eq:estdiffU2}
   \norm{U_2-\tilde U_2}_{L_T^\infty L^\infty}\leq \norm{Q(X)- Q(\tilde X)}_{L_T^1 E}
   \leq C(\bar
    M)\big(T\norm{X-\tilde X}_{L^\infty_T\bar V}+\Gamma(X,\tilde X)\big),
  \end{equation}
  and 
  \begin{equation}
    \label{eq:estdiffzet2}
    \norm{\zeta_2-\tilde \zeta_2}_{L_T^\infty
      L^\infty}\leq T\norm{U_2-\tilde U_2}_{L_T^\infty L^\infty}\leq C(\bar M)\big(T\norm{X-\tilde X}_{L^\infty_T\bar V}+\Gamma(X,\tilde X)\big).
  \end{equation}
Combining the last two inequalities yields
\begin{equation} \label{eq:estdiffU22}
    \norm{\bar U_2-\bar{\tilde {U}}_2}_{L_T^\infty
      L^\infty} \leq C(\bar
    M)\big(T\norm{X-\tilde X}_{L^\infty_T\bar
      V}+\Gamma(X,\tilde X)\big).
  \end{equation}
Finally from \eqref{eq:sysfix2A} we get 
\begin{equation}\label{eq:estdiffU32}
 \norm{\bar U_2-\bar{\tilde {U}}_2}_{L_T^\infty
      L^2} \leq C(\bar
    M)\big(T\norm{X-\tilde X}_{L^\infty_T\bar
      V}+\Gamma(X,\tilde X)\big).
  \end{equation}
 Thus adding up
  \eqref{eq:estdiffZ2}, \eqref{eq:estdiffzet2},
  \eqref{eq:estdiffU22}, and \eqref{eq:estdiffU32}
  we have that
  \begin{equation}\label{est:dc1}
    \norm{X_2-\tilde X_2}_{L^\infty_T\bar V}\leq C(\bar M)\big(T\norm{X-\tilde X}_{L^\infty_T\bar V}+\Gamma(X,\tilde X)\big).
  \end{equation}
\end{proof}

\begin{theorem}[Short time solution]
  \label{th:short0} 
  For any initial data $X_0=(y_0,U_0,h_0,r_0)\in\G$,
  there exists a time $T>0$ such that there exists
  a unique solution $X=(y,U,h,r)\in C([0,T],\bar
  V)$ of \eqref{eq:sysdiss} with $X(0)=X_0$. 
  Moreover $X(t)\in\G$ for all $t\in[0,T]$. 
\end{theorem}

\begin{proof}
  In order to prove the existence and uniqueness
  of the solution we use an iteration
  argument. Therefore we set
  $X_{n+1}=\mathcal{P}(X_n)$ and
 $X_n(0)=X_0$   for all $n\in\mathbb N$. This
  implies that $X_n$ for $n=1,2,\dots$ belongs to
  $\Ima(\mathcal{P})$.  We have
  \begin{align*}
    \norm{X_{n+1}-X_n}_{L^\infty_T\bar V}&\leq
    C(\bar M)\big(T\norm{X_n-X_{n-1}}_{L^\infty_T\bar V}+\Gamma(X_n,X_{n-1})\big)\\
    &\leq C(\bar M)\Big(T\big(\norm{X_n-X_{n-1}}_{L^\infty_T\bar V}+\norm{X_{n-1}-X_{n-2}}_{L^\infty_T\bar V}\big) \\
    &\qquad\qquad+\gamma\Gamma(X_{n-1},X_{n-2})\Big)\\
    &\leq C(\bar M)(T+\gamma)\big(\norm{X_n-X_{n-1}}_{L^\infty_T\bar V}+\norm{X_{n-1}-X_{n-2}}_{L^\infty_T\bar V}\big)
  \end{align*}
where we used Lemma~\ref{lem:contrgamma}.
  Hence, for $T$
  and $\gamma$ small enough, we have
  \begin{equation*}
    \norm{X_{n+1}-X_n}_{L^\infty_T\bar V}
    \leq \frac14 (\norm{X_n-X_{n-1}}_{L^\infty_T\bar V}+\norm{X_{n-1}-X_{n-2}}_{L^\infty_T\bar V}) \quad \text{ for } n\geq 2.
  \end{equation*}
  Summation over all $n\geq 2$ on the left-hand
  side then yields
  \begin{equation*}
    \sum_{n=2}^N   \norm{X_{n+1}-X_n}_{L^\infty_T\bar V}\leq\frac14(\sum_{n=1}^{N-1}\norm{X_{n+1}-X_n}_{L^\infty_T\bar V}+\sum_{n=0}^{N-2}\norm{X_{n+1}-X_n}_{L^\infty_T\bar V})
  \end{equation*}
  and
  \begin{equation*}
    \frac12\sum_{n=0}^N   \norm{X_{n+1}-X_n}_{L^\infty_T\bar V}\leq \norm{X_1-X_0}_{L^\infty_T\bar V}+\norm{X_2-X_1}_{L^\infty_T\bar V}
  \end{equation*}
  independently on $N$. Since
  $\norm{X_{n+1}-X_n}_{L^\infty_T\bar V}\geq 0$ for all
  $n\in\mathbb{N}$, the series
  $\sum_{n=0}^{\infty} \norm{X_{n+1}-X_n}_{L^\infty_T\bar V}$ is
  increasing, bounded from above and hence
  convergent. In particular,
  \begin{equation*}
    \norm{X_{m}-X_n}_{L^\infty_T\bar V}\leq \sum_{i=n}^{m-1} \norm{X_{i+1}-X_i}_{L^\infty_T\bar V}\leq \sum_{i=n}^{\infty} \norm{X_{i+1}-X_i}_{L^\infty_T\bar V}, 
  \end{equation*}
  and therefore $\{X_n\}_{n=1}^\infty$ is a Cauchy
  sequence and tends to a unique limit $X(t)$. The
  continuity of $\mathcal{P}$ in $L^\infty_T\bar V$
  follows from \eqref{eq:contGamma} and
  \eqref{eq:contracP}. Hence, we obtain that
  $X(t)$ is not only unique but also a fix point
  of the mapping $\mathcal{P}$ to the initial
  condition $X_0$.

It is left to prove that
  $U_\xi=w$ and $y_\xi=q$. Recall that $Q(X)$ is defined via \eqref{eq:Qlag3}
  and $Q(X)$ is differentiable if and only if $y$ is differentiable. A
  formal computation gives us that
  \begin{align}
    \label{eq:derQxi}
   Q_\xi(X) &=\chi_{\{\tau(\xi)>t\}}(-\frac12 h-(U^2+\frac12 k^2-P(X))q-k\bar r)\\ \nn
   &\quad+\left(2c^2(\chi^{\prime 2}+\chi\chi^{\prime\prime})(y)+2c\chi\circ y\bar U+\bar U^2-U^2-\frac12 k^2+P(X)\right)(y_\xi-q),
  \end{align}
  and $Q_\xi\in L_{\text{loc}}^1([0,1]\times\Real)$ if $\zeta_\xi=y_\xi-1\in L^2(\Real)\cap L^\infty(\Real)$. 
  In addition, as
  $q(t,\xi)=\chi_{\{\tau(\xi)>t\}}(\xi)q(t,\xi)$ and $w(t,\xi)=\chi_{\{\tau(\xi)>t\}}(\xi)w(t,\xi)$,
  we have 
  \begin{subequations}
    \label{eq:qQxi}
  \begin{align}
   (q-y_\xi)_t&=(w-U_\xi),\\
   (w-U_\xi)_t&= \big(2c^2(\chi^{\prime 2}+\chi\chi^{\prime\prime})(y)\\ \nn
   &\quad +2c\chi\circ y \bar U+\bar U^2-U^2-\frac12 k^2+P(X)\big)(y_\xi-q).
  \end{align}
\end{subequations}
This means in particular if $q_0=y_{0,\xi}$ and $w_0=U_{0,\xi}$, that 
\begin{multline*}
 \norm{(q-y_\xi)(t,\dott)}_{E}+\norm{(w-U_\xi)(t,\dott)}_{E}\\
 \leq C(M) \int_0^t (\norm{(q-y_\xi)(t',\dott)}_{E}+\norm{(w-U_\xi)(t',\dott)}_{E})dt'
\end{multline*}
and thus using Gronwall's inequality yields that $y_\xi=q$ and $U_\xi=w$.
  
 Let us
  prove that $X(t)\in\G$ for all $t$. From
  \eqref{eq:G1} and \eqref{eq:G2}, we get
  $q(t,\xi)\geq0$, $h(t,\xi)\geq0$ and
  $qh=w^2+\bar r^2$ for all $t$ and almost all $\xi$
  and therefore, since $U_\xi=w$ and $y_\xi=q$,
  the conditions \eqref{eq:lagcoord3} and
  \eqref{eq:lagcoord6} are fulfilled. Since
  $\zeta(t,\xi)=\zeta(0,\xi)+\int_0^t
  U(t,\xi)\,dt$, we obtain by the Lebesgue dominated
  convergence theorem that 
  $\lim_{\xi\to-\infty}\zeta(t,\xi)=0$ because
  $\bar U(t,\xi)=U(t,\xi)$ for $\xi\leq -\norm{\zeta}_{L^\infty_TL^\infty}$ and $\bar U\in H^1(\Real)$. Hence, since in addition
  $X(t)\in B_{\bar M}$, $X(t)$ fulfills all the
  conditions listed in \eqref{eq:lagcoord} and $X(t)\in\G$. 
\end{proof}

\begin{remark}\label{rem:dertime}
The set $\G\cap B_M$ is closed with
respect to the topology of $\bar V$. We have
\begin{align*}
 y_{\xi,t}&=U_{\xi}, \\
  h_t&=2(U^2+\frac12 k_0^2-P(X))U_\xi, \\
\intertext{and}\\ 
 \bar r_t&=-k_0U_\xi,
\end{align*}
for all $\xi\in\Real$ and $t\in\Real_+$, 
since $U_\xi(t,\xi)=0$ for
$t\geq\tau(\xi)$. This means in particular, that $y_\xi$,  $\bar r$, and $h$ are differentiable almost everywhere with respect to time (in the classical sense).
\end{remark}

 We have that
$(\zeta,U,\zeta_\xi,U_\xi,h,r)$ is a fixed point of
$\PP$, and the results of Lemma~\ref{lem:G}
hold for $X=\tilde X=(\zeta,U,\zeta_\xi,U_\xi,h,r)$. 
Since this lemma is going to be used extensively
we rewrite it for the fixed point solution $X$. For
this purpose, we redefine $B_M$ and
$\kappa_{1-\gamma}$, see \eqref{eq:defBMfix} and \eqref{eq:defKgamma}, as
\begin{equation*}
 B_M=\{ X\in \bar V\mid \norm{X}_{\bar V}+\norm{\frac{1}{y_\xi+h}}_{L^\infty}\leq M\},
\end{equation*}
with $X=(\zeta,U,\zeta_\xi,U_\xi,h,r)$,
\begin{equation}
  \label{eq:defKgamma2}
  \kappa_{1-\gamma}=\{\xi\in\Real\mid \frac{h_0}{y_{0,\xi}+h_0}(\xi)\geq 1-\gamma\text{, } U_{0,\xi}(\xi)\leq 0, \text{ and } r_0(\xi)=0\}, \quad \gamma\in[0,\frac12].
\end{equation}
Note that every condition imposed on points $\xi\in \kappa_{1-\gamma}$ is motivated by what is known about wave breaking. If wave breaking occurs at some time $t_b$ energy is concentrated on sets of measure zero in Eulerian coordinates, which correspond to the sets where $\frac{h}{y_\xi+h}(t_b,\xi)=1$ in Lagrangian coordinates. Furthermore, it is well- known that wave breaking in the context of the 2CH system means that the spatial derivative becomes unbounded from below and hence $U_\xi(t,\xi)\leq 0$ for $t_b-\delta\leq t\leq t_b$ for such  points, see \cite{ConstantinIvanov:2008, GY}. Finally, it has been shown in \cite[Theorem 6.1]{GHR4} that wave breaking within finite time can only occur at points $\xi$ where $r_0(\xi)=0$.

Recall that $g(X)$ denotes
$g(y,\bar U,c,y_\xi,U_\xi,h,\bar r,k)$. Lemma \ref{lem:G}
rewrites as follows.

\begin{lemma}\label{lem:G2}
Let  $M_0$ be a constant, and consider initial data $X_0\in \mathcal{G}\cap B_{M_0}$. Denote  the solution of \eqref{eq:sysdiss} with initial data $X_0$ by $X=(\zeta, U, \zeta_\xi, U_\xi, h,r)\in C([0,T], B_M)$. Introduce $\bar M=\norm{Q(X)}_{L^\infty_TL^\infty}+\norm{P(X)+\frac12k^2-U^2}_{L^\infty_TL^\infty}+M_0$. 
Then the following statements hold:

(\textit{i}) We have 
\begin{equation}\label{eq:G3sol}
 \norm{\frac{1}{y_\xi+h}(t,\dott)}_{L^\infty}\leq 2e^{C(\bar M)T}\norm{\frac{1}{y_{0,\xi}+h_0}}_{L^\infty},
\end{equation}
and 
\begin{equation}
\label{eq:G3bsol}
 \norm{(y_\xi+h)(t,\dott)}_{L^\infty}\leq 2e^{C(\bar M)T}\norm{y_{0,\xi}+h_0}_{L^\infty}
\end{equation}
for all $t\in[0,T]$ and a constant $C(\bar M)$ which depends on $\bar M$. 

(\textit{iii}) There exists a $\gamma\in(0,\frac12)$ depending only on $\bar M$ and $T$ such that if $\xi\in \kappa_{1-\gamma}$, then $X(t,\xi)\in \Omega_1$ for all $t\in [0,T]$, $\frac{y_\xi}{y_\xi+h}(t,\xi)$ is a decreasing function and $\frac{U_\xi}{y_\xi+h}(t,\xi)$ is an increasing function, both with respect to time, and therefore we have 
\begin{equation}
 \frac{U_{0,\xi}}{y_{0,\xi}+h_0}(\xi)\leq\frac{U_\xi}{y_\xi+h}(t,\xi)\leq 0\quad \text{and} \quad 0\leq \frac{y_\xi}{y_\xi+h}(t,\xi)\leq \frac{y_{0,\xi}}{y_{0,\xi}+h_0}(\xi).
\end{equation}
In addition, for $\gamma$ sufficiently small, depending only on $\bar M$ and $T$, we have 
\begin{equation}\label{eq:G5sol}
 \kappa_{1-\gamma}\subset \{\xi\in\Real\mid 0\leq \tau(\xi)<T\}.
\end{equation}

(\textit{iv})
Moreover, for any given $\gamma\in(0,\frac12)$, there exists $\hat T>0$ such that 
\begin{equation}\label{eq:G4sol}
 \{\xi\in\Real\mid 0<\tau(\xi)<\hat T\}\subset \kappa_{1-\gamma}.
\end{equation}
\end{lemma}

To prove global existence of the solution we will
use the estimate contained in the following lemma.

% --------- lemma
\begin{lemma} 
  \label{lem:globest}
  Given $M_0>0$ and $T_0>0$, there exists a
  constant $M$ which only depends on $M_0$ and
  $T_0$ such that, for any $X_0=(y_0,U_0,h_0,r_0)\in B_{M_0}$, we have $X(t)\in B_{M}$ for all $t\in[0,T]$, 
  where $X(t)$ denotes the short time solution on
  $[0,T]$ with $T\leq T_0$ given by Theorem
  \ref{th:short0} for initial data $X_0$. 
\end{lemma}
% -------end lemma 
\begin{proof} 
To simplify the notation we will generically denote by $C$ constants which only depend on $\chi$, $c_0$, and $k_0$, and by $C(M_0,T_0)$ constants which in addition depend on $M_0$ and $T_0$.

Let us introduce
 \begin{equation*}
    \Sigma=\int_{\Real} \bar U^2y_\xi\,d\xi+\norm{h}_{L^1}.
  \end{equation*}
  Since $h\geq
  0$ and $h=U_\xi^2+\bar r^2-h\zeta_\xi$, we have $\norm{h}_{L^1}=\int_\Real
  h\,d\xi<\infty$. 
  We can estimate the $\norm{\bar U}_{L^\infty}^2$
  as follows: 
  \begin{subequations}
    \label{eq:derivULinf}
    \begin{align*}
      \bar U^2(\xi)&=2\int_{-\infty}^{\xi}\bar U\bar U_\xi\,d\eta\\
      &=2\int_{-\infty}^{\xi}\bar
      UU_\xi\,d\eta-2\int_{-\infty}^{\xi}c\bar
      U\chi'\circ y
      y_\xi\,d\eta\\
      &\leq \int_{\{\eta\mid  y_\xi(\eta)>0\}}\Big(\bar
      U^2y_\xi+\frac{U_\xi^2}{y_\xi}\Big)\,d\eta+2\int_{\Real}\vert c\bar
      U\vert \chi'\circ y
      y_\xi\,d\eta\\
      &\leq \int_{\{\eta\mid  y_\xi(\eta)>0\}}(\bar{U}^2y_\xi+h)\,d\eta+2C\norm{\bar U}_{L^\infty}\\
      &\leq \Sigma+2C\norm{\bar U}_{L^\infty},
    \end{align*}
  \end{subequations}
where we used that $y_\xi(\eta)=0$ implies $U_\xi(\eta)=0$ and therefore in the first integral in the second line the integrand is zero whenever $y_\xi(\xi)=0$. Thus it suffices to integrate over $\{\eta\in\Real\mid y_\xi(\eta)>0\} \cap \{\eta\leq\xi\}$ which justifies the subsequent estimate.
  Finally, inserting that $\norm{\bar U}_{L^\infty}\leq
  \frac{1}{4C}\norm{\bar U}_{L^\infty}^2+C$, we
  get
  \begin{equation}\label{est:Uinf}
    \norm{\bar U}_{L^\infty}^2\leq 2\Sigma+C.
  \end{equation}
From \eqref{eq:Plag2}, we get 
\begin{equation}\label{est:Pinf}
 \norm{P(X)}_{L^\infty}\leq C(1+\Sigma+\norm{U}_{L^\infty}^2)\leq C(1+\Sigma). 
\end{equation}
Similarly one obtains 
\begin{equation}\label{est:Qinf}
 \norm{Q(X)}_{L^\infty}\leq C(1+\Sigma).
\end{equation}
We
can now compute the derivative of $\Sigma$. From
\eqref{eq:sysdiss} we get
\begin{align*}
  \frac{d\Sigma}{dt}&=\int_{\Real}2\bar U\bar
  U_ty_\xi\,d\xi+\int_{\Real}\bar U^2y_{\xi
    t}\,d\xi+\int_\Real h_t\,d\xi\\
  &=\int_\Real2\bar U(-Q(X)-cU\chi'\circ y)y_\xi\,d\xi+\int_\Real \bar
  U^2U_\xi\,d\xi+\int_{\Real}2(U^2+\frac12 k^2-P(X))U_\xi\,d\xi\\
  &= A_1+A_2+A_3.
\end{align*}
Note that we can put the time derivative under the integral, since $y_\xi$ and $h$ are differentiable almost everywhere with respect to time (cf. Remark~\ref{rem:dertime}). 
We estimate each of these  integrals separately.
Thus
\begin{align*}
 A_1& = -2\int_\Real Q(X)\bar Uy_\xi d\xi-2\int_\Real \left(c\bar U^2\chi'\circ yy_\xi+c^2\bar U\chi\circ y\chi'\circ yy_\xi\right) d\xi\\ 
& \leq -2\int_\Real P(X)_\xi\bar Ud\xi+C\Sigma+C\norm{\bar U}_{L^\infty}\\
& \leq 2\int_\Real P(X)\bar U_\xi d\xi+C\Sigma+C\norm{\bar U}_{L^\infty},
\end{align*}
after integration by parts in the last step, since $P(X)_\xi=Q(X)y_\xi$.
Thus
\begin{align*}
  A_2&=\int_\Real \bar U^2\bar
  U_\xi\,d\xi+\int_\Real c\bar U^2(\chi'\circ
  y)y_\xi d\xi\\
  &=\int_\Real c\bar U^2(\chi'\circ y)y_\xi d\xi\leq
  C\Sigma.
\end{align*}
Furthermore
\begin{align*}
 A_3& = 2\int_\Real U^2 U_\xi d\xi+\int_\Real k^2U_\xi d\xi-2\int_\Real P(X)\bar U_\xi d\xi-2c\int_\Real P(X)\chi'\circ yy_\xi d\xi\\
& = -2\int_\Real P(X)\bar U_\xi d\xi -2c\int_\Real P(X)\chi'\circ yy_\xi d\xi+\frac{2}{3}c^3+k^2c\\ 
& \leq -2\int_\Real P(X)\bar U_\xi d\xi +C\norm{P(X)}_{L^\infty}+C.
\end{align*}
Finally, by adding these estimates, we get
\begin{align*}
 \frac{d\Sigma}{dt}& \leq C\Sigma+C+C\norm{\bar U}_{L^\infty}+C\norm{P(X)}_{L^\infty}\\
& \leq C\Sigma +C+C\norm{\bar U}_{L^\infty}^2+C\norm{P(X)}_{L^\infty}\\
& \leq C\Sigma+C,
\end{align*}
by \eqref{est:Uinf} and \eqref{est:Pinf}.  Hence,
Gronwall's lemma implies that
$\max_{t\in[0,T_0]}\Sigma(t)$ can be bounded by some constant only depending on $M_0$ and $T_0$. Using now \eqref{est:Uinf}, \eqref{est:Pinf}, and \eqref{est:Qinf}, we immediately obtain that the same is true for $\norm{\bar U(t,\dott)}_{L^\infty}$, $\norm{P(t,\dott)}_{L^\infty}$, and $\norm{Q(t,\dott)}_{L^\infty}$ with $t\in [0,T_0]$.
From \eqref{eq:sysdiss}, we obtain that 
\begin{equation}
 \vert \zeta(t,\xi)\vert \leq \vert \zeta(0,\xi)\vert +\int_0^t\abs{U(t',\xi)}dt', 
\end{equation}
and hence also $\norm{\zeta(t,\dott)}_{L^\infty}$ can be bounded  on $[0,T_0]$ by a constant only depending on $M_0$ and $T_0$. 

Applying Young's inequality to \eqref{eq:Plag1} and \eqref{eq:Qlag1} and following the proof of Lemma~\ref{lem:PQ}
we get 
\begin{equation}
\begin{aligned}
&\norm{(P(X)-U^2-\frac12 k^2)(t,\dott)}_{L^2}+\norm{Q(X)(t,\dott)}_{L^2}\\
& \quad \leq C(M_0,T_0)\\
&\qquad\quad+C(M_0,T_0)\big(\norm{\bar U(t,\dott)}_{L^2}+\norm{ \zeta_\xi(t,\dott)}_{L^2}+\norm{h(t,\dott)}_{L^2}+\norm{\bar r(t,\dott)}_{L^2}\big).
\end{aligned}
\end{equation}
Let 
$$
\alpha(t)=\norm{\bar U(t,\dott)}_{L^2}+\norm{\zeta_\xi(t,\dott)}_{L^2}+\norm{U_\xi(t,\dott)}_{L^2}+\norm{h(t,\dott)}_{L^2}+\norm{\bar r(t,\dott)}_{L^2},
$$
then
\begin{equation}
 \alpha(t)\leq \alpha(0)+C(M_0,T_0)+C(M_0,T_0)\int_0^t \alpha(t') dt'.
\end{equation}
Hence Gronwall's lemma gives us $\alpha(t)\leq C(M_0,T_0)$.

Similarly, one can show that 
$$
\norm{\zeta_\xi(t,\dott)}_{L^\infty}+\norm{U_\xi(t,\dott)}_{L^\infty}+\norm{h(t,\dott)}_{L^\infty}+\norm{\bar r(t,\dott)}_{L^\infty}\leq C(M_0,T_0).
$$

It remains to prove that $\norm{\frac{1}{y_\xi+h}}_{L^\infty_TL^\infty}$ can be bounded by some constant depending on $M_0$ and $T_0$, but this follows immediately form \eqref{eq:G3sol}. This completes the proof. 
\end{proof}

We can now prove  global existence of solutions. 
\begin{theorem}[Global solution] 
  \label{th:global}
  For any initial data $X_0=(y_0,U_0,h_0,r_0)\in\G$,
  there exists a unique global solution
  $X=(y,U,h,r)\in C(\Real_+,\G)$ of \eqref{eq:sysdiss} with $X(0)=X_0$. 
\end{theorem}

\begin{proof}
By assumption $X_0\in\G$, and therefore there exists a constant $M_0$ such that $X_0\in B_{M_0}$. By Theorem~\ref{th:short0} there exists a $T>0$ such that we can find a unique short time solution $X(t)\in \G$ on $[0,T]$. Moreover, according to Lemma~\ref{lem:estshort}, the length of the time interval for which the solution exists and is unique, is linked to $M_0$. Thus we can only find a unique global solution if $\norm{X(t)}_{\bar V}+\norm{\frac{1}{y_\xi+h}}_{L^\infty}$ does not blow up within a finite time interval, but this follows from Lemma~\ref{lem:globest}.
\end{proof}

\section{Stability of solutions} \label{sec:stability}

\begin{definition} \label{def:metric1}
The mapping $d_\Real \colon \G\times \G \to \Real_+$ 
\begin{equation}
 d_\Real(X,\tilde X)= \norm{X-\tilde X}_V+\norm{g(X)-g(\tilde X)}_{L^2(\Real)}+\kappa(X,\tilde X),
\end{equation}
 for $X$, $\tilde X\in \G$ defines a metric on $\G$. The function $\kappa(X,\tilde X)$ is defined as follows 
\begin{equation}\label{eq:defk}
 \kappa(X,\tilde X)=\begin{cases} 
1, &\quad \text{if } \meas(\{\xi\in\Real\mid (r(\xi)=0\text{ and }\tilde r(\xi)\not =0) \\
 & \phantom{\quad \text{if } \meas(\{\xi\in\Real\mid (r(\xi)=0} \text{ or }(r(\xi)\not =0\text{ and }\tilde r(\xi)=0)\})>0,\\
 0, & \quad \text{otherwise}.\\
 \end{cases}
\end{equation}
\end{definition}

In what follows we will denote
$$
d\colon\Real^8\times\Real^8\to\Real_+, \quad d(X, \tilde X)=\vert Z-\tilde Z\vert + \vert g(X)-g(\tilde X)\vert+\iota(X,\tilde X),
$$
with 
\begin{equation}\label{eq:i}
 \iota(X,\tilde X)=\begin{cases} 
1, &\quad \text{if } (r=0\text{ and }\tilde r\not =0) \text{ or }(r\not =0\text{ and }\tilde r=0),\\
 0, & \quad \text{otherwise}.\\
 \end{cases}
\end{equation}

Here it should be noted that for any two solutions $X$ and $\tilde X$ of the 2CH system, $\kappa(X,\tilde X)$ and $\iota(X,\tilde X)$ are independent of time, since $r_t=0$ and $\tilde r_t=0$. 

\subsection{Necessary estimates}

Denote 
\begin{equation} \label{eq:omega+}
\begin{aligned}
\Omega_-&=\{ x\in\Real^8 \mid x_5\leq 0 \text{ and } x_7+x_8x_4=0\}, \\
\Omega_+&=\{x\in\Real^8\mid x_5\geq 0 \text{ and }x_7+x_8x_4=0\}.
\end{aligned}
\end{equation}
Note that $\Omega_-\cup\Omega_+=\Omega_1\cup\Omega_2=\{x\in\Real^8\mid x_7+x_8x_4=0\}$ (cf.~Definition~\ref{def:Omega}). See Figure \ref{fig:omega_omrl}.
\begin{figure}
  \includegraphics[width=8cm]{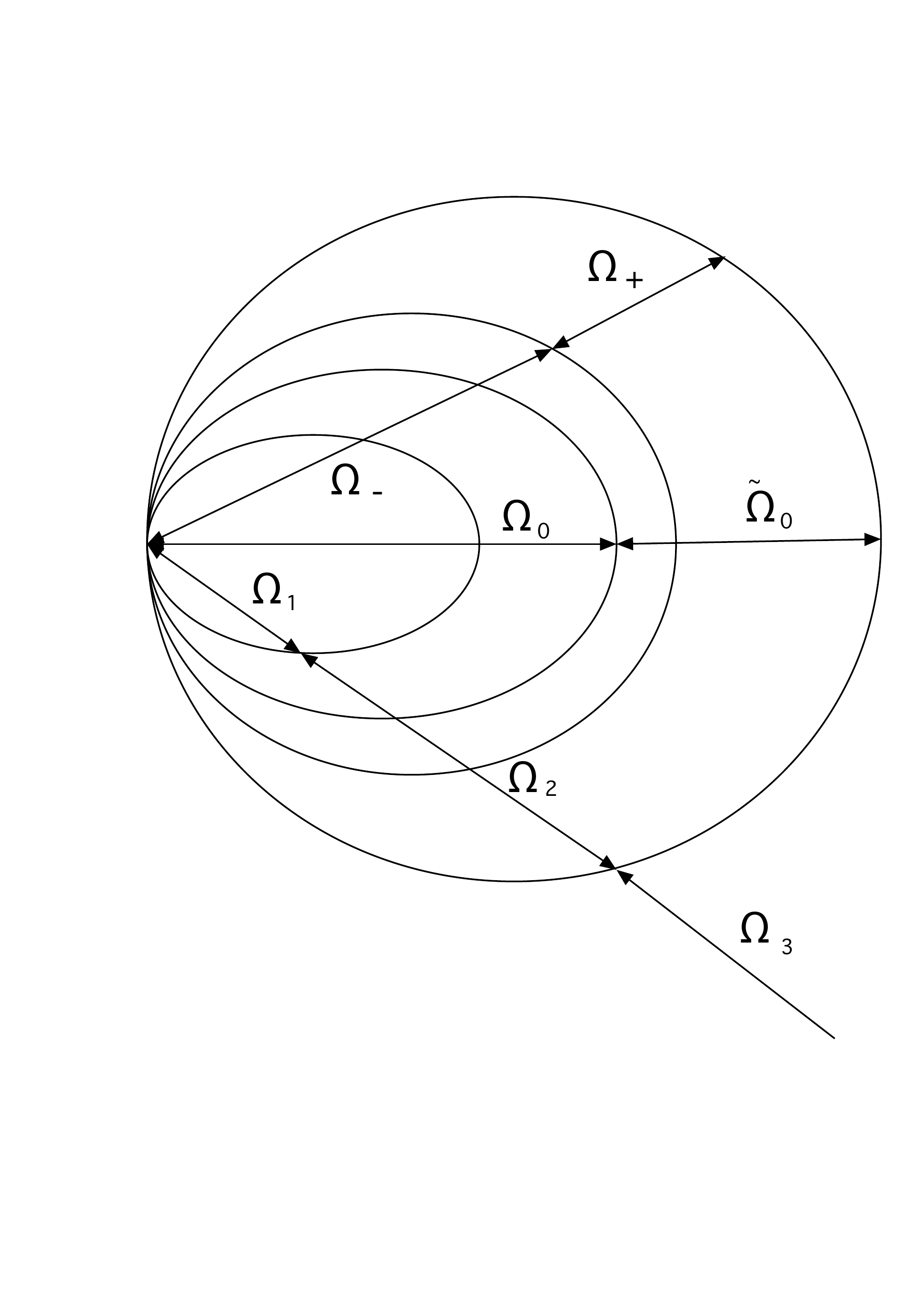}
  \label{fig:omega_omrl}
  \caption{The various regions, cf.~Definition~\ref{def:Omega} and \eqref{eq:omega+}.}
\end{figure}

%-------- lemma --
\begin{lemma}\label{lem:g1}
The restrictions of $g$ to $\Omega_2$, $\Omega_3$, and $\Omega_-$ are Lipschitz on bounded sets. More precisely,
\begin{equation}
 \vert  g(X)-g(\tilde X)\vert \leq C(\bar M)\left(\vert Z-\tilde Z\vert +(\vert \bar r\vert +\vert\bar{\tilde{r}}\vert)\vert k-\tilde k\vert\right) ,
\end{equation}
for any $X$, $\tilde X$ in $\Omega_2\cap B_{\bar M}$, $X$, $\tilde X$ in $\Omega_3\cap B_{\bar M}$, or $X$, $\tilde X$ in $\Omega_-\cap B_{\bar M}$.
\end{lemma}

\begin{proof}
  The cases when both $X$, $\tilde X\in \Omega_1$,
  $X$, $\tilde X\in\Omega_2$, and $X$, $\tilde
  X\in\Omega_3$ are straightforward. Let us
  consider the case when $X\in\Omega_1$ and
  $\tilde X\in \Omega_2\cap\Omega_-$, that is,
  $-U_\xi-2k\bar r+2y_\xi\leq y_\xi+h$ and $\tilde
  y_\xi+\tilde h\leq -\tilde U_\xi-2\tilde
  k\bar{\tilde{r}}+2\tilde y_\xi$,
  respectively. We have
  \begin{align}
   \vert g(X)-g(\tilde X)\vert & = \vert -U_\xi-2k\bar r+2y_\xi-\tilde y_\xi-\tilde h\vert\\ \nn 
& \leq \vert y_\xi-\tilde y_\xi\vert +\vert h-\tilde h\vert + h+U_\xi+2k\bar r-y_\xi\\ \nn 
& \leq \vert y_\xi-\tilde y_\xi\vert +\vert h-\tilde h\vert +h+U_\xi+2k\bar r-y_\xi-\tilde h-\tilde U_\xi-2\tilde k\bar{\tilde{r}}+\tilde y_\xi\\ \nn 
& \leq C(\bar M)(\vert Z-\tilde Z\vert +(\vert \bar r\vert +\vert\bar{\tilde{r}}\vert) \vert k-\tilde k\vert).
  \end{align}
\end{proof}

\begin{lemma}
\label{lem:estg}
 Given $M>0$, there exist $\bar M>0$, $T>0$ and $\delta>0$ which depend only on $M$ such that for any $\xi\in\Real$ satisfying $d(X_0(\xi),\tilde X_0(\xi))< \delta$, 
we have 
\begin{equation}\label{eq:g1}
 \vert g(X(t,\xi))-g(\tilde X(t,\xi))\vert \leq C(\bar M)\left(\vert Z(t,\xi)-\tilde Z(t,\xi)\vert+(\vert \bar r(t,\xi)\vert +\vert\bar{\tilde{r}}(t,\xi)\vert)\vert k-\tilde k\vert\right),
\end{equation}
for all $t\in[0,T]$ where $X(t)$ and $\tilde X(t)$ are the solutions to \eqref{eq:sysdiss} with initial data $X_0$ and $\tilde X_0$ belonging to $B_M$.
\end{lemma}

\begin{proof}
Without loss of generality we assume $T\leq 1$. We already know that there exists $\bar M$ only depending on $M$ such that $X(t)$, $\tilde X(t)\in B_{\bar M}$ for all $t\in[0,1]$.
We consider $\xi\in\Real$ such that $d(X_0(\xi),\tilde X_0(\xi))<\delta$ for a $\delta$ that we are going to determine. For simplicity we drop $\xi $ in the notation from now on. Since $X(t)$ and $\tilde X(t)$ are solutions of \eqref{eq:sysdiss}, we see that due to Lemma~\ref{lem:g1} the estimate \eqref{eq:g1} will be proved if we can show that either $X(t)$ and $\tilde X(t)$ belong to $\Omega_-$, $X(t)$ and $\tilde X(t)$ belong to $\Omega_2$, or $X(t)$ and $\tilde X(t)$ belong to $\Omega_3$. 

If $X_0\in \Omega_3$ and $\tilde X_0\in \Omega_3^c=\Omega_1\cup\Omega_2$, by \eqref{eq:i}, we have $d(X_0,\tilde X_0)\geq1$. Thus by choosing $\delta<\frac{1}2$, we impose that either $X_0,\tilde X_0\in\Omega_3$ or  $X_0,\tilde X_0\in\Omega_3^c$. In addition, since $r(t)=r(0)$ and $\tilde r(t)=\tilde r(0)$, we can conclude that points which initially lie inside $\Omega_3$ will remain in $\Omega_3$ for all times and points starting outside $\Omega_3$ can never enter $\Omega_3$. 

Hence in order to verify the claim it is left to show that either $X(t)$ and $\tilde X(t)$ belong to $\Omega_2$ or $X(t)$ and $\tilde X(t)$ belong to $\Omega_-$.
To this end denote by $\Omega_0$ and $\tilde\Omega_0$ the following sets  (cf.~Figure \ref{fig:omega_omrl})
\begin{equation}
\begin{aligned}
 \Omega_0&=\{x\in\Real^8\mid 2\vert x_7x_8\vert +x_4\leq x_6 \text{, }x_5\leq 0 \text{ and } x_7+x_8x_4=0\}, \\ 
 \tilde\Omega_0&=\Omega_2\setminus\Omega_0.
\end{aligned}
\end{equation}
We will distinguish three cases: 

(\textit{i}) $X_0$ and $\tilde X_0$ in $\Omega_0$: Since $\Omega_0\subset
\Omega_-$, we infer that $X_0$ and $\tilde X_0$ belong to $\Omega_-$. Hence $X(t)$ and $\tilde
X(t)$ will satisfy \eqref{eq:g1} as long as both $X(t)$ and $\tilde X(t)$ belong
to $\Omega_-$. Let us prove that $X(t)$ (and, in the same way, $\tilde X(t)$) for $T$ small enough remains in $\Omega_-$. Denote by $t_0$ the first time when $X(t)$ leaves $\Omega_-$. By
continuity we must have $U_\xi(t_0)=0$. Since $r(t)=r(0)=0$ for all $t$ in $[0,T]$, it implies
$y_\xi(t_0)h(t_0)=\bar r^2(t_0)=k_0^2y_\xi^2(t_0)$, by \eqref{eq:lagcoord6} and the definition \eqref{eq:rbar} of $\bar r(t)$. Hence, as the origin belongs to 
$\Omega_-$, either 
$y_\xi(t_0)=0$ or $h(t_0)=k_0^2y_\xi(t_0)$. Points reaching the origin remain
there, that is, if $y_\xi(t)=U_\xi(t)= \bar r(t)=0$ for $t=t_0$ then it remains
true for $t\geq t_0$. Hence we infer that, $y_\xi(t_0)\not =0$ and
$h(t_0)=k_0^2y_\xi(t_0)$. Let $z(t)=y_\xi(t)-2k_0\bar r(t)-h(t)$. We have, by
assumption, $z(0)\leq 0$, because $X_0\in \Omega_0$ and $\bar
r(0)=-k_0y_{\xi}(0)$. Using that $\bar r(t_0)=-k_0y_\xi(t_0)$ and $h(t_0)=k_0^2y_\xi(t_0)$, $z(t_0)=y_\xi(t_0)-2k_0\bar
r(t_0)-h(t_0)=y_\xi(t_0)+k_0^2y_{\xi}(t_0)= y_\xi(t_0)+h(t_0)\geq \frac{1}{\bar
  M}$. On the other hand we can compute $z_t$ and we obtain $z_t\leq C_1(\bar
M)$ for some constant $C_1(\bar M)$ only depending on $\bar M$ and therefore on
$M$. Thus $z(t)\leq z(0)+C_1(\bar M)T$. Hence if we choose $T$ small enough, that
means $T<(\bar MC(\bar M))^{-1}$, we obtain $z(t_0)<\frac{1}{\bar M}$, which is a
contradiction and we have proved that $X(t)$ remains in $\Omega_-$. Similarly
one proves that $\tilde X(t)$ remains in $\Omega_-$.

(\textit{ii}) $X_0\in\Omega_0$ and $\tilde X_0\in \tilde\Omega_0$: 
First of all we want to make sure that \eqref{eq:g1} holds at time $t=0$. According to Lemma~\ref{lem:g1}, this will be the case if we can prove that $d(X_0,\tilde X_0)<\delta$   implies either $X_0$, $\tilde X_0\in \Omega_-$ or $X_0$, $\tilde X_0\in \Omega_2$. If $X_0\in\Omega_2$, then \eqref{eq:g1} holds since $\tilde X_0\in \tilde\Omega_0\subseteq \Omega_2$. 

If $X_0\notin\Omega_2$, we have $X_0\in\Omega_1$ so that $X_0\in\Omega_-$. Assume that $\tilde X_0\notin\Omega_-$. Then $\vert U_{0,\xi}-\tilde U_{0,\xi}\vert \leq \delta$ implies
$\vert \tilde U_{0,\xi}\vert \leq\delta$ and $\vert U_{0,\xi}\vert \leq \delta$,
as $U_{0,\xi}$ and $\tilde U_{0,\xi}$ have opposite signs. Since
$X_0\in\Omega_0$, we have $y_{0,\xi}-2k_0\bar r_0\leq h_0$ which implies using
\eqref{eq:lagcoord6} that
\begin{equation*}
y_{0,\xi}^2-2k_0\bar r_0y_{0,\xi}= y_{0,\xi}^2+2\bar r_0^2\leq U_{0,\xi}^2+\bar r_0^2. 
\end{equation*}
Thus $y_{0,\xi}\leq\delta$, $\abs{\bar r_0}\leq\delta$ and we have 
\begin{equation}\label{est:deltadelta}
 \delta\geq g(\tilde X_0)-g(X_0)\geq \tilde y_{0,\xi} +\tilde h_0-\vert U_{0,\xi}\vert -2\vert k_0\bar r_0\vert-2y_{0,\xi}\geq \frac{1}{\bar M}-C(\bar M)\delta.
\end{equation}
Taking $\delta$ sufficiently small, we are led to a contradiction. Hence $\tilde X_0\in\Omega_-$.

We have
already seen in (\textit{i}) that we can choose $T$ so small that $X(t)$ remains
in $\Omega_-$ for $t\in[0,T]$. Let us denote $\tilde z(t)=\tilde
y_\xi(t)-2\tilde k_0\bar{\tilde {r}}(t)-\tilde h(t)$. For $z$ as defined in (\textit{i}),
we have $z(0)\leq 0$. Hence $\tilde z(0)\leq z(0)+\vert \tilde
z(0)-z(0)\vert\leq \vert y_\xi(0)-\tilde y_\xi(0)\vert +\vert h(0)-\tilde
h(0)\vert +\vert \bar r(0)\vert \vert k_0-\tilde k_0\vert +\vert \tilde k_0\vert\vert
\bar r(0)-\bar{\tilde{r}}(0)\vert $ and therefore $\tilde z(0)<C(\bar
M)\delta$. 
Let us now consider the first time $t_0$ when $\tilde X(t_0)$ leaves $\Omega_-$. Again, as in (\textit{i}), we obtain that $\tilde h(t_0)=k_0^2\tilde y_\xi(t_0)$ and $\tilde z(t_0)\geq \frac{1}{\bar M}$. In addition we know $\tilde z(t_0)\leq \tilde z_0+C(\bar M)T\leq C(\bar M)(\delta+T)$, which leads to a contradiction if we choose $T$ and $\delta$ small enough.

(\textit{iii}) $X_0$ and $\tilde X_0$ in
$\tilde\Omega_0$: In this case, since
$\tilde\Omega_0\subset \Omega_2$, $X_0$ and
$\tilde X_0$ belong to $\Omega_2$. We have
$z(0)\geq0$. Let us prove that $X(t)$ (and, in the
same way, $\tilde X(t)$) for $T$ small enough remains in
$\Omega_2$. Note that because the origin is a
repulsive point, solutions cannot reach the origin
from within $\Omega_2$. This means in particular
that $X(t)$ can only leave $\Omega_2$ if it starts
in $\Omega_-\cap \Omega_2$ or after entering
$\Omega_-\cap\Omega_2$. Therefore we assume
without loss of generality $X_0\in\Omega_-\cap
\Omega_2$. Denote by $t_0$ the first time when
$X(t)$ leaves $\tilde\Omega_2$. Then we must have
$\vert U_\xi(t_0)\vert+2\vert k_0\bar r\vert(t_0)
+2y_\xi(t_0)=y_\xi(t_0)+h(t_0)$ which gives
$U_\xi(t_0)=y_\xi(t_0)-2k_0\bar
r(t_0)-h(t_0)=z(t_0)$. Hence we obtain after some
computations using the latter equation together
with \eqref{eq:lagcoord6} and $\bar
r(t_0)+ky_\xi(t_0)=0$ that
$-z(t_0)=\frac{y_\xi(t_0)+h(t_0)}{\sqrt{5+4k_0^2}}\geq
\frac{1}{\bar M\sqrt{5+4\bar M^2}}$. Moreover,
$-z(t_0)\leq -z(0)+C_1(\bar M)T$ which implies
$\frac{1}{\bar M\sqrt{5+4\bar M^2}}\leq
\frac{y_\xi(t_0)+h(t_0)}{\sqrt{5+4k_0^2}}\leq
C_1(\bar M)T$, which leads to a contradiction if
we choose $T$ small enough. We have thus proved
that $X(t)$ remains in $\tilde\Omega_2$ and the same result
holds for $\tilde X(t)$.
\end{proof}

A close look at the proof of the last lemma shows that $\delta$ is only turning up in (\textit{ii}) and can be bounded by a constant only depending on $\bar M$. Hence, Lemma \ref{lem:estg} can be extended to the following lemma.

\begin{lemma}\label{lem:est:diffg}
 Given $M>0$ and $T>0$, then the solutions $X(t)$ and $\tilde X(t)$ with initial data $X_0$ and $\tilde X_0$ in $B_M$, respectively, belong to $B_{\bar M}$ for 
 $t\in[0,T]$.  For all $\bar t\in[0,T]$ there exist $\bar T\in[0,T]$ and $\delta$ positive depending on $\bar M$ and independent of $\bar t$, such that for any $\xi\in\Real$ satisfying $d(X(\bar t,\xi),\tilde X(\bar t,\xi))< \delta$, we have 
\begin{equation}
 \vert g(X(t,\xi))-g(\tilde X(t,\xi))\vert \leq C(\bar M)(\vert Z(t,\xi)-\tilde Z(t,\xi)\vert+(\vert\bar r\vert +\vert \bar{\tilde{r}}\vert)\vert k-\tilde k\vert)
\end{equation}
for all $t\in [\bar t,\bar t+\bar T]\cap [0,T]$.
\end{lemma}

The next lemma shows why the function $g$ is so important and why we choose exactly the metric $d_\Real$ instead of the norm $\norm{\dott}_V$ we used in the last section. 

\begin{lemma}\label{lem:tauttau}
 Given $M>0$ and $T>0$, then the solutions $X(t)$ and $\tilde X(t)$ with initial data $X_0$ and $\tilde X_0$ in $B_M$, respectively,  belong to $B_{\bar M}$ and we have for any $\xi\in\Real$ the following estimates:\\
\noindent 
(i) If $\tau(\xi)\leq t_1<t_2\leq \tilde \tau(\xi)$
\begin{align}
 \int_{t_1}^{t_2} \tilde h(t,\xi) dt &
 \leq \vert U_\xi(t_1,\xi)-\tilde U_\xi(t_1,\xi)\vert\\ \nn
 &\quad +C(\bar M)\int_{t_1}^{t_2} \big(\vert Z(t,\xi)-\tilde Z(t,\xi)\vert +\vert g(X(t,\xi))- g(\tilde X(t,\xi))\vert \big)dt.
\end{align}
(ii) If $\tilde\tau(\xi)\leq t_1<t_2\leq\tau(\xi)$
\begin{align}
 \int_{t_1}^{t_2} h(t,\xi) dt &
 \leq \vert U_\xi(t_1,\xi)-\tilde U_\xi(t_1,\xi)\vert\\ \nn
 &\quad +C(\bar M)\int_{t_1}^{t_2} \big(\vert Z(t,\xi)-\tilde Z(t,\xi)\vert +\vert g(X(t,\xi))- g(\tilde X(t,\xi))\vert \big)dt, 
\end{align}
for a constant $C(\bar M)$ depending only on $\bar M$.
\end{lemma}

\begin{proof}
We assume without loss of generality $\tau(\xi)<\tilde\tau(\xi)$. We  distinguish three cases.\\
(\textit{i}) If $\tilde X(t,\xi)\in \Omega_3$ for $t\in[t_1,t_2]$, we know $g(\tilde X)=\tilde y_\xi+\tilde h$ and $g(X)=0$ on $[t_1,t_2]$ because $y_\xi(t)=U_\xi(t)=\bar r(t)=0$ so that $X(t,\xi)\in\Omega_1$. Thus we get
\begin{equation}
 \int_{t_1}^{t_2}\tilde h(t,\xi)dt\leq \int_{t_1}^{t_2} \vert g(\tilde X(t,\xi))-g(X(t,\xi))\vert dt.
\end{equation}
As already pointed out several times before, $\tilde X(t)$ remains in $\Omega_3$ for all times if it starts in $\Omega_3$.

It is left to show what happens if $\tilde X(t,\xi)\in\Omega_3^c$ for all times. The result will follow from the following two estimates. \\
(\textit{ii}) If $\tilde X(t,\xi)\in \Omega_+$ for $t\in [t_1,t_2]$, we know $g(\tilde X)=\tilde y_\xi+\tilde h$ and $g(X)=0$ on $[t_1,t_2]$.  Hence we have 
\begin{equation}
 \int_{t_1}^{t_2}\tilde h(t,\xi)dt\leq \int_{t_1}^{t_2} \vert g(\tilde X(t,\xi))-g(X(t,\xi))\vert dt.
\end{equation}
(\textit{iii})  If $\tilde X(t_2,\xi)\in \Omega_-$, we know that $\tilde U_\xi(t_2,\xi)\leq 0$. Thus \eqref{eq:sysdiss} implies that 
\begin{equation*}
\begin{aligned}
 \frac12 \int_{t_1}^{t_2}\tilde h(t,\xi)dt
& =\tilde U_\xi(t_2,\xi)-\tilde U_\xi(t_1,\xi)\\
&\quad-\int_{t_1}^{t_2}(\tilde U^2+\frac12 \tilde k^2-P(\tilde X))\tilde y_\xi(t,\xi) dt-\int_{t_1}^{t_2}\tilde k\bar{\tilde{r}}(t,\xi)dt. 
\end{aligned}
\end{equation*}
Since $\tilde U_\xi(t_2,\xi)\leq 0$ and $U_\xi(t,\xi)=y_\xi(t,\xi)=\bar r(t,\xi)=0$ for all $t\in[t_1,t_2]$, we get
\begin{equation*}
\begin{aligned}
 \frac12 \int_{t_1}^{t_2} \tilde h(t,\xi)dt &\leq \vert U_\xi(t_1,\xi)-\tilde U_\xi(t_1,\xi)\vert\\ \nn
& \quad +C(\bar M)\int_{t_1}^{t_2}\vert y_\xi(t,\xi)-\tilde y_\xi(t,\xi)\vert dt+C(\bar M)\int_{t_1}^{t_2}\vert \bar r(t,\xi)-\bar{\tilde{r}}(t,\xi)\vert dt.
\end{aligned}
\end{equation*}

In general $\tilde X(t_2,\xi)\not\in\Omega_-$, does not imply that $\tilde X(t,\xi)\in \Omega_+$ for all $t\in [t_1,t_2]$.   
%In general $\tilde X(t,\xi)$ will not remain in $\Omega_+$ for all $t\in[t_1,t_2]$. It may happen that $\tilde X(t_2,\xi)\in
%\Omega_-$. 
Therefore we define $t_3$ as 
\begin{equation*}
  t_3=\inf\{t\in[t_1,t_2] \mid X(t',\xi) \in\Omega_+ \text{ for all } t'\geq t\}. 
\end{equation*}
Then, $\tilde X(t,\xi)\in \Omega_+$ for all
$t\in[t_3,t_2]$ and $U_\xi(t_3,\xi)\leq 0$ so that
the general case is proved by combining the cases
(\textit{ii}) and (\textit{iii}).
\end{proof}

\subsection{Stability results}

\begin{theorem}\label{thm:shortstab}
 Given $M>0$ there exist constants $\bar T\leq 1$ and $K$ depending only on $M$ such that for any initial data $X_0$ and $\tilde X_0$ in $B_M$
\begin{equation}
 \sup_{t\in[0,\bar T]}d_\Real(X(t), \tilde X(t))\leq Kd_\Real(X_0,\tilde X_0).
\end{equation}
%for any $T\leq \bar T$.
\end{theorem}

\begin{proof}
 We assume without loss of generality $T\leq 1$. We already know that there exists $\bar M$ only depending on $M$ (since $T\leq 1$) such that $X(t)$ and $\tilde X(t)$ belong to $B_{\bar M}$  for all $t\in[0,T]$. 
Let us introduce for the moment the following metric
\begin{equation}\label{eq:metric1}
\begin{aligned}
 \tilde d_\Real (X,\tilde X)& =\norm{y-\tilde y}_{L^\infty_T L^\infty}+ \norm{\bar U-\bar{\tilde{U}}}_{L^\infty_T L^2}+\vert c-\tilde c\vert \\ 
&\quad + \norm{Z-\tilde Z}_{L^\infty_T W}+\vert k-\tilde k\vert+\norm{g(X)-g(\tilde X)}_{L^\infty_T L^2}+\kappa(X,\tilde X).
\end{aligned}
\end{equation}
Then 
\begin{equation}
 d_\Real(X(t),\tilde X(t))\leq \tilde d_\Real(X,\tilde X), \quad t\in [0,T].
\end{equation}
This means in particular that the theorem is proved once we show that 
\begin{equation}\label{eq:claimstab2}
 \tilde d_\Real(X,\tilde X)\leq Kd_\Real(X_0,\tilde X_0).
\end{equation}

We estimate each of the terms in \eqref{eq:metric1}.  First, we want to show that
\begin{align}\label{est:QPstabAA}
 \norm{Q(X)-Q(\tilde X)}_{L^1_T E}&+\norm{\big(P(X)-U^2-\frac12 k^2\big)-\big(P(\tilde X)-\tilde U^2-\frac12 \tilde k^2\big)}_{L^1_TE}\\ \nn
& \leq C(\bar M)(d_\Real (X_0,\tilde X_0)+T\tilde d_\Real(X,\tilde X)).
\end{align}
To this end we first observe that by following closely the proof of Lemma~\ref{lem:PQ} one can show that 
\begin{align} \label{eq:4.19}
& \norm{Q(X)-Q(\tilde X)}_{L^1_T E}+\norm{\big(P(X)-U^2-\frac12 k^2\big)-\big(P(\tilde X)-\tilde U^2-\frac12 \tilde k^2\big)}_{L^1_TE}\\ \nn
&\qquad \leq C(\bar M)\Big(T\tilde d_\Real(X,\tilde X)\\ \nn
&\qquad\qquad+\int_\Real\big(\int_\tau^{\tilde\tau} \tilde h(t,\xi)\chi_{\{\tilde \tau(\xi)>\tau(\xi)\}}dt+\int_{\tilde\tau}^\tau  h(t,\xi)\chi_{\{\tau(\xi)>\tilde\tau(\xi)\}} dt \big)d \xi\Big). 
\end{align}
The first step towards proving the claim is to estimate properly the integral term on the right-hand side. Applying Lemma~\ref{lem:tauttau} yields 
\begin{align}\label{pluginest}
\int_\Real& \left(\int_\tau^{\tilde\tau} \tilde h(t,\xi)\chi_{\{\tilde \tau(\xi)>\tau(\xi)\}}dt+\int_{\tilde\tau}^\tau  h(t,\xi)\chi_{\{\tau(\xi)>\tilde\tau(\xi)\}} dt \right)d\xi\\ \nn 
& \leq \int_\Real (\vert U_\xi(\tau,\xi)-\tilde U_\xi(\tau,\xi)\vert \chi_{\{ \tau<\tilde\tau\}}+\vert  U_\xi(\tilde\tau,\xi)-\tilde U_\xi(\tilde\tau,\xi)\vert \chi_{\{\tilde\tau<\tau\}})d\xi \\ \nn 
& \qquad +C(\bar M) \int_\Real\int_0^T \big(\vert Z(t,\xi)-\tilde Z(t,\xi)\vert \\ \nn
&\qquad\qquad\qquad\qquad+\vert g(X(t,\xi))-g(\tilde X(t,\xi))\vert) (\chi_{\{\tau<\tilde\tau\}}+\chi_{\{\tilde\tau<\tau\}}\big)dt d\xi\\ \nn 
& \leq \int_\Real (\vert U_\xi(\tau,\xi)-\tilde U_\xi(\tau,\xi)\vert \chi_{\{ \tau<\tilde\tau\}}+\vert  U_\xi(\tilde\tau,\xi)-\tilde U_\xi(\tilde\tau,\xi)\vert \chi_{\{\tilde\tau<\tau\}})d\xi \\ \nn 
& \qquad +C(\bar M)T\big((\meas (\kappa_{1-\gamma}))^{1/2}+(\meas (\tilde\kappa_{1-\gamma}))^{1/2}\big) \tilde d_\Real(X,\tilde X), 
\end{align}
where we  in the last step used that $T$ is chosen so small that all points such that $\tau(\xi)<T$ or $\tilde\tau(\xi)<T$ belong to $\kappa_{1-\gamma}\cup\tilde\kappa_{1-\gamma}$. This is possible according to Lemma~\ref{lem:G2} (\textit{iv}). 

Fix $\xi\in\Real$ such that $0<\tau(\xi)<\tilde\tau(\xi)\leq T$.  Then $U_{\xi,t}=\frac12 h+(U^2+\frac12 k^2-P(X))y_\xi+k\bar r$ and $\tilde U_{\xi,t}=\frac12 \tilde h+(\tilde U^2+\frac12 \tilde k^2-P(\tilde X))\tilde y_\xi+\tilde k\bar{\tilde{r}}$ for $t\in[0,\tau(\xi)]$. Hence 
\begin{align} \label{eq:4.23}
 \vert U_{\xi}(\tau(\xi),\xi)&-\tilde U_{\xi}(\tau(\xi),\xi)\vert \\ \nn
 &  \leq \vert U_{\xi}(0,\xi)-\tilde U_\xi(0,\xi)\vert +\frac 12 \int_0^{\tau(\xi)} \vert h-\tilde h\vert (t,\xi)dt\\ \nn
& \quad +\int_0^{\tau(\xi)} \vert \big(U^2+\frac12 k^2-P(X)\big)-\big(\tilde U^2+\frac12 \tilde k^2-P(\tilde X)\big)\vert \vert y_\xi\vert (t,\xi) dt \\ \nn 
& \quad +\int_0^{\tau(\xi)} \vert \tilde U^2+\frac12 \tilde k^2-P(\tilde X)\vert \vert y_\xi-\tilde y_\xi\vert (t,\xi)dt\\ \nn
& \quad + \int_0^{\tau(\xi)} \vert k\vert\vert \bar r-\bar{\tilde{r}}\vert (t,\xi)dt+\int_0^{\tau(\xi)} \vert k-\tilde k\vert\vert \bar{\tilde{r}}\vert(t,\xi) dt \\ \nn
& \leq \vert U_\xi(0,\xi)-\tilde U_\xi(0,\xi)\vert \\ \nn
&\quad +C(\bar M) \int_0^{\tau(\xi)} \vert Z-\tilde Z\vert (t,\xi)dt +C(\bar M)T\vert k-\tilde k\vert\\ \nn 
& \quad + \gamma C(\bar M)\norm{\big(U^2+\frac12 k^2-P(X)\big)-\big(\tilde U^2+\frac12 \tilde k^2-P(\tilde X)\big)}_{L^1_TE}\\ \nn 
& \leq \vert U_\xi(0,\xi)-\tilde U_\xi(0,\xi)\vert\\ \nn
&\quad+ C(\bar M)\int_0^{\tau(\xi)} \vert Z-\tilde Z\vert(t,\xi)dt+C(\bar M)T\vert k-\tilde k\vert\\ \nn 
& \quad +\gamma C(\bar M)\Big(T\tilde d_\Real(X,\tilde X)\\ \nn
&\qquad\qquad\qquad\qquad+\int_\Real\big(\int_{\tau(\xi)}^{\tilde\tau(\xi)} \tilde h(t,\xi)\chi_{\{\tilde \tau(\xi)>\tau(\xi)\}}dt\\ \nn
&\qquad\qquad\qquad\qquad\qquad+\int_{\tilde\tau(\xi)}^{\tau(\xi)}  h(t,\xi)\chi_{\{\tau(\xi)>\tilde\tau(\xi)\}} dt \big)d\xi\Big).
\end{align}
Here we used that for $\xi\in\kappa_{1-\gamma}$, we have $y_\xi(t,\xi)=(y_\xi(t,\xi)+h(t,\xi))\frac{y_\xi(t,\xi)}{y_\xi(t,\xi)+h(t,\xi)}\leq \bar M\gamma$. In the last estimate we applied \eqref{eq:4.19}. For the case $0<\tilde\tau(\xi)<\tau(\xi)\leq T$, a similar treatment as in \eqref{eq:4.23} yields an estimate of the same form with every $\tau$ replaced by 
$\tilde\tau$ and vice versa.
Inserting this  estimate and the estimate \eqref{eq:4.23}  into \eqref{pluginest} implies, since 
$\meas(\tilde\kappa_{1-\gamma}), \meas(\kappa_{1-\gamma})\leq C(\bar M)$, that 
\begin{align}
 (1-\gamma C(\bar M))&\int_\Real \Big(\int_\tau^{\tilde\tau} \tilde h(t,\xi)\chi_{\{\tilde \tau(\xi)>\tau(\xi)\}}dt+\int_{\tilde\tau}^\tau  h(t,\xi)\chi_{\{\tau(\xi)>\tilde\tau(\xi)\}} dt \Big) d\xi\\ \nn 
&\qquad\qquad \leq C(\bar M)(d_\Real (X_0,\tilde X_0)+T\tilde d_\Real(X,\tilde X)).
 \end{align}
Choosing $\gamma$ small enough we find that 
\begin{align}
 &\int_\Real \Big(\int_\tau^{\tilde\tau} \tilde h(t,\xi)\chi_{\{\tilde \tau(\xi)>\tau(\xi)\}}dt+\int_{\tilde\tau}^\tau  h(t,\xi)\chi_{\{\tau(\xi)>\tilde\tau(\xi)\}} dt \Big)d\xi \\ \nn 
& \phantom{blablablablablabla} \leq C(\bar M)(d_\Real (X_0,\tilde X_0)+T\tilde d_\Real(X,\tilde X)),
\end{align}
 and especially 
\begin{align}\label{est:QPstab}
 \norm{Q(X)-Q(\tilde X)}_{L^1_T E}&+\norm{\big(P(X)-U^2-\frac12 k^2\big)-\big(P(\tilde X)-\tilde U^2-\frac12 \tilde k^2\big)}_{L^1_TE}\\ \nn
& \leq C(\bar M)(d_\Real (X_0,\tilde X_0)+T\tilde d_\Real(X,\tilde X)),
\end{align}
which is \eqref{est:QPstabAA}.

We now return to the proof of  \eqref{eq:claimstab2}, where we carefully investigate all terms in $\tilde d_\Real(X,\tilde X)$ separately using \eqref{est:QPstab}.
The equations $U_t-\tilde U_t=Q(\tilde X)-Q(X)$ and $U_0=\bar U_0+ c\chi\circ y_0$ imply 
\begin{align}\label{est:stab1}
 \norm{U(t,\dott)-\tilde U(t,\dott)}_{L^\infty}&\leq \norm{U_0-\tilde U_0}_{L^\infty}+\norm{Q(X)-Q(\tilde X)}_{L^1_T E}\\ \nn
& \leq \norm{U_0-\tilde U_0}_{L^\infty}+C(\bar M)(d_\Real(X_0,\tilde X_0)+T\tilde d_\Real(X,\tilde X))\\ \nn 
& \leq C(\bar M)(d_\Real(X_0,\tilde X_0)+T\tilde d_\Real(X,\tilde X)),
\end{align}
and $y_t-\tilde y_t=U-\tilde U$ yields 
\begin{align}\label{est:stab2}
 \norm{y(t,\dott)-\tilde y(t,\dott)}_{L^\infty}& \leq \norm{y_0-\tilde y_0}_{L^\infty}+T\norm{U-\tilde U}_{L^\infty_TL^\infty}\\ \nn 
& \leq C(\bar M)(d_\Real(X_0,\tilde X_0)+Td_\Real(X,\tilde X)).
\end{align}
Since $\bar U_t=-Q(X)-c\chi'\circ y U$ and
$\bar{\tilde{U}}_t=-Q(\tilde X)-\tilde c\chi'\circ
\tilde y\tilde U$, after combining
\eqref{est:QPstab}, \eqref{est:stab1} and
\eqref{est:stab2}, we obtain
\begin{equation}\label{est:stab3}
 \norm{\bar U(t,\dott)-\bar{\tilde{U}}(t,\dott)}_{L^2}\leq 
 C(\bar M)\big(d_\Real(X_0,\tilde X_0)+T\tilde d_\Real(X,\tilde X)\big).
\end{equation}
To estimate $\norm{Z-\tilde Z}_{L^\infty_TW}$ we split $\Real$ into two sets. Let us introduce $\mathcal{N}=\kappa_{1-\gamma}\cup \tilde\kappa_{1-\gamma}$. For $\xi\in\mathcal{N}^c$, we have $Z_t=F(X)Z$ and $\tilde Z_t=F(\tilde X)\tilde Z$ for all $t\in[0,T]$. Thus $Z_t-\tilde Z_t= (F(X)-F(\tilde X))Z+F(\tilde X)(Z-\tilde Z)$, and 
\begin{align*}
 \norm{(Z-\tilde Z)(t,\dott)}_{W(\mathcal{N}^c)}&
 \leq \norm{Z_0-\tilde Z_0}_{W}\\
 &\quad+\int_0^t\Big(\norm{(F(X)-F(\tilde X))Z(t',\dott)}_{W(\mathcal{N}^c)}\\
 &\qquad\qquad\qquad+\norm{F(\tilde X)(Z-\tilde Z)(t',\dott)}_{W(\mathcal{N}^c)}\Big)dt'.
\end{align*}
We get after applying Gronwall's lemma
\begin{multline}
  \norm{(Z-\tilde Z)(t,\dott)}_{W(\mathcal{N}^c)}\\
  \leq C(\bar M) \left(\norm{Z_0-\tilde Z_0}_{W}+\int_0^t\norm{(F(X)-F(\tilde X))Z(t',\dott)}_{W(\mathcal{N}^c)}dt'\right)
\end{multline}
for $t\in[0,T]$.  By definition, 
\begin{equation*}
\begin{aligned}
&(F(X)-F(\tilde X))Z \\
&\quad=\Big(0, (U^2+\frac12 k^2-P(X)-\tilde U^2-\frac12 \tilde k^2+P(\tilde X))y_\xi+(k-\tilde k)\bar r, \\
&\qquad\qquad2(U^2+\frac12 k^2-P(X)-\tilde U^2-\frac12 \tilde k^2+P(\tilde X))U_\xi,-(k-\tilde k)U_\xi\Big).
\end{aligned}
\end{equation*}
 Moreover, $y_\xi$, $U_\xi$, and $\bar r$ are bounded by some constants only depending on $\bar M$ and $\norm{\bar r}_{L^\infty_TL^2}+\norm{U_\xi}_{L^\infty_TL^2}\leq C(\bar M)$, so that,
\begin{align*}
 \norm{(F(X)-F(\tilde X))Z}_{L^1_TE}&\leq C(\bar M)\norm{\big(U^2+\frac12 k^2-P(X)\big)-\big(\tilde U^2+\frac12 \tilde k^2-P(\tilde X)\big)}_{L^1_TE}\\ 
& \quad +C(\bar M)T\vert k-\tilde k\vert,
\end{align*}
and hence
\begin{align}\label{est:104}
 \norm{(Z-\tilde Z)(t,\dott)}_{W(\mathcal{N}^c)}&\leq C(\bar M)(\norm{Z_0-\tilde Z_0}_{W}+T\vert k-\tilde k\vert)\\ \nn
& \quad +C(\bar M)\norm{\big(U^2+\frac12 k^2-P(X)\big)-\big(\tilde U^2+\frac12 \tilde k^2-P(\tilde X)\big)}_{L^1_TE}\\ \nn
& \leq C(\bar M) (d_\Real(X_0,\tilde X_0)+T\tilde d_\Real(X,\tilde X)).
\end{align}

For $\xi\in\mathcal{N}$, we assume without loss of generality $0\leq \tau(\xi)\leq\tilde\tau(\xi)\leq T$. Recall that $\meas(\kappa_{1-\gamma})\leq C(\bar M)$ and $\meas(\tilde\kappa_{1-\gamma})\leq C(\bar M)$ and hence $\meas(\mathcal{N})\leq C(\bar M)$. For $t\in[\tau,\tilde\tau]$, we have $Z(t,\xi)=Z(\tau,\xi)$ and $\tilde Z_t=F(\tilde X)\tilde Z$. Thus
\begin{equation}
 \frac{d}{dt}(\tilde Z-Z)=F(\tilde X)\tilde Z=F(\tilde X)(\tilde Z-Z)+F(\tilde X) Z,
\end{equation}
and after applying Gronwall's lemma we obtain 
\begin{align}
 \vert \tilde Z(t,\xi)-Z(t,\xi)\vert & \leq C(\bar M) \big(\vert Z(\tau,\xi)-\tilde Z(\tau,\xi)\vert +\int_\tau^{\tilde\tau}\vert F(\tilde X)Z(t',\xi)\vert dt'\big)\\ \nn 
& \leq C(\bar M) \big(\vert Z(\tau,\xi)-\tilde Z(\tau,\xi)\vert +\frac12\int_\tau^{\tilde\tau} h(\tau,\xi)dt'\big)
\end{align}
for $t\in[\tau,\tilde\tau]$. In particular,
\begin{align}
 \int_{\tau(\xi)}^{\tilde\tau(\xi)} h(\tau(\xi),\xi)dt&=\int_{\tau(\xi)}^{\tilde\tau(\xi)} h(t,\xi)dt=\int_{\tau(\xi)}^{\tilde\tau(\xi)} (h(t,\xi)-\tilde h(t,\xi))dt +\int_{\tau(\xi)}^{\tilde\tau(\xi)} \tilde h(t,\xi)dt\\\ \nn 
& \leq C(\bar M)\Big(\vert U_\xi(\tau(\xi),\xi)-\tilde U_\xi(\tau(\xi),\xi)\vert \\ \nn
&\quad+\int_{\tau(\xi)}^{\tilde\tau(\xi)}(\vert Z(t,\xi)-\tilde Z(t,\xi)\vert +\vert g(X(t,\xi))-g(\tilde X(t,\xi))\vert) dt\Big),
\end{align}
where we used Lemma~\ref{lem:tauttau} in the last step, and we get 
\begin{equation}\label{est:100}
\begin{aligned}
 \vert \tilde Z(t,\xi)-Z(t,\xi)\vert &\leq C(\bar M)\Big(\vert Z(\tau(\xi),\xi)-\tilde Z(\tau(\xi),\xi)\vert\\
 &\quad+\int_{\tau(\xi)}^{\tilde\tau(\xi)}(\vert Z(t',\xi)-\tilde Z(t',\xi)\vert +\vert g(X(t',\xi))-g(\tilde X(t',\xi))\vert) dt'\Big)
\end{aligned}
\end{equation}
when $t\in[\tau,\tilde\tau]$. For $t\leq \tau(\xi)$, we have $Z_t=F(X)Z$ and $\tilde Z_t=F(\tilde X)\tilde Z$, so that we obtain after applying Gronwall's lemma once more 
\begin{align}\label{est:101}
 \vert Z(\tau(\xi),\xi)-\tilde Z(\tau(\xi),\xi)\vert & \leq C(\bar M)(\vert Z(0,\xi)-\tilde Z(0,\xi)\vert +\norm{(F(X)-F(\tilde X))Z}_{L^1_TE})\\ \nn
& \leq C(\bar M)\big(\vert Z(0,\xi)-\tilde Z(0,\xi)\vert +d_\Real(X_0,\tilde X_0)+T\tilde d_\Real(X,\tilde X)\big).
\end{align}
Finally, combining \eqref{est:100} and \eqref{est:101} we end up with 
\begin{align}\label{est:102}
 \vert Z(t,\xi)-\tilde Z(t,\xi)\vert & \leq C(\bar M)\Big(\vert Z(0,\xi)-\tilde Z(0,\xi)\vert \\ \nn
 &\quad+\int_{\tau(\xi)}^{\tilde\tau(\xi)}(\vert Z(t',\xi)-\tilde Z(t',\xi)\vert +\vert g(X(t',\xi))-g(\tilde X(t',\xi))\vert ) dt'\\ \nn 
& \quad + d_\Real (X_0,\tilde X_0)+T\tilde d_\Real(X,\tilde X)\Big), \quad t\in[0,\tilde\tau].
\end{align}
Integrating \eqref{est:102} over the bounded domain $\mathcal{N}$ and applying Minkowski's inequality for integrals, then yields
\begin{equation}\label{est:103}
 \norm{Z(t,\dott)-\tilde Z(t,\dott)}_{W(\mathcal{N})}\leq C(\bar M)(d_\Real(X_0,\tilde X_0)+T\tilde d_\Real(X,\tilde X)).
\end{equation}
 Adding up \eqref{est:104} and \eqref{est:103}, we have 
\begin{equation}\label{est:105}
\norm{Z-\tilde Z}_{L^\infty_TW}\leq C(\bar M)(d_\Real(X_0,\tilde X_0)+T\tilde d_\Real(X,\tilde X)).
\end{equation}
Finally, it is left to estimate $\norm{g(X(t,\dott))-g(\tilde X(t,\dott))}_{L^2}$.
We have to distinguish several cases and therefore we introduce the set $$I=\{\xi\in\Real\mid \iota(X_0(\xi), \tilde X_0(\xi))=0\}.$$ 
(\textit{i}) If $\xi\in I$, we can choose $\delta\leq \frac12$ depending on $T$ and $M$ as in Lemma~\ref{lem:est:diffg}. If $\xi$ is such that $d(X_0(\xi),\tilde X_0(\xi))<\delta$, then $\vert g(X(t,\xi))-g(\tilde X(t,\xi))\vert \leq C(\bar M)(\vert Z(t,\xi)-\tilde Z(t,\xi)\vert+(\vert \bar r(t,\xi)\vert +\vert \bar{\tilde{r}}(t,\xi)\vert )\vert k-\tilde k\vert)$. On the other hand, if $\xi$ is such that $d(X_0(\xi),\tilde X_0(\xi))\geq \delta$, we have $\vert g(X(t,\xi))-g(\tilde X(t,\xi))\vert \leq C(\bar M)\frac{d(X_0(\xi),\tilde X_0(\xi))}{\delta}$ since $\vert g(X(t,\xi))\vert$ and $\vert g(\tilde X(t,\xi))\vert$ can be bounded by a constant only depending on $\bar M$. Thus we get that, since $\delta$ only depends on $\bar M$,
\begin{multline}
 \vert g(X(t,\xi))-g(\tilde X(t,\xi))\vert \\
 \leq C(\bar M)\big(\vert Z(t,\xi)-\tilde Z(t,\xi)\vert +(\vert \bar r\vert +\vert \bar{\tilde{r}}\vert)\vert k-\tilde k\vert +d(X_0(\xi),\tilde X_0(\xi))\big).
\end{multline} 
Note that since $\iota(X_0(\xi),\tilde X_0(\xi))=0$ by assumption, $d(X_0(\xi),\tilde X_0(\xi))$ is square integrable on $I$. Hence
\begin{align}\label{est:106}
 \norm{g(X(t,\dott))-g(\tilde X(t,\dott))}_{L^2(I)}& \leq C(\bar M)\big(\norm{Z(t,\dott)-\tilde Z(t,\dott)}_{W}+d_\Real(X_0,\tilde X_0)\big) \\ \nn
& \leq C(\bar M)(d_\Real(X_0,\tilde X_0)+T\tilde d_\Real(X,\tilde X)).
\end{align}

(\textit{ii}) If $\iota(X_0(\xi),\tilde X_0(\xi))=1$, either $X(t,\xi)\in \Omega_2$ or $X(t,\xi)\in\Omega_1$ and $\tilde X(t,\xi)\in\Omega_3$ (or the symmetric case where either $\tilde X(t,\xi)\in \Omega_2$ or $\tilde X(t,\xi)\in\Omega_1$ and $X(t,\xi)\in\Omega_3$). If $X(t,\xi)\in\Omega_2$, then $g(X(t,\xi))=y_\xi(t,\xi)+h(t,\xi)$ and $g(\tilde X(t,\xi))=\tilde y_\xi(t,\xi)+\tilde h(t,\xi)$, and we have 
$$\vert g(X(t,\xi))-g(\tilde X(t,\xi))\vert \leq \vert y_\xi(t,\xi)-\tilde y_\xi(t,\xi)\vert +\vert h(t,\xi)-\tilde h(t,\xi)\vert.$$
If $X(t,\xi)\in\Omega_1$, we know that $d(X_0(\xi),\tilde X_0(\xi))\geq 1$ and hence 
$$\vert g(X(t,\xi))-g(\tilde X(t,\xi))\vert \leq C(\bar M)d(X_0(\xi),\tilde X_0(\xi)).$$ 
Since the set $\Omega_1$ has finite measure, we get 
\begin{align}\label{est:116}
 \norm{g(X(t,\dott))-g(\tilde X(t,\dott))}_{L^2(I^c)}& \leq C(\bar M)\big(\norm{Z(t,\dott)-\tilde Z(t,\dott)}_{W}+d_\Real(X_0,\tilde X_0)\big) \\ \nn
& \leq C(\bar M)(d_\Real(X_0,\tilde X_0)+T\tilde d_\Real(X,\tilde X)).
\end{align}

Combining \eqref{est:stab2}, \eqref{est:stab3}, \eqref{est:105}, \eqref{est:106}, and \eqref{est:116} yields
\begin{equation}
 \tilde d_\Real(X,\tilde X)\leq C(\bar M)(d_\Real(X_0,\tilde X_0)+T\tilde d_\Real(X,\tilde X)), 
\end{equation}
which implies, if we choose $T$ small enough, that
\begin{equation}
 \tilde d_\Real(X,\tilde X)\leq C(\bar M)d_\Real(X_0,\tilde X_0).
\end{equation}
This finishes the proof.
\end{proof}

\begin{theorem}\label{thm:shortstab2}
 Given $M>0$ and $T>0$, then the solutions $X(t)$ and $\tilde X(t)$ with initial data $X_0$ and $\tilde X_0$, respectively, in $B_M$, belong to $B_{\bar M}$ for 
 $t\in [0,T]$.  For any given $\tilde t\in[0,T]$, there exist $K$ and $\tilde T$ depending on $\bar M$ and independent of $\tilde t$, such that 
\begin{equation}
 \sup_{t\in[\tilde t,\tilde t+\tilde T]\cap[0,T]} d_\Real (X(t),\tilde X(t))\leq Kd_\Real(X(\tilde t),\tilde X(\tilde t)).
\end{equation}
\end{theorem}

\begin{proof}
 A close inspection of the proof of Theorem \ref{thm:shortstab} shows that all our estimates use upper bounds, and hence we can replace the constants therein depending on $\bar M$  by the $\bar M$ in the statement of the present theorem. Since $X(t)$ belongs to $B_{\bar M}$ for all $t\in[0,T]$, we can apply Lemma~\ref{lem:G2} (\textit{iv}) which tells us that for any initial time $\tilde t$ we can find for any $\gamma$ a time interval such that all points enjoying wave breaking are contained in $\kappa_{1-\gamma}$ and this time interval is independent of the initial time $\tilde t$. Hence, after this observation we can follow the proof of Theorem \ref{thm:shortstab}.
\end{proof}

\begin{theorem}
 For any time $T>0$ there exists a constant $K$ only depending on $M$ and $T$ such that 
\begin{equation}
 \sup_{t\in[0,T]}d_\Real (X(t),\tilde X(t))\leq Kd_\Real(X_0,\tilde X_0)
\end{equation}
for any solutions $X(t)$ and $\tilde X(t)$ in $B_{\bar M}$, $t\in[0,T]$, with initial data $X_0$ and $\tilde X_0$, respectively, in $B_M$.
\end{theorem}

\begin{proof}
 There exists $\bar M$ only depending on $M$ and $T$ such that $X(t)$ and $\tilde X(t)$ belong to $B_{\bar M}$ for all $t\in[0,T]$, cf.~Theorem~\ref{th:global}. From the short time stability result Theorem~\ref{thm:shortstab2} we know that there exist constants $K$ and $\bar T$ depending only on  $\bar M$ such that 
\begin{equation}
 d_\Real(X(t), \tilde X(t))\leq Kd_\Real(X(\tilde t),\tilde X(\tilde t))
\end{equation}
for any $t\in[0,T]\cap [\tilde t, \tilde t+\bar T]$. To obtain global stability we therefore split up the interval $[0,T]$ into smaller time intervals where the last inequality is valid. For any $T>0$ there exists $N\in\mathbb{N}$ such that $T\leq (N+1)\bar T$ and accordingly we define $t_0=0$, $t_1=\bar T$,$\dots$, $t_N=N\bar T$, and $t_{N+1}=T$.
Hence $d_\Real(X(t),\tilde X(t))\leq Kd_\Real(X(t_i),\tilde X(t_i))$ for all $t\in[t_i,t_{i+1}]$, due to the last lemma. Hence we finally obtain 
\begin{equation}
 d_\Real(X(t),\tilde X(t))\leq K^{N+1}d_\Real(X_0,\tilde X_0),
\end{equation}
which proves the claim.
\end{proof}

In addition to the last stability result, one can also show the Lipschitz continuity of every solution with respect to time.

\begin{lemma} \label{lem:Lip}
Given $M>0$ and $T>0$, then we have for any solution $X(t)$ with initial data $X_0\in B_M$,  
\begin{equation}
 d_\Real(X(t),X(\bar t))\leq C_T(M)\abs{\bar t-t}, \quad t,\bar t \leq T,
\end{equation}
for a constant $C_{T}(M)$ which only depends on $M$ and $T$.
\end{lemma}

\begin{proof}
 We already know that for any $T>0$, we have $\norm{X(t,\dott)}_{\bar V}+\norm{g(X(t,\dott))-1}_{L^2}$ for $t\in[0,T]$, can be bounded by some constant $C_T(M)$ depending on $M$ and $T$. Then we have 
 \begin{subequations}
\begin{align}
  \norm{\zeta(t,\dott)-\zeta(\bar t,\dott)}_{L^\infty}&\leq \int_{\bar t}^t \norm{U(t',\dott)}_{L^\infty}dt'\leq C_T(M)\vert \bar t-t\vert,\\
  \norm{\bar U(t,\dott)-\bar U(\bar t,\dott)}_{L^2}&\leq \int_{\bar t}^t \norm{-Q(X)(t',\dott)-c\chi'\circ  y(t',\dott)\bar U(t',\dott)}_{L^2}dt'\\
  &\leq C_T(M)\vert \bar t-t\vert,  \nn\\
  \norm{Z(t,\dott)-Z(\bar t,\dott)}_{W}&\leq \int_{\bar t}^t \norm{F(X)(t',\dott)Z(t',\dott)}_{L^2}dt'\leq C_T(M)\vert \bar t-t\vert.
\end{align}
\end{subequations}

Hence 
\begin{equation}
 \norm{X(t,\dott)-X(\bar t,\dott)}_V\leq C_T(M)\vert \bar t-t\vert.
\end{equation}

It remains to estimate $\norm{g(X(\bar t,\xi))-g(X(t,\xi))}_{L^2}$.
If for a given $\xi\in\Real$, $X(t,\xi)\in \Omega_1$ and $X(\bar t,\xi)\in\Omega_1$, we have 
\begin{equation}\label{g1}
 \vert g(X(t,\xi))-g(X(\bar t,\xi))\vert= \int_{\bar t}^t \vert g_{1,t}(X(t',\xi))\vert dt',
\end{equation}
where, slightly abusing the notation, $g_{1,t}$ denotes the time derivative of $-U_\xi-2k\bar r+2y_\xi$.
Similarly if $X(t,\xi)\in\Omega_1^c$ and $X(\bar t,\xi)\in\Omega_1^c$, we get 
\begin{equation}
 \vert g(X(t,\xi))-g(X(\bar t,\xi))\vert= \int_{\bar t}^t \vert g_{2,t}(X(t',\xi))\vert dt',
\end{equation}
where $g_{2,t}$ denotes the time derivative of $y_\xi+h$.
If $X(t,\xi)\in \Omega_1$ and $X(\bar t,\xi)\in\Omega_1^c$ (the case $X(t,\xi)\in\Omega_1^c$, $X(\bar t,\xi)\in\Omega_1$ can be treated in much the same way), we can find a $\tilde t\in(\bar t,t)$ such that $g_1(X(\tilde t,\xi))=g_2(X(\tilde t,\xi))$ and therefore 
\begin{align}\label{g2}
 \vert g(X(t,\xi))-g(X(\bar t,\xi))\vert &\leq \vert g_1(X(t,\xi))-g_1(X(\tilde t,\xi))\vert +\vert g_2(X(\tilde t,\xi))-g_2(X(\bar t,\xi))\vert\\ \nn
& \leq \int_{\bar t}^{\tilde t} \vert g_{2,t}(X(t',\xi))\vert dt'+\int_{\tilde t}^t \vert g_{1,t}(X(t',\xi))\vert dt'\\ \nn
& \leq \int_{\bar t}^t \vert g_{1,t}(X(t',\xi))\vert +\vert g_{2,t}(X(t',\xi))\vert dt'.
\end{align}
Since $\norm{g_{1,t}(X(t',\dott))}_{L^2}$ and $\norm{g_{2,t}(X(t',\dott))}_{L^2}$, can be uniformly bounded by a constant $C_T(M)$ for all $t'\in[\bar t,t]$, we get after using \eqref{g1}-\eqref{g2} together with applying Minkowski's inequality for integrals, that 
\begin{align}
 \norm{g(X(t,\xi))-g(X(\bar t,\xi))}_{L^2}&\leq \int_{\bar t}^t \norm{g_{1,t}(X(t',\dott))}_{L^2}+\norm{g_{2,t}(X(t',\dott))}_{L^2} dt'\\ \nn 
&\leq C_T(M)\vert \bar t-t\vert.
\end{align}

\end{proof}

\section{From Eulerian to Lagrangian variables and vice versa}

So far the derivation of our system of ordinary differential equations \eqref{eq:sysdiss} in Lagrangian coordinates is only valid for initial data $u_0$ in Eulerian coordinates, where no concentration of mass takes place. However, it is well known that in the case of conservative solutions, concentration of mass is linked to wave breaking. Since our description of dissipative solutions in Lagrangian variables until wave breaking occurs, coincides with the one used in \cite{GHR4} for conservative solutions, one might hope that the sets of Eulerian and Lagrangian coordinates can be described in much the same way, and that the mappings from Eulerian to Lagrangian coordinates and vice versa can be defined using the same ideas. It turns out that we can do so. For the sake of completeness we summarize these results here. We start by introducing the set of Eulerian and Lagrangian coordinates together with the set of relabeling functions which allows us to identify equivalence classes in the Lagrangian variables.

\begin{definition} [Eulerian coordinates]\label{def:D} The set $\D$ is
  composed of all triplets $(u,\rho,\mu)$ such that
  $u\in H_{0,\infty}(\Real)$, $\rho\in L^2_{\rm
    const}(\Real)$ and $\mu$ is a positive finite
  Radon measure whose absolutely continuous part,
  $\muac$, satisfies
\begin{equation}
\label{eq:abspart}
\muac=(u_x^2+\bar\rho^2)\,dx.
\end{equation}
\end{definition}

\begin{definition}[Relabeling functions]\label{def:G1}
 We denote by $G$ the subgroup of the group of homeomorphisms from $\Real$ to $\Real$ such that 
\begin{subequations}
\label{eq:Gcond}
 \begin{align}
  \label{eq:Gcond1}
  f-\id \text{ and } f^{-1}-\id &\text{ both belong to } W^{1,\infty}(\Real), \\
  \label{eq:Gcond2}
  f_\xi-1 &\text{ belongs to } L^2(\Real),
 \end{align}
\end{subequations}
where $\id$ denotes the identity function. Given $\kappa>0$, we denote by $G_\kappa$ the subset  of $G$ defined by 
\begin{equation}
 G_\kappa=\{ f\in G\mid  \norm{f-\id}_{W^{1,\infty}}+\norm{f^{-1}-\id}_{W^{1,\infty}}\leq\kappa\}. 
\end{equation}
\end{definition}

\begin{definition}[Lagrangian coordinates]
The subsets $\F$ and $\F_\kappa$ of $\G$ are defined as
 \begin{equation*}
\mathcal{F}_\kappa=\{X=(y,U,h,r)\in\mathcal{G}\mid  y+H\in G_\kappa\},
\end{equation*}
and
\begin{equation*}
\mathcal{F}=\{X=(y,U,h,r)\in\mathcal{G}\mid  y+H\in G\},
\end{equation*}
where $H(t,\xi)$ is defined by 
\begin{equation*}
 H(t,\xi)=\int_{-\infty}^\xi h(t,\tilde\xi)d\tilde\xi.
\end{equation*}
\end{definition}

Note that $H(t,\xi)$ is finite, since from \eqref{eq:lagcoord6}, we have $h=U_\xi^2+\bar r^2-\zeta_\xi h$ and therefore $h\in L^1(\Real)$. In addition it should be pointed out that the condition on $y+H$ is closely linked to $\norm{\frac{1}{y_\xi+h}}_{L^\infty}$ as the following lemma shows.

\begin{lemma}[{\cite[Lemma 3.2]{HolRay:07}}] 
\label{lem:charH}
Let $\kappa\geq0$. If $f$ belongs to $G_\kappa$,
then $1/(1+\kappa)\leq f_\xi\leq 1+\kappa$ almost
everywhere. Conversely, if $f$ is absolutely
continuous, $f-\id\in W^{1,\infty}(\Real)$, $f$ satisfies
\eqref{eq:Gcond2} and there exists $d\geq 1$ such
that $1/d\leq f_\xi\leq d$ almost everywhere, then
$f\in G_\kappa$ for some $\kappa$ depending only
on $d$ and $\norm{f-\id}_{W^{1,\infty}}$.
\end{lemma}

An immediate consequence of \eqref{eq:lagcoord2} is therefore the following result.

\begin{lemma}\label{lem:Fpres}
 The space $\mathcal{G}$ is preserved by the governing equations \eqref{eq:sysdiss}.
\end{lemma}

For the sake of simplicity, for any $X=(y, U, h,r)\in \mathcal{F}$ and any function
$f\in\Gr$, we denote $(y\circ f, U\circ f, h\circ ff_\xi, r\circ f f_\xi)$ by $X\circ f$.

\begin{proposition} The map 
from $\Gr\times\F$ to $\F$ given by $(f,X)\mapsto
X\circ f$ defines an action of the group $\Gr$ on
$\F$.
\end{proposition}

Since $\Gr$ is acting on $\F$, we can consider the
quotient space $\quot$ of $\F$ with respect to the
action of the group $G$. The equivalence relation
on $\F$ is defined as follows: For any
$X,X'\in\F$, we say that $X$ and $X'$ are equivalent if there
exists a relabeling function $f\in\Gr$ such that $X'=X\circ f$. We
denote by $\Pi(X)=[X]$ the projection of $\F$ into the
quotient space $\quot$, and introduce the mapping
$\Gamma\colon\F\rightarrow\F_0$ given by
\begin{equation*}
\Gamma(X)=X\circ (y+H)\inv
\end{equation*}
for any $X=(y,U,h,r)\in\F$. We have $\Gamma(X)=X$ when
$X\in\F_0$. It is not hard to prove that $\Gamma$ is
invariant under the action of $\Gr$, that is,
$\Gamma(X\circ f)=\Gamma(X)$ for any $X\in\F$ and
$f\in\Gr$. Hence, there corresponds to $\Gamma$ a
mapping $\tilde\Gamma$ from the quotient space $\quot$ to
$\F_0$ given by $\tilde\Gamma([X])=\Gamma(X)$ where
$[X]\in\quot$ denotes the equivalence class of
$X\in\F$. For any $X\in\F_0$, we have
$\tilde\Gamma\circ\Pi(X)=\Gamma(X)=X$. Hence, $\tilde\Gamma\circ
\Pi|_{\F_0}=\id|_{\F_0}$. Any topology defined on
$\F_0$ is naturally transported into $\quot$ by
this isomorphism. We equip $\F_0$ with the metric
induced by the $E$-norm, i.e.,
$d_{\F_0}(X,X')=d_\Real(X,X')$ for all
$X,X'\in\F_0$. Since $\F_0$ is closed in $E$, this
metric is complete. We define the metric on
$\quot$ as
\begin{equation*}
d_\quot([X],[X'])=d_\Real(\Gamma(X),\Gamma(X')),
\end{equation*}
for any $[X],[X']\in\quot$. Then, $\quot$ is
isometrically isomorphic with $\F_0$ and the
metric $d_\quot$ is complete. As in
\cite{HolRay:07}, we can prove the following
lemma.
\begin{lemma}
\label{lem:picont} Given $\alpha\geq 0$.
The restriction of $\Gamma$ to $\F_\alpha$ is a
continuous mapping from $\F_\alpha$ to $\F_0$.
\end{lemma}

\begin{remark}
The mapping $\Gamma$ is not continuous from $\F$ to
$\F_0$. The spaces $\F_\alpha$ were precisely
introduced in order to make the mapping $\Gamma$
continuous.
\end{remark}

We denote by $S\colon\F\times [0,\infty)\rightarrow \F$
the continuous semigroup which to any initial
data $X_0\in \F$ associates the solution $X(t)$
of the system of differential equations
\eqref{eq:sysdiss} at time $t$. As indicated earlier,
the two-component Camassa--Holm system is invariant with
respect to relabeling.  More precisely, using our
terminology, we have the following result.

\begin{theorem} 
\label{th:sgS} 
For any $t>0$, the mapping $S_t\colon\F\rightarrow\F$
is $\Gr$-equivariant, that is,
\begin{equation}
\label{eq:Hequivar}
S_t(X\circ f)=S_t(X)\circ f
\end{equation}
for any $X\in\F$ and $f\in\Gr$. Hence, the mapping
$\tilde S_t$ from $\quot$ to $\quot$ given by
\begin{equation*}
\tilde S_t([X])=[S_tX]
\end{equation*}
is well-defined. It generates a continuous
semigroup.
\end{theorem}

We have the
following diagram:
\begin{equation}
\label{eq:diag}
\xymatrix{
\F_0\ar[r]^{\Pi}&\quot\\
\F_\alpha\ar[u]^{\Gamma}&\\
\F_0\ar[u]^{S_t}\ar[r]^\Pi&\quot\ar[uu]_{\tilde S_t}
}
\end{equation}

Next we describe the correspondence
between Eulerian coordinates (functions in $\D$)
and Lagrangian coordinates (functions in
$\quot$). In order to do so, we have to take into account the fact that the set $\D$ allows the energy
density to have a singular part and a positive
amount of energy can concentrate on a set of
Lebesgue measure zero.

We first define the mapping $L$ from $\D$ to $\F_0$ which to
any initial data in $\D$ associates an initial
data for the equivalent system in $\F_0$.

\begin{theorem} 
\label{th:Ldef}
For any $(u,\rho,\mu)$ in $\D$, let
\begin{subequations}
\label{eq:Ldef}
\begin{align}
\label{eq:Ldef1}
y(\xi)&=\sup\left\{y\ |\ \mu((-\infty,y))+y<\xi\right\},\\
\label{eq:Ldef2}
h(\xi)&=1-y_\xi(\xi),\\
\label{eq:Ldef3}
U(\xi)&=u\circ{y(\xi)},\\
\label{eq:Ldef4}
r(\xi)&=\rho\circ{y(\xi)}y_\xi(\xi).
\end{align}
\end{subequations}
Then $(y,U,h,r)\in\F_0$. We denote by $L\colon \D\rightarrow \F_0$ the mapping which to any element $(u,\rho,\mu)\in\D$ associates $X=(y,U,h,r)\in \F_0$ given by \eqref{eq:Ldef}.  
\end{theorem}

On the other hand, to any element in $\F$ there
corresponds a unique element in $\D$ which is
given by the mapping $M$ defined below.

\begin{theorem}
\label{th:umudef} 
Given any element $X=(y,U,h,r)\in\F$. Then, the
measure $y_{\#}(\bar r(\xi)\,d\xi)$ is absolutely
continuous, and we define $(u,\rho,\mu)$ as follows, for any $\xi$ such that  $x=y(\xi)$, 
\begin{subequations}
\label{eq:umudef}
\begin{align}
\label{eq:umudef1}
&u(x)=U(\xi),\\
\label{eq:umudef2}
&\mu=y_\#(h(\xi)\,d\xi),\\
\label{eq:umudef3}
&\bar \rho(x)\,dx=y_\#(\bar r(\xi)\,d\xi),\\
\label{eq:umudef4}
&\rho(x)=k+\bar\rho(x).
\end{align}
\end{subequations}
We have that $(u,\rho,\mu)$ belongs to $\D$. We
denote by $M\colon \F\rightarrow\D$ the mapping
which to any $X$ in $\F$ associates the element $(u, \rho,\mu)\in \D$ as given
by \eqref{eq:umudef}. In particular, the mapping $M$ is
invariant under relabeling.
\end{theorem}

Finally, one has to declare the connection between the equivalence classes in Lagrangian coordinates and the set of Eulerian coordinates.

\begin{theorem}\label{th:LMinv}The mappings $M$ and
  $L$ are invertible. We have
\begin{equation*}
L\circ M=\id_\quot\text{ and }M\circ L=\id_\D.
\end{equation*}
\end{theorem}

\section{Continuous semigroup of solutions}

In the last section we defined the connection between Eulerian and Lagrangian coordinates, which is the main tool when defining weak solutions of the 2CH system. Also stability results will heavily depend on this relation since we want to measure distances between solutions of the 2CH system by measuring the distance in Lagrangian coordinates rather than in Eulerian coordinates. 

Accordingly, we define $T_t$ as
\begin{equation*}
  T_t=M\circ S_t\circ L.
\end{equation*}

The metric $d_\D$ is defined as 
\begin{equation}\label{eq:meter}
  d_D((u_1,\rho_1,\mu_1),(u_2,\rho_2,\mu_2))=d_{\F_0}(L(u_1,\rho_1,\mu_1),L(u_2,\rho_2,\mu_2)).
\end{equation}

\begin{definition}
  \label{eq:defweakconssol}
  Assume that $u\colon[0,\infty)\times\Real \to \Real$  and $\rho\colon [0,\infty) \times\Real \to\Real$ satisfy \\
  (i) $u\in L^\infty_{\rm{loc}}([0,\infty), H_\infty(\Real))$ and $\rho\in L^\infty_{\rm{loc}}([0,\infty), L^2_{\rm const}(\Real))$, \\
  (ii) the equations
  \begin{multline}\label{eq:weak1}
    \iint_{[0,\infty)\times\Real}\Big[
    -u(t,x)\phi_t(t,x)
    +\big(u(t,x)u_x(t,x)+P_x(t,x)\big)\phi(t,x)\Big]dxdt\\
    =\int_\Real u(0,x)\phi(0,x)dx,
  \end{multline}
  \begin{equation}\label{eq:weak2}
    \iint_{[0,\infty)\times\Real} \Big[(P(t,x)-u^2(t,x)-\frac{1}{2}u_x^2(t,x)-\frac{1}{2}\rho^2(t,x))\phi(t,x)+P_x(t,x)\phi_x(t,x)\Big]dxdt=0,
  \end{equation}
  and
  \begin{equation}
   \label{eq:weak3}
   \iint_{[0,\infty)\times\Real}\Big[ -\rho(t,x)\phi_t(t,x)-u(t,x)\rho(t,x)\phi_x(t,x)\Big]dxdt=\int_\Real \rho(0,x)\phi(0,x)dx,
  \end{equation}
  hold for all $\phi\in
  C^\infty_0([0,\infty)\times\Real)$. Then we say that
  $(u,\rho)$ is a weak global solution of the two-component
  Camassa--Holm system. 
\end{definition}

\begin{theorem} \label{th:mainX}
  The mapping $T_t$ is a continuous
  semigroup of solutions with respect to the metric
  $d_\D$.  Given any initial data
  $(u_0,\rho_0,\mu_0)\in\D$, let
  $(u(t,\dott),\rho(t,\dott),\mu(t,\dott))=T_t(u_0,\rho_0,\mu_0)$. Then
  $(u,\rho)$ is a weak solution to
  \eqref{eq:sysdiss}
and $(u,\rho,\mu)$ is a weak solution to 
\begin{equation}\label{eq:weak5}
 (u^2+\mu+\rho^2-\bar \rho^2)_t+(u(u^2+\mu+\rho^2-\bar\rho^2))_x=(u^3-2Pu)_x.
\end{equation}
The function
\begin{equation}
F(t)=\int_\Real d\mu(t,x)-\int_\Real \big(u_x^2(t,x)+\bar \rho^2(t,x) \big)dx
\end{equation}
which is an increasing function, equals the
amount of energy that has concentrated at sets of
measure zero up to time $t$.  
Moreover, for every $t\in [0,\infty)$, we clearly have 
\begin{equation}\label{eq:weak4}
 \int_\Real \big(u_x^2(t,x)+\bar\rho^2(t,x)\big)dx =\int_\Real d\mu_{ac}(t,x)\leq \int_\Real d\mu(t,x).
\end{equation}
\end{theorem}
\begin{remark} (i) Equation \eqref{eq:weak5} is also valid in the conservative 
case, cf.~\cite[Thm.~5.2]{GHR4}.  However, note that there is a difference in the definition of the quantity $P$ in the two cases. \\
(ii) An example that illustrates this theorem is given by the symmetric peakon--antipeakon collision in case of the CH equation. Consider the case of $n=2$ in % \eqref{eq:multipeakon}
and let $p_1(0)=-p_2(0)$, and
$q_1(0)=-q_2(0)<0$. Then  the solution $u$ will vanish
pointwise at a collision time $t^*$ when
$q_1(t^*)=q_2(t^*)$, that is, $u(t^*,x)=0$ for all
$x\in \Real$. At time $t=t^*$, the total energy, $\mu(\Real)$, has concentrated at the origin. Using our mapping from Eulerian to Lagrangian coordinates, we obtain after collision, that is, for $t>t^*$, 
that $y_\xi(t,\xi)=1$ for $\xi\in\Real\setminus[0,\mu(\Real)]$ and $y_\xi(t,\xi)=0$ for $\xi\in [0,\mu(\Real)]$. Going back from Lagrangian to Eulerian coordinates we have $u(t,x)=0$, while the whole energy  still is concentrated at the origin. 
\end{remark}

\begin{proof}
Recall that $P(t,x)-u^2(t,x)-\frac12 k^2$ is defined by \eqref{rep:Peul} and since $\meas(\{x\in\Real\mid \mu(t,\{x\})\not=0\})=0$, we get 
\begin{align}\label{rep:Peul2}
 P(t,x)&-u^2(t,x)-\frac12 k^2\\ \nn
 &\quad =-2c\chi(x)\bar u(t,x)-\bar u^2(t,x)\\ \nn
& \qquad+\frac12\int_{\{x\in\Real\mid \mu(t,\{x\})=0\}} e^{-\vert x-z\vert}(2c\chi\bar u+\bar u^2+\frac12 u_x^2+\frac12 \bar \rho^2+k\bar\rho)(t,z)dz\\ \nn
& \qquad +\frac12\int_\Real e^{-\vert x-z\vert }2c^2(\chi^{\prime 2}+\chi\chi'')(z)dz.
\end{align}
Thus setting $x=y(t,\xi)$ and performing the change of variables $z=y(t,\eta)$ yields that $P(t,\xi)-U^2(t,\xi)-\frac12 k^2=P(t,y(t,\xi))-u^2(t,y(t,\xi))-\frac12 k^2$ coincides with \eqref{eq:Plag1}. Similar considerations yield that $Q(t,\xi)=P_x(t,y(t,\xi))$ coincides with \eqref{eq:Qlag1}.

 We will only prove \eqref{eq:weak1} and \eqref{eq:weak5} here since \eqref{eq:weak2} and \eqref{eq:weak3} can be shown in much the same way. Using the change of variables $x=y(t,\xi)$ and $U_\xi=u_x\circ y  y_\xi$ we get 
\begin{align}
 &\iint_{\Real_+\times \Real} (-u\phi_t+uu_x\phi)(t,x) dxdt\\ \nn 
&\qquad = \iint_{\Real_+\times\Real} (-Uy_\xi(t,\xi)\phi_t(t,y(t,\xi))+UU_\xi(t,\xi)\phi(t,y(t,\xi)))d\xi dt\\ \nn 
& \qquad=\iint_{\Real_+\times\Real}(-Uy_\xi(t,\xi)(\phi(t,y(t,\xi)))_t \\\nn
&\qquad\qquad\qquad+ U^2y_\xi(t,\xi)\phi_x(t,y(t,\xi))+UU_\xi(t,\xi)\phi(t,y(t,\xi)))d\xi dt\\ \nn 
& \qquad= \int_{\Real}Uy_\xi(0,\xi)\phi(0,y(0,\xi))d\xi-\iint_{\Real_+\times\Real}Qy_\xi(t,\xi)\phi(t,y(t,\xi))d\xi dt\\ \nn 
& \qquad=\int_\Real u(0,x)\phi(0,x)dx-\iint_{\Real_+\times\Real}P_x(t,x)\phi(t,x)dx.
\end{align}
Note that we do not have to restrict the domain in Lagrangian coordinates to $\Real_+\times \{\xi\in\Real\mid y_\xi(t,\xi)\not =0\}$, since the integrand  vanishes whenever $y_\xi(t,\xi)=0$, and therefore integrating in addition over the set $\Real_+\times \{\xi\in\Real\mid y_\xi(t,\xi)=0\}$ has no influence on the value of the integral. 

We now turn to the proof of \eqref{eq:weak5}, which is equivalent to showing 
\begin{equation}\label{eq:weak6}
 (u^2+\mu+k^2+2k\bar \rho)_t+(u(u^2+\mu+k^2+2k\bar\rho))_x=(u^3-2Pu)_x,
\end{equation}
in the sense of distributions. Using the change of variables $x=y(t,\xi)$ and $U_\xi=u_x\circ y y_\xi$ we get 
\begin{align}\label{eq:iint1}
 &\iint_{\Real_+\times\Real}u^2\phi_t(t,x)dx dt \\ \nn 
& \qquad = \iint_{\Real_+\times\Real} [(\phi(t,y(t,\xi)))_t-\phi_x(t,y(t,\xi))U(t,\xi)] U^2(t,\xi)y_\xi(t,\xi)d\xi dt\\ \nn 
& \qquad = -\iint_{\Real_+\times\Real} (U^2(t,\xi)y_{t,\xi}(t,\xi)+2U(t,\xi)U_t(t,\xi)y_\xi(t,\xi))\phi(t,y(t,\xi))d\xi dt\\ \nn 
& \qquad \quad-\iint_{\Real_+\times \Real} U^3(t,\xi)\phi_\xi(t,y(t,\xi))d\xi dt\\ \nn 
& \qquad =-\iint_{\Real_+\times\Real} (U^2(t,\xi)U_\xi(t,\xi)-2U(t,\xi)Q(t,\xi)y_\xi(t,\xi))\phi(t,y(t,\xi))d\xi dt\\ \nn 
& \qquad \quad-\iint_{\Real_+\times\Real} U^3(t,\xi)\phi_\xi(t,y(t,\xi))d\xi dt \\ \nn 
& \qquad = -\iint_{\Real_+\times\Real}u^2(t,x)u_x(t,x)\phi(t,x)dxdt -\iint_{\Real_+\times\Real}u^3(t,x)\phi_x(t,x)dx dt\\ \nn 
& \qquad \quad+\iint_{\Real_+\times\Real} 2U(t,\xi)P_\xi(t,\xi)\phi(t,y(t,\xi))d\xi dt.
\end{align}
Next 
\begin{align}\label{eq:iint2}
 & \iint_{\Real_+\times\Real} \phi_t(t,x)d\mu(t,x)dt= \iint_{\Real_+\times\Real}\phi_t(t,y(t,\xi))h(t,\xi)d\xi dt\\ \nn 
& \qquad =\iint_{\Real_+\times\Real} [(\phi(t,y(t,\xi)))_t-\phi_x(t,y(t,\xi))y_t(t,\xi)]h(t,\xi)d\xi dt\\ \nn 
& \qquad =-\iint_{\Real_+\times\Real}h_t(t,\xi)\phi(t,y(t,\xi))d\xi dt-\iint_{\Real_+\times\Real}U(t,\xi)h(t,\xi)\phi_x(t,y(t,\xi))d\xi dt\\ \nn 
& \qquad =-\iint_{\Real_+\times\Real} 2(U^2(t,\xi)+\frac12 k^2-P(t,\xi))U_\xi(t,\xi)\phi(t,y(t,\xi))d\xi dt\\ \nn 
& \qquad \quad-\iint_{\Real_+\times\Real}U(t,\xi)h(t,\xi)\phi_x(t,y(t,\xi))d\xi dt\\ \nn 
& \qquad = -\iint_{\Real_+\times\Real}2(u^2+\frac12 k^2-P)(t,x)u_x(t,x)\phi(t,x)dxdt\\ \nn 
& \qquad \quad -\iint_{\Real_+\times\Real} u(t,x)\phi_x(t,x)d\mu(t,x).
\end{align}
Since $k(t)=k(0)$ we get
\begin{equation}\label{eq:iint3}
 \iint_{\Real_+\times\Real} k^2\phi_t(t,x) dx dt=0.
\end{equation}
Finally
\begin{align}\label{eq:iint4}
& \iint_{\Real_+\times\Real} 2k\bar \rho(t,x)\phi_t(t,x) dx dt \\ \nn
& \qquad = \iint_{\Real_+\times\Real} 2k\bar r(t,\xi)\phi_t(t,y(t,\xi)) d\xi dt\\ \nn 
& \qquad = \iint_{\Real_+\times\Real} 2k\bar r(t,\xi)[(\phi(t,y(t,\xi)))_t-\phi_x(t,y(t,\xi))U(t,\xi)]d\xi dt\\ \nn 
& \qquad = -\iint_{\Real_+\times\Real} 2k\bar r_t(t,\xi)\phi(t,y(t,\xi))d\xi dt-\iint_{\Real_+\times\Real} 2k\bar r(t,\xi)U(t,\xi)\phi_x(t,y(t,\xi))d\xi dt\\ \nn 
& \qquad = \iint_{\Real_+\times\Real} 2k^2U_\xi(t,\xi)\phi(t,y(t,\xi))d\xi dt -\iint_{\Real_+\times\Real}2k\bar\rho(t,x)u(t,x)\phi_x(t,x)dx dt\\ \nn 
& \qquad =-\iint_{\Real_+\times\Real} 2k^2U(t,\xi)\phi_\xi(t,y(t,\xi))d\xi dt -\iint_{\Real_+\times\Real}2k\bar\rho(t,x)u(t,x)\phi_x(t,x)dx dt\\ \nn 
& \qquad = -\iint_{\Real_+\times\Real}2k^2u(t,x)\phi_x(t,x)dx dt -\iint_{\Real_+\times\Real}2k\bar\rho(t,x)u(t,x)\phi_x(t,x)dx dt. 
\end{align}
Adding up the equalities \eqref{eq:iint1}--\eqref{eq:iint4} proves \eqref{eq:weak6}.

As far as  \eqref{eq:weak4} is concerned, we observe that 
\begin{align*}
 \int_\Real (u_x^2+\bar\rho^2)(t,x)dx& = \int_{\{x\in\Real\mid \mu(t,\{x\})=0\}} (u_x^2+\bar\rho^2)(t,x)dx=\int_{\{\xi\in\Real\mid y_\xi(t,\xi)\not =0\}} h(t,\xi)d\xi\\ 
& \leq \int_\Real h(t,\xi)d\xi=\int_\Real d\mu(t,x)dx.
\end{align*}

\end{proof}

\section{Stability with respect to the initial data} \label{sec:stabil}

Already in \cite{GHR4} we showed  that $\rho$ has a regularizing effect and that wave breaking is closely linked to the initial value of $\rho$. In particular, we showed that discontinuities travel at finite speed although the 2CH system has an infinite speed of propagation, see \cite{henry:09}. Moreover, we obtained that wave breaking can only occur at points $x$ which satisfy $\rho_0(x)=0$. For completeness and to motivate assumptions that were made throughout this paper, we recall this result here. 
 
Given $(u,\rho,\mu)\in\D$, $p\in\mathbb{N}$ and an
open set $I$, we say that $(u,\rho,\mu)$ is
$p$-regular on an open set $I$ if
\begin{equation*}
  u\in W^{p,\infty}(I),\ \rho\in W^{p-1,\infty}(I)\ \text{ and } \muac=\mu\text{ on }I.
\end{equation*}
By notation, we set
$W^{0,\infty}(I)=L^\infty(I)$.
\begin{theorem}\cite[Theorem 6.1]{GHR4}
  \label{th:presreg}
  We consider the initial data
  $(u_0,\rho_0,\mu_0)$. Assume that
  $(u_0,\rho_0,\mu_0)$ is p-regular on a given
  interval $(x_0,x_1)$ and
  \begin{equation}
    \label{eq:rhopos}
    \rho_0(x)^2\geq c>0
  \end{equation}
  for $x\in (x_0,x_1)$. Then, for any
  $t\in\Real_+$, $(u,\rho,\mu)(t,\cdot)$ is
  p-regular on the interval $(y(t,\xi_0),
  y(t,\xi_1))$, where $\xi_0$ and $\xi_1$ satisfy
  $y(0,\xi_0)=x_0$ and $y(0,\xi_1)=x_1$ and are
  defined as
  \begin{equation*}
    \xi_0=\sup\{\xi\in\Real\ |\ y(0,\xi)\leq x_0\} \text{ and }
    \xi_1=\inf\{\xi\in\Real\ |\ y(0,\xi)\geq x_1\}.
  \end{equation*}
\end{theorem}

In the case of conservative solutions we have been able to prove in
\cite[Theorem 6.3]{GHR4} that any conservative solution $(u,\mu)$ of
the CH equation with initial data $(u_0,\mu_0)$ such that
$\mu_0=\mu_{0,\rm ac}$ can be approximated by smooth conservative
solutions $(u_n,\rho_n,\mu_n)$ of the 2CH system with $\rho_n(0,x)\not
=0$ for all $x\in\Real$. Observe that the approximate solutions of the
2CH system do not experience wave breaking.

In the context of dissipative solutions we cannot hope that we can approximate dissipative solutions
of the CH equation by solutions of the 2CH system which do not enjoy wave breaking according to the
definition of our metric in Lagrangian coordinates.  This is illustrated in Figure
\ref{fig:compsol}.  Note that, in \cite{GHR4}, we show the regularizing effect of strictly positive
$\rho$. This effect is also present in the example shown in the figure because, for strictly
positive $\rho_0$, dissipative and conservative solutions coincide and the second solution (dashed
line) is $C^\infty(\Real)$.  However, $\rho_0$ is small and we also have numerical errors, so that
the solution appears as if it contains peaks.  In that way, it looks very much like the conservative
solution of the scalar CH equation with initial data $u_0$. The numerical scheme used to compute
these solutions is an adaptation of the scheme studied in \cite{cohenray} and the code is
available at \cite{github}.

\begin{figure}[h]
  \centering
  \includegraphics[width=6cm]{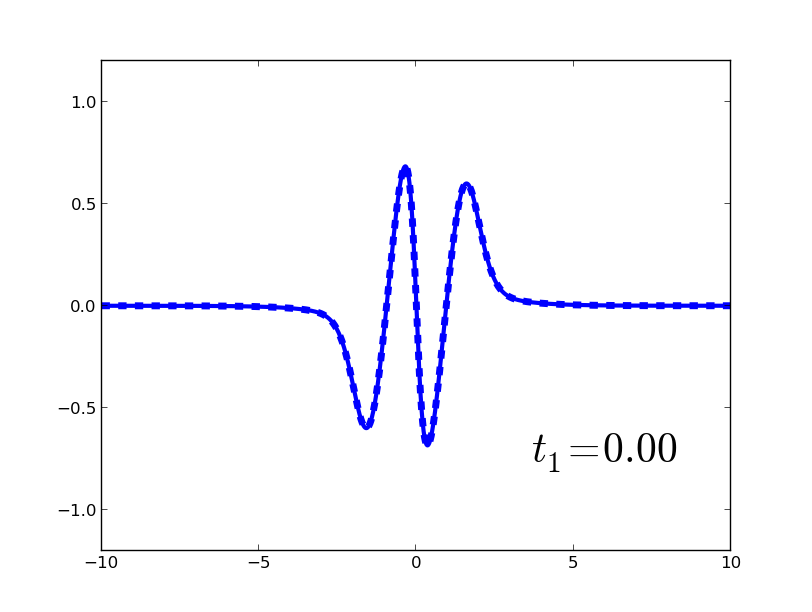}
  \includegraphics[width=6cm]{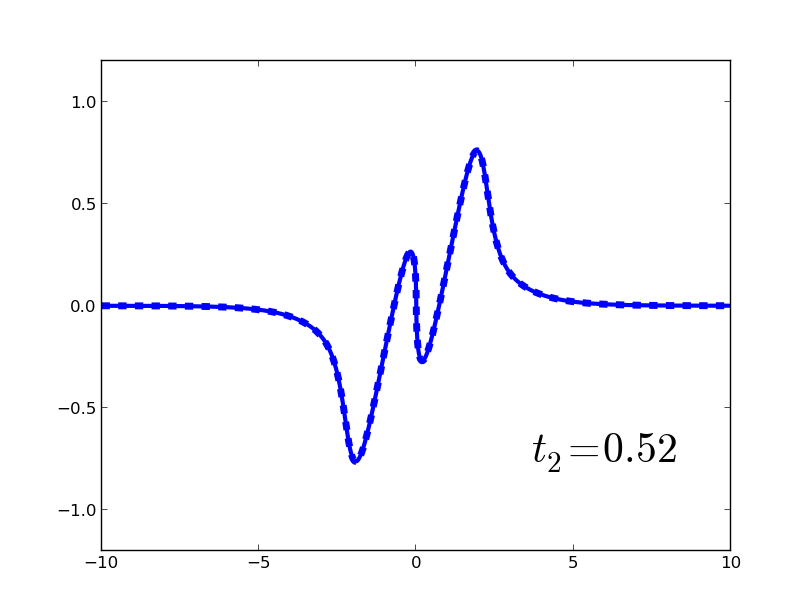}\\
  \includegraphics[width=6cm]{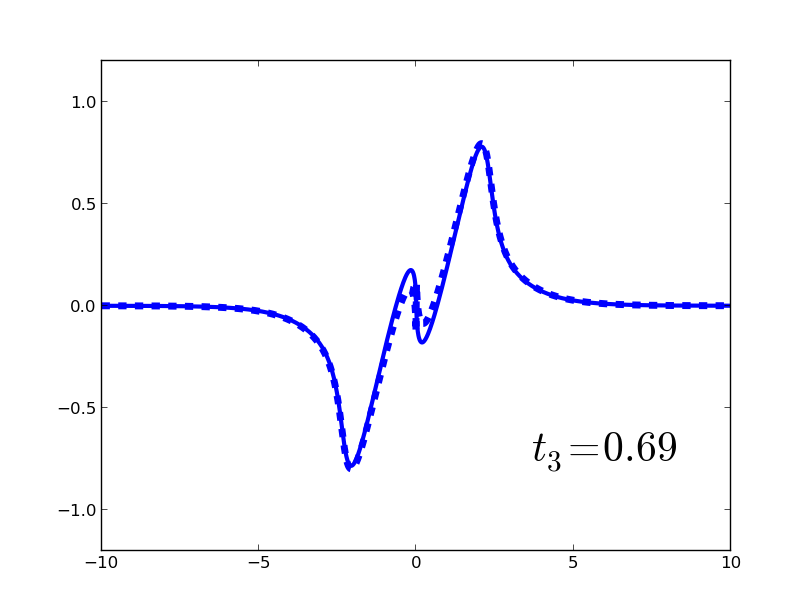}
  \includegraphics[width=6cm]{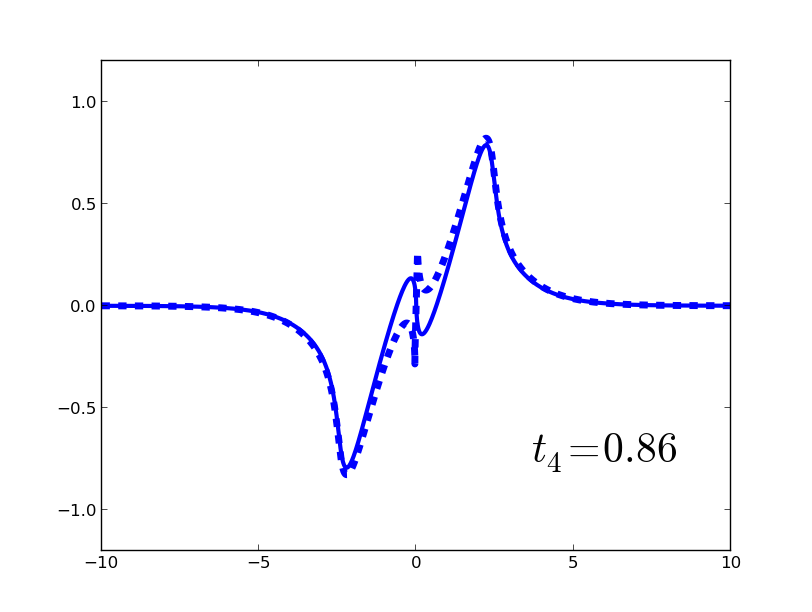}\\
  \includegraphics[width=6cm]{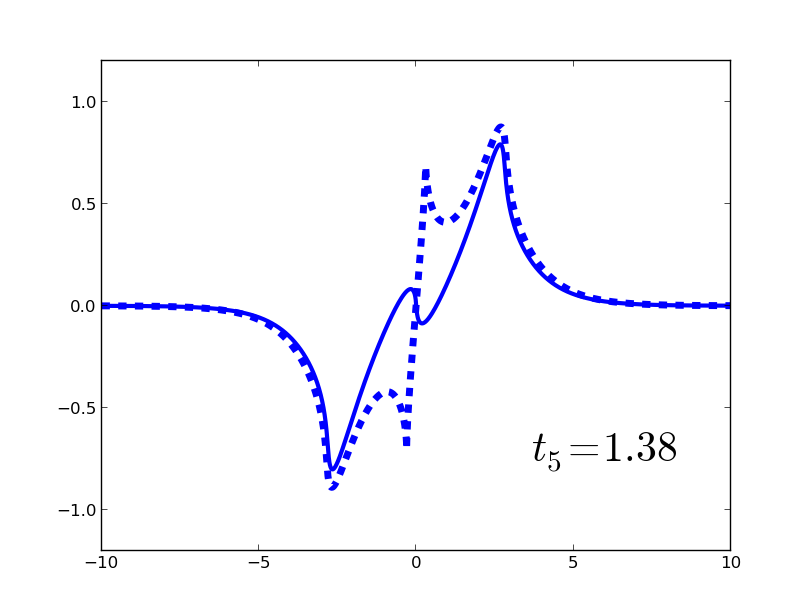}
  \includegraphics[width=6cm]{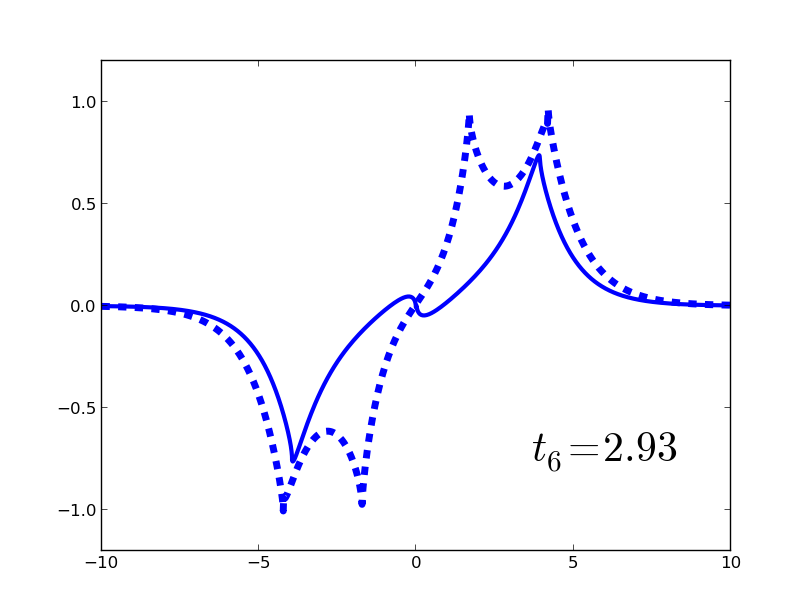}
  \caption{Computed dissipative solutions for initial data $u_0(x)=\alpha e^{-x^2}x(x - 1)(x +1)$
    and $\rho_0(x)=\epsi e^{-x^2/10}$ for $\epsi=0$ (solid line) and $\epsi=0.01$ (dashed line). The
    figures display the function $u$ only. When $\epsi=0$, the function $u$ solves the dissipative
    CH equation as $\rho$ will remain identically zero. In the case with $\epsi>0$, the fact that
    $\rho_0$ is strictly positive implies that no dissipation of energy takes place. In contrast, in
    the case $\epsi=0$, dissipation starts occurring between the times $t_1$ and $t_2$. Prior to
    that, $\rho_0$ is so small that we cannot see the difference between the two solutions. However,
    the solutions look very different for $t\geq t_4$, after the first solution has experienced
    dissipation.  This example show that the semigroup of dissipative solution is not continuous
    with respect to $\rho$ in a standard sense. This fact is reflected in the construction of our
    metric, which completely separates initial data for which $\rho_0$ vanishes and initial data for
    which $\rho_0$ is strictly positive in the same region.}
  \label{fig:compsol}
\end{figure}

\begin{lemma}\label{lem:urhog}

Given $(u,\rho)\in H_{0,\infty}(\Real)\times L^2_{\rm const}(\Real)$, let us denote by $X_{e}(x)$ the following vector 
\begin{equation}\label{defXe}
X_e(x)=(x,\bar u(x), c,1, u_x(x), u_x^2(x)+\bar \rho^2(x), \bar \rho(x), k),
\end{equation}
then for $g$ as in Definition~\ref{eq:defg}, we have $g(X_e(\dott))-1\in L^1(\Real)$ and the following holds:

(\textit{i}) Given two elements $(u_{1},\rho_1)$ and $(u_2,\rho_2)$ in $H_{0,\infty}(\Real)\times L^2_{\rm const}(\Real)$ such that 
\begin{equation}\label{measure}
\meas(\{x\in\Real\mid (\rho_1(x)=0 \text{ and }\rho_2(x)\not=0)\text{ or }(\rho_1(x)\not =0 \text{ and }\rho_2(x)=0)\})=0,
\end{equation}
then 
\begin{equation}\label{urhog}
\norm{g(u_1,\rho_1)-g(u_2,\rho_2)}_{L^1 }\leq C(\norm{u_{1,x}-u_{2,x}}_{L^2 }+\norm{\bar\rho_1-\bar\rho_2}_{L^2 }),
\end{equation}
where $C$ denotes a constant dependent on the $L^2(\Real)$-norm of $u_{1,x}$, $u_{2,x}$, $\bar\rho_1$, and $\bar\rho_2$.

(\textit{ii}) Given any sequence 
 $(u_n,\rho_n)$ in $H_{0,\infty}(\Real)\times L^2_{\rm const}(\Real)$ which converges to $(u,\rho)$ in $H_{0,\infty}(\Real)\times L^2_{\rm const}(\Real)$ such that for all $n\in\mathbb{N}$,
$$\meas(\{x\in\Real\mid (\rho_0(x)=0 \text{ and }\rho_n(x)\not=0)\text{ or }(\rho_0(x)\not =0 \text{ and }\rho_n(x)=0)\})=0,$$
 then $g(u_n,\rho_n)-1$ converges to $g(u,\rho)-1$ in $L^1(\Real)$.
\end{lemma}

\begin{proof}
 We will only prove (\textit{i}) since (\textit{ii}) is an immediate consequence of (\textit{i}). We will have to consider different cases and therefore we introduce the following sets:
\begin{align*}
 \Omega_1& =\{x\in\Real\mid \vert u_x(x)\vert +2\vert k\bar \rho(x)\vert +2\leq 1+u_x^2(x)+\bar \rho^2(x), u_x(x)\leq 0 \text{, and } \rho(x)=0\},\\ \nn
 \Omega_{2-}& = \{x\in\Real\mid 1+u_x^2(x)+\bar \rho^2(x)\leq \vert u_x(x)\vert +2\vert k\bar \rho(x)\vert +2, u_x(x)\leq 0 \text{, and }\rho(x)=0\},\\ \nn
  \Omega_{2+}& = \{x\in\Real\mid  1+u_x^2(x)+\bar \rho^2(x)\leq \vert u_x(x)\vert +2\vert k\bar \rho(x)\vert +2, 0< u_x(x) \text{, and }\rho(x)=0\},\\ \nn
  \Omega_{2*}& = \{x\in\Real\mid   \vert u_x(x)\vert +2\vert k\bar \rho(x)\vert +2\leq 1+u_x^2(x)+\bar \rho^2(x), 0< u_x(x) \text{, and }\rho(x)=0\},\\ \nn
  \Omega_3& = \{x\in\Real\mid \rho(x)\not =0\}. 
\end{align*}
Note that $\Omega_{2-}\cup\Omega_{2+}\cup\Omega_{2*}=\Omega_{2}$. In addition we will denote the sets which correspond to $(u_1,\rho_1)$ and $(u_2,\rho_2)$ by superscripts $1$ and $2$, respectively.

Observe that we have by the definition of $\Omega_1$ for $x\in\Omega_1^i$, $i=1,2$, since $0=\bar \rho_i+k_i$, 
\begin{equation*}
 \vert u_{i,x}(x)\vert +\bar\rho_i^2(x)+1\leq u_{i,x}^2(x), \quad i=1,2,
\end{equation*}
which implies that $\vert u_{i,x}(x)\vert\geq 1$ for $i=1,2$ and therefore 
\begin{equation}\label{est:omega1}
 \meas(\{x\in \Omega_1^1\})+\meas(\{x\in\Omega_1^2\})\leq \norm{u_{1,x}}_{L^2 }^2+\norm{u_{2,x}}_{L^2 }^2<\infty.
\end{equation}

\noindent(\textit{i}) If $x\in \Omega_1^{1,c}\cap \Omega_1^{2,c}=I_1$, then 
\begin{align*}
 \vert g(u_1,\rho_1)(x)& -g(u_2,\rho_2)(x)\vert\\ 
&  =\vert (u_{1,x}+u_{2,x})(u_{1,x}-u_{2,x})(x)+(\bar\rho_1+\bar\rho_2)(\bar\rho_1-\bar\rho_2)(x)\vert.
\end{align*}
Thus, using the Cauchy--Schwarz inequality, we obtain 
\begin{align}\label{est:estI1} 
 \int_{I_1}& \vert g(u_1,\rho_1)(x) -g(u_2,\rho_2)(x)\vert dx \\ \nn 
& \leq (\norm{u_{1,x}}_{L^2 }+\norm{u_{2,x}}_{L^2 })\norm{u_{1,x}-u_{2,x}}_{L^2 }+(\norm{\bar\rho_1}_{L^2 }+\norm{\bar\rho_2}_{L^2 })\norm{\bar\rho_1-\bar\rho_2}_{L^2 }. 
\end{align}

\noindent(\textit{ii}) If $x\in\Omega_1^1\cap \Omega_1^2=I_2$, then we have
\begin{equation*}
 \vert g(u_1,\rho_1)(x) -g(u_2,\rho_2)(x)\vert =\vert (u_{2,x}-u_{1,x})(x)+2(\bar\rho_1+\bar\rho_2)(\bar\rho_1-\bar\rho_2)(x)\vert,
\end{equation*}
and accordingly
\begin{align}
 \int_{I_2}&\vert g(u_1,\rho_1)(x)-g(u_2,\rho_2)(x)\vert dx\\ \nn
   &\leq \int_{I_2}\abs{(u_{1,x}-u_{2,x})(x)}dx
   +2(\norm{\bar\rho_1}_{L^2 }+\norm{\bar\rho_2}_{L^2 })
   \norm{\bar\rho_1-\bar\rho_2}_{L^2 } \\ \nn
   &\leq \meas(I_2)^{1/2}\norm{u_{1,x}-u_{2,x}}_{L^2 }
   +2(\norm{\bar\rho_1}_{L^2 }+\norm{\bar\rho_2}_{L^2 })
   \norm{\bar\rho_1-\bar\rho_2}_{L^2 } \\ \nn
    &\leq (\norm{u_{1,x}}_{L^2 }+\norm{u_{2,x}}_{L^2 })\norm{u_{1,x}-u_{2,x}}_{L^2 }+2(\norm{\bar\rho_1}_{L^2 }+\norm{\bar\rho_2}_{L^2 })\norm{\bar\rho_1-\bar\rho_2}_{L^2 },
\end{align}
using \eqref{est:omega1}.

\noindent (\textit{iii}) $x\in (\Omega_1^1\cap \Omega_{2-}^2)\cup(\Omega_{2-}^1\cap\Omega_1^2)=I_3$: Without loss of generality, we assume $x\in  \Omega_1^1\cap \Omega_{2-}^2$ (the other case follows similarly).
Then, 
\begin{align*}
 \vert g(u_1,\rho_1)(x)&-g(u_2,\rho_2)(x)\vert\\ 
 & = \vert -u_{1,x}(x)+2\bar\rho_1^2(x)+2-1-u_{2,x}^2(x)-\bar\rho_2^2(x)\vert \\ 
& \leq \vert (u_{1,x}^2-u_{2,x}^2)(x)+(\bar\rho_1^2-\bar\rho_2^2)(x)\vert +\abs{-u_{1,x}^2(x)-u_{1,x}(x)+\bar \rho_1^2(x)+1}\\
& = \vert (u_{1,x}^2-u_{2,x}^2)(x)+(\bar\rho_1^2-\bar\rho_2^2)(x)\vert +u_{1,x}^2(x)+u_{1,x}(x)-\bar \rho_1^2(x)-1\\
& \leq \vert (u_{1,x}^2-u_{2,x}^2)(x)+(\bar\rho_1^2-\bar\rho_2^2)(x)\vert \\
& \quad +u_{1,x}^2(x)+u_{1,x}(x)-\bar \rho_1^2(x)-1+1+\bar\rho_2^2(x)-u_{2,x}(x)-u_{2,x}^2(x)\\
& \leq 2\vert (u_{1,x}^2-u_{2,x}^2)(x)\vert +2\vert (\bar \rho_1^2-\bar\rho_2^2)(x)\vert+\vert u_{1,x}(x)-u_{2,x}(x)\vert,
\end{align*}
and hence, using \eqref{est:omega1}, we obtain 
\begin{align}
 \int_{I_3} & \vert g(u_1,\rho_1)(x)-g(u_2,\rho_2)(x)\vert dx\\ \nn
& \leq 3(\norm{u_{1,x}}_{L^2 }+\norm{u_{2,x}}_{L^2 })\norm{u_{1,x}-u_{2,x}}_{L^2 }+2(\norm{\bar\rho_1}_{L^2 }+\norm{\bar\rho_2}_{L^2 })\norm{\bar\rho_1-\bar\rho_2}_{L^2 }.
\end{align}

\noindent (\textit{iv}) $x\in (\Omega_1^1\cap \Omega_{2+}^2)\cup(\Omega_{2+}^1\cap\Omega_1^2)=I_4$: Without loss of generality, we assume $x\in  \Omega_1^1\cap \Omega_{2+}^2$ (the other case follows similarly). By assumption we have $-u_{1,x}(x)+2\bar\rho_1^2(x)+2\leq 1+u_{1,x}^2(x)+\bar\rho_1^2(x)$ and $1+u_{2,x}^2(x)+\bar\rho_2^2(x)\leq u_{2,x}(x)+2\bar\rho_2^2(x)+2$, which implies that either
\begin{align*}
 0&\leq g(u_1,\rho_1)(x) -g(u_2,\rho_2)(x)\\
 &=-u_{1,x}(x)+2\bar\rho_1^2(x)+2-1-u_{2,x}^2(x)-\bar\rho_2^2(x)\\ 
& \leq (u_{1,x}^2-u_{2,x}^2)(x)+(\bar\rho_1^2-\bar\rho_2^2)(x)\\ 
& \leq (u_{1,x}+u_{2,x})(u_{1,x}-u_{2,x})(x)+(\bar\rho_1+\bar\rho_2)(\bar\rho_1-\bar\rho_2)(x), \\
\intertext{or} 
 0&\leq g(u_2,\rho_2)(x) -g(u_1,\rho_1)(x)\\
 &=1+u_{2,x}^2(x)+\bar\rho_2^2(x)+u_{1,x}(x)-2\bar\rho_1^2(x)-2\\ 
& \leq u_{2,x}(x)+2\bar\rho_2^2(x)+u_{1,x}(x)-2\bar \rho_1^2(x)\\ 
& \leq (u_{2,x}-u_{1,x})(x)+2(\bar\rho_1+\bar\rho_2)(\bar\rho_2-\bar\rho_1)(x),
\end{align*}
where we used in the last step that $0\leq -u_{1,x}(x)$ since $x\in\Omega_1^1$.
Thus applying \eqref{est:omega1} yields
\begin{align}
 \int_{I_4}& \vert g(u_1,\rho_1)(x)-g(u_2,\rho_2)(x)\vert dx\\ \nn 
& \leq  (\norm{u_{1,x}}_{L^2 }+\norm{u_{2,x}}_{L^2 })\norm{u_{1,x}-u_{2,x}}_{L^2 }+2(\norm{\bar\rho_1}_{L^2 }+\norm{\bar\rho_2}_{L^2 })\norm{\bar\rho_1-\bar\rho_2}_{L^2 }.
\end{align}

\noindent(\textit{v}) $x\in (\Omega_1^1\cap \Omega_{2*}^2)\cup(\Omega_{2*}^1\cap\Omega_1^2)=I_5$: Without loss of generality, we assume $x\in  \Omega_1^1\cap \Omega_{2*}^2$ (the other case follows similarly). By assumption we have 
$-u_{1,x}(x)+2\bar\rho_1^2(x)+2\leq 1+u_{1,x}^2(x)+\bar\rho_1^2(x)$ and $u_{2,x}(x)+2\bar\rho_2^2(x)+2\leq 1+u_{2,x}^2(x)+\bar\rho_2^2(x)$, which implies that either 
\begin{align*}
 0&\leq g(u_1,\rho_1)(x) -g(u_2,\rho_2)(x)\\
 &\leq -u_{1,x}(x)+2\bar\rho_1^2(x)+2-1-u_{2,x}^2(x)-\bar\rho_2^2(x)\\
& \leq (u_{1,x}^2-u_{2,x}^2)(x)+(\bar\rho_1^2-\bar\rho_2^2)(x)\\ 
& =(u_{1,x}+u_{2,x})(u_{1,x}-u_{2,x})(x)+(\bar\rho_1+\bar\rho_2)(\bar\rho_1-\bar\rho_2)(x), \\
\intertext{or} 
 0&\leq g(u_2,\rho_2)(x)-g(u_1,\rho_1)(x)\\ 
 &=1+u_{2,x}^2(x)+\bar\rho_2^2(x)+u_{1,x}(x)-2\bar\rho_1^2(x)-2\\
& \leq u_{2,x}^2(x)+\bar\rho_2^2(x)-\bar\rho_1^2(x)\\ \nn 
& \leq u_{2,x}(u_{2,x}-u_{1,x})(x)+(\bar\rho_1+\bar\rho_2)(\bar\rho_2-\bar\rho_1)(x),
\end{align*}
where we used in the last step that $0\leq -u_{1,x}u_{2,x}(x)$. Thus we get
\begin{align}\label{est:estI5}
 \int_{I_5} &  \vert(g(u_1,\rho_1)-g(u_2,\rho_2))(x)\vert dx\\ \nn
& \leq (\norm{u_{1,x}}_{L^2 }+\norm{u_{2,x}}_{L^2 })\norm{u_{1,x}-u_{2,x}}_{L^2 }+(\norm{\bar\rho_1}_{L^2 }+\norm{\bar\rho_2}_{L^2 })\norm{\bar\rho_1-\bar\rho_2}_{L^2 }.
\end{align}
Finally adding \eqref{est:estI1}--\eqref{est:estI5}, we end up with \eqref{urhog}. 
\end{proof}

\begin{lemma} \label{lemma:kont}
Given a sequence
 $(u_n,\rho_n)$ in $H_{0,\infty}(\Real)\times L^2_{\rm const}(\Real)$ which converges to $(u,\rho)$ in $H_{0,\infty}(\Real)\times L^2_{\rm const}(\Real)$ such that for all $n\in\mathbb{N}$,
$$\meas(\{x\in\Real\mid (\rho_0(x)=0 \text{ and }\rho_n(x)\not=0)\text{ or }(\rho_0(x)\not =0 \text{ and }\rho_n(x)=0)\})=0,$$
 then $(u_n,\rho_n, (u_{n,x}^2+\bar\rho_n^2)dx)$ converges to $(u,\rho, (u_x^2+\bar \rho^2)dx)$ in $\D$.
\end{lemma}

\begin{proof} First of all note that according to Lemma~\ref{lem:urhog} the function $g(X_{e,n})-1$ converges to $g(X_e)-1$ in $L^1(\Real)$. Since the set of Eulerian and Lagrangian coordinates and the mappings between them coincide with the ones used in \cite{GHR4}, one can prove everything  as in \cite[Lemma 6.4]{GHR4}, except  $g(X_n)\to g(X) \in L^2(\Real)$. Thus we will only show that $g(X_n)\to g(X)$ in $L^2(\Real)$. 

Let $X_n=(y_n,U_n,h_n,r_n)$ and $X=(y,U,h,r)$ be the representatives in $\F_0$ given by \eqref{eq:Ldef} of $L(u_n,\rho_n,(u_{n,x}^2+\bar\rho_n^2)dx)$ and $L(u,\rho,(u_x^2+\bar\rho^2)dx)$, respectively and assume that $X_n\to X$ in $V$. Abbreviate by $b_n=u_{n,x}^2+\bar \rho_n^2$ and $b=u_x^2+\bar\rho^2$ and note that $b_n\to b$ in $L^1(\Real)$. Since the measures $(u_{n,x}^2+\bar\rho_n^2)dx$ and $(u_x^2+\bar\rho^2)dx$ are purely absolutely continuous, we obtain that $y_\xi(\xi)>0$ almost everywhere and in particular 
\begin{equation}
y_\xi=\frac{1}{b\circ y+1} \quad \text{ and }\quad y_{n,\xi}=\frac{1}{b_n\circ y_n+1}.
\end{equation}
This implies in particular that $g_n\circ y_ny_{n,\xi}:=g(X_{e,n})\circ y_ny_{n,\xi}=g(X_n)$ almost everywhere and $g\circ yy_\xi:=g(X_e)\circ yy_\xi=g(X)$ almost everywhere so that 
\begin{align}\label{eq:stability3}
 g(X_n)-g(X)& =g_n\circ y_ny_{n,\xi}-g\circ yy_\xi\\ \nn 
& = (g_n\circ y_n(b\circ y+1)-g\circ y(b_n\circ y_n+1))y_\xi y_{n,\xi}\\ \nn 
& =(g_n\circ y_n-g\circ y)y_\xi y_{n,\xi}\\ 
& \quad +\big(g_n\circ y_n (b\circ y-b_n\circ y_n)+(g_n\circ y_n-g\circ y)b_n\circ y_n\big)y_\xi y_{n,\xi}. \nn
\end{align}
We will study the first term on the right-hand side in detail and explain afterwards how the other terms can be treated similarly.
We have 
\begin{equation}\label{eq:stability1}
 (g_n\circ y_n-g\circ y)y_\xi y_{n,\xi}= (g_n-g)\circ y_n y_\xi y_{n,\xi}+(g\circ y_n-g\circ y)y_\xi y_{n,\xi}.
\end{equation}
Using now the change of variables $x=y_n(\xi)$, since $y_\xi(\xi)\leq 1$, we get
\begin{equation}\label{eq:stability2}
 \norm{(g_n-g)\circ y_ny_\xi y_{n,\xi}}_{L^1}\leq \norm{(g_n-g)\circ y_ny_{n,\xi}}_{L^1}\leq \norm{g_n-g}_{L^1}.
\end{equation}
Since $g\in L^1(\Real)$, we can find to any $\varepsilon>0$ a continuous function $l$ with compact support such that $\norm{g-l}_{L^1}\leq\varepsilon/3$. Hence we can decompose the second term on the right-hand side of \eqref{eq:stability1} into 
\begin{align}
 (g\circ y_n-g\circ y)y_\xi y_{n,\xi}& = (g\circ y_n-l\circ y_n)y_\xi y_{n,\xi}\\ \nn 
&\quad + (l\circ y_n-l\circ y)y_\xi y_{n,\xi}+(l\circ y-g\circ y) y_\xi y_{n,\xi}.
\end{align}
By arguing as in \eqref{eq:stability2}, one can show that 
\begin{equation}
 \norm{(g\circ y_n-l\circ y_n)y_\xi y_{n,\xi}}_{L^1}+\norm{(g\circ y-l\circ y)y_\xi y_{n,\xi}}_{L^1}\leq \frac{2}{3}\varepsilon.
\end{equation}
Moreover, since $y_n\to y$ in $L^\infty(\Real)$ and $l$ is continuous with compact support, we obtain by applying the Lebesgue dominated convergence theorem, that  $l\circ y_n\to l\circ y$ in $L^1(\Real)$ and thus we can choose $n$ big enough so that
\begin{equation}
 \norm{(l\circ y_n-l\circ y)y_\xi y_{n,\xi}}\leq \norm{l\circ y_n-l\circ y}_{L^1}\leq \frac{\varepsilon}{3}. 
\end{equation}
Hence, we get that $\norm{(g\circ y_n-g\circ y)y_\xi y_{n,\xi}}\leq \varepsilon$ so that 
\begin{equation}
 \lim_{n\to\infty}\norm{(g\circ y_n-g\circ y)y_\xi y_{n,\xi}}=0.
\end{equation}
As far as the second term on the right-hand side of \eqref{eq:stability3} is concerned, we observe that $g_n\circ y_ny_{n,\xi}\leq 2(y_\xi+h)=2$ and $b_n\circ y_n\leq h\leq 1$, so that we can follow the same procedures as for the first term. Thus we finally obtain $g(X_n)\to g(X)$ in $L^1(\Real)$ and since $g(X_n)\leq 2$ and $g(X)\leq 2$ this implies $g(X_n)\to g(X)$ in $L^2(\Real)$.
\end{proof}

\begin{lemma} \label{lemma:kont1}
 Let $(u_n,\rho_n, \mu_n)$ be a sequence in $\D$ that converges to $(u,\rho,\mu)$ in $\D$. Then 
\begin{align*}
  u_n&\rightarrow u \text{ in } L^\infty(\Real),& \bar \rho_n& \overset{\ast}{\rightharpoonup}\bar \rho,& k_n &\to k\in \Real,\\
   \mu_n &\overset{\ast}{\rightharpoonup}\mu, &  g(X_{e,n})&\overset{\ast}{\rightharpoonup} g(X_e),&&
\end{align*}
where $X_{e,n}$ and $X_e$ are defined by \eqref{defXe}, for $(u_n,\rho_n,\mu_n)$ and $(u,\rho,\mu)$, respectively.
\end{lemma}

\begin{proof}
 Since the set of Eulerian and Lagrangian coordinates and the mappings between them coincide with the ones used in \cite{GHR4}, one can prove everything as in \cite[Lemma 6.4]{GHR4} except for $g(X_{e,n})\overset{\ast}{\rightharpoonup}  g(X_e)$, which we will prove now. 
 
We denote by $X_n=(y_n,U_n,h_n,r_n)$ and $X=(y,U,h,r)$ the representative of $L(u_n,\rho_n,\mu_n)$ and $L(u,\rho,\mu)$ given by \eqref{eq:Ldef}. By weak-star convergence we mean that 
\begin{equation}\label{est:estweak}
 \lim_{n\to\infty} \int_\Real g(X_{e,n})\phi dx=\int_\Real g(X_e)\phi dx
\end{equation}
 for all continuous functions with compact support. It follows from \eqref{eq:defg}, that $y_\xi(\xi)=0$ implies that $g(X(\xi))=0$ and therefore 
\begin{equation}
\begin{aligned}
 \int_\Real g(X_e)\phi dx &=\int_{\Real} g(X_e)\phi dx \\
 &=\int_{\{\xi\in\Real\mid y_\xi\not=0\}} g(X)\phi\circ yd\xi=\int_\Real g(X)\phi\circ yd\xi.
\end{aligned}
\end{equation}
Similarly one obtains that $\int_\Real
g(X_{e,n})\phi dx=\int_\Real g(X_n)\phi\circ
y_n d\xi$. Since $y_n\to y$ in $L^\infty(\Real)$
and $y-\id\in L^\infty(\Real)$, the support of
$\phi\circ y_n$ is contained in some compact set
which can be chosen independently of $n$ and, from
Lebesgue's dominated convergence theorem, we have
$\phi\circ y_n\to \phi\circ y$ in
$L^2(\Real)$. Hence, since $g(X_n)\to g(X)$ in
$L^2(\Real)$,
\begin{equation}
 \lim_{n\to\infty} \int_\Real g(X_n) \phi\circ y_nd\xi= \int_\Real g(X)\phi\circ yd\xi. 
\end{equation}
Combining all the equalities we obtained so far yields \eqref{est:estweak}.
\end{proof}

\end{document}